\documentclass[11pt]{article}
\usepackage{amsfonts}
\usepackage[T1]{fontenc}
\usepackage{amssymb,amsmath}
\usepackage{mathrsfs}
\usepackage{amsthm}
\usepackage{appendix}
\usepackage{fancyhdr}
\usepackage{times}
\usepackage{color}
\usepackage[usenames,dvipsnames,svgnames,table]{xcolor}
\usepackage{textcomp}
\usepackage{hyperref}
\usepackage{graphicx}
\usepackage{caption}
\usepackage{subcaption}
\usepackage{float}
\usepackage[normalem]{ulem}
\usepackage[displaymath, mathlines]{lineno}
\setpagewiselinenumbers

\theoremstyle{plain}
\numberwithin{equation}{section}
\newtheorem{theorem}{Theorem}[section]
\theoremstyle{definition}
\newtheorem{definition}{Definition}[section]

\theoremstyle{plain}
\newtheorem{corollary}{Corollary}[section]
\newtheorem{lemma}{Lemma}[section]
\newtheorem{proposition}{Proposition}[section]

\theoremstyle{definition}
\newtheorem{remark}{Remark}[section]

\newcommand{\D}{\mathbb{D}} 
\newcommand{\E}{\mathbb{E}}        %esperanza
\newcommand{\R}{\mathbb{R}}          %conjunto numeros reales
\newcommand{\NN}{\mathbb{N}}         %conjunto numeros naturales
\newcommand{\p}{\mathbb{P}}
\renewcommand{\P}{\mathbb{P}}          %probabilidad
\newcommand{\F}{\mathcal{F}}            %sigma algebra
\newcommand{\I}{\mathbf{1}}              %indicatriz
\newcommand{\T}{\mathcal{T}}             %arboles
\newcommand{\N}{\mathcal{N}}             %excursion measure
\newcommand{\gG}{\mathcal{G}}

\newcommand{\M}{\mathscr{M}}              %finite measures

\newcommand{\cF}{\mathcal{F}}  

\newcommand{\exc}{\mathbf{e}}

\newcommand{\diffd}{\mathrm{d}}   %\mathrm{d}
\newfont{\indic}{bbmss12}
\def\un#1{\hbox{{\indic 1}$_{#1}$}}

\newcommand{\cN}{{\ensuremath{\mathcal N}} }
\newcommand{\intens}{\vartheta}

\begin{document}
 \title{ Ray-Knight representation of flows of branching processes with competition  by pruning of L\'evy trees}
 
 \author{J.Berestycki \footnote{Department of Statistics, University of Oxford, 1 South Parks Road, Oxford, OX13TG United Kingdom.  \ E-mail: \,  \texttt{julien.berestycki@stats.ox.ac.uk}},\,   
M.C. Fittipaldi \footnote{Universidad Nacional Aut\'onoma de M\'exico, {Departamento de Matem\'aticas, Facultad de Ciencias, Av. Universidad 3000, Circuito Exterior, Ciudad Universitaria,Delegación Coyoacán, C.P. 04510, D.F., México.} E-mail: \, \texttt{clarafitti@matem.unam.mx }} \,  and  \, 
 J.Fontbona\footnote{Department of Mathematical Engineering and Center for Mathematical Modeling,  UMI(2807) UCHILE-CNRS, University of  Chile,  { Casilla 170-3, Correo 3, Santiago-Chile.}  E-mail:\,  \texttt{fontbona@dim-uchile.cl}.  } }
 \date{\today}

 \maketitle

 \begin{abstract} 
  We introduce flows of branching processes with competition,  which  describe the evolution of general continuous state branching populations in which interactions between individuals give rise to a negative density dependence term. This generalizes  the logistic branching processes studied by Lambert  \cite{L1}.  Following the approach developed by Dawson and Li \cite{DL},  we first construct such processes  as the solutions of certain flow of stochastic differential equations.  We then  propose a  novel genealogical description  for   branching processes with competition based on  interactive pruning of L\'evy-trees, and establish a  Ray-Knight representation result for these processes in terms of the local times of suitably  pruned  forests.
 \end{abstract}
 
\bigskip


\medskip


\tableofcontents

\section{Introduction and main results}
 \subsection{Continuous state branching processes and their genealogies}\label{introCSBP}

Continuous state branching processes model the evolution of the size of a continuum population in which individuals reproduce and die but do not interact. 
Mathematically, a continuous state branching process (or  CSBP, for short) is a c\`adl\`ag $[0,\infty)$-valued strong Markov processes $Y= (Y_t: t\geq 0)$ with law $ \p_y $ given the initial state $y \geq 0$,  satisfying the branching property:  for any $y_{1},y_{2} \in [0,\infty)$, $Y$  under $\p_{y_{1}+y_{2}}$ has the same law as the independent sum $Y^{(1)} + Y^{(2)}$, with   $Y^{(i)}$  distributed as $Y$ under $\p_{y_{i}}$ ($i=1,2$). 

CSBPs arise as the possible scaling limits of discrete Galton-Watson processes, and their laws are completely characterized by their so-called  {\it branching mechanism}, which is the Laplace exponent $\lambda \mapsto \psi(\lambda)$ of a spectrally positive L\'evy process $X$. 

More precisely, one has 
\[
\E_{x}(e^{- \theta Y_{t}}) = e^{-xu_{t}(\theta)},  \, x\geq 0, 
\]
where $u_t$ is the unique nonnegative solution of the differential equation     
\[
\dfrac{\partial u_t(\theta)}{\partial t}= -\psi (u_t(\theta)), \quad u_0(\theta)=\theta.
\]  
By a celebrated result of Lamperti, $Y$ can also be  obtained in a pathwise way from  $X$  by  means a  random time-change  \cite{L}. We term the process $Y$  a $\psi$-CSBP to specify the underlying branching mechanism.

In this work, we focus on the case where the function  $\psi(\lambda)=\ln \E(e^{-\lambda X_1})$ has the following properties \eqref{psi2} and \eqref{psi3}, which    {\bf we assume  throughout}:  

One has
\begin{equation}\label{psi2}
 \psi(\lambda)= \alpha\lambda+\frac{1}{2}\sigma^{2}\lambda^{2} +\int_{(0,\infty)}(e^{-\lambda x} -1
                   +\lambda x)\Pi(\diffd x) \,,   \quad \lambda\geq 0 \, ,
\end{equation}
for some $\alpha, \sigma\geq 0$ and $\Pi$ a measure on $(0,\infty)$ such that $\int_{(0,\infty)} (x\wedge x^{2})\Pi(\diffd x)\,<\,\infty$. 
Moreover, Grey's condition holds, that is, one has 
\begin{equation}\label{psi3}
 \int_1^{\infty} \dfrac{\diffd \lambda}{\psi(\lambda)}<\infty .
\end{equation} 
Condition  \eqref{psi2}
 ensures that  $Y$ is conservative (i.e.  $\forall$ $t,x>0$, $\p_{x}(Y_{t}<\infty)= 1$ ) and (sub)critical  ($\psi'(0+)\geq 0$). Condition  \eqref{psi3}  implies  that  there is a.s. extinction (i.e. $\p_{x}(\exists t\geq 0: Y_{t}=0)= 1$, $\forall$ $x>0$) and that $\sigma > 0$ or $\int_{0}^{1} r\Pi(\diffd r)=\infty$, so the paths of $Y$ have infinite variation a.s. We refer the reader to \cite{DLGI}, \cite{Ky},  \cite{BookLi} for these facts and for general background on CSBP.

The branching property  allows us to construct a family of $\psi$-CSBP as a  two parameters process $(Y_t(v), t\ge 0, v\ge 0)$, with  $v$ ranging over all possible initial population sizes. 
In a more general  setting, this was first done  by Bertoin and Le Gall  \cite{BLGI} using families of nested  subordinators.  More recently, Dawson and Li constructed such a process by means of a  {\it stochastic flow} of SDE driven by Gaussian space-time white noise and a Poisson random measure. Precisely, in \cite{DL}  the process $ (Y_t(v):t\geq 0, v\geq 0)$ was obtained as  the unique strong solution of the family of stochastic differential equations:
\begin{equation}\label{eqflow csbp}
\begin{split}      
Y_t(v)=&\ v - \alpha\int_0^t Y_s(v)\diffd s + \sigma \int_{0}^{t}\int_{0}^{Y_{s-}(v)} W(\diffd s,\diffd u )     + \int_{0}^{t}\int_{0}^{Y_{s-}(v)}\int_{0}^{\infty}r \tilde{N}(\diffd s,\diffd \nu,\diffd r), \\
       &   t\ge 0, v\geq 0,               \quad 
\end{split}
\end{equation}
where $W(\diffd s, \diffd u)$ is a Gaussian white noise process on $[0, \infty)^{2}$ based on the Lebesgue measure $\diffd s\otimes \diffd u$ and $\tilde N$  is the compensated version of a Poisson random measure $N$  on $[0, \infty)^{3}$ with intensity $\diffd s\otimes \diffd \nu \otimes \Pi(\diffd r)$.  
Stochastic calculus easily yields the fact that, for each $v\geq 0$,  the above process is a $\psi$-CSBP, started from the initial population size $v$.  
Moreover, it is shown in \cite{DL}  that $(Y_t(v) : t \geq 0, v\geq 0)$ has a version with the following properties: 
\begin{itemize}
   \item[i)] for each $v \geq 0$, $t \mapsto Y_t(v)$ is a c\`adl\`ag process on $[0, \infty)$;
   \item[ii)] for each $t \geq 0$, $v \mapsto Y_t(v)$ is a non-negative and non-decreasing c\`adl\`ag process on $[0, \infty)$;
   \item[iii)]  for any $0\leq v_1\leq v_2\leq \dots \leq v_n$, the processes  $(Y_t(v_{j})- Y_t(v_{j-1}): t\geq 0)$, $j=1,\dots, n$ are independent CSBP with branching mechanism $\psi$ issued from  $v_j-v_{j-1}$ (branching property); and 
   \item[iv)]   if for each $0\leq s$  we denote by $(Y_{s,t}(v):t\geq s)$ the solution of \eqref{eqflow csbp} starting at time $s$, that is,
 \begin{equation}\label{flow property CSBP}
 Y_{s,t} (v) =\  v - \alpha\int_s^t Y_{s,r}(v)\diffd r + \sigma \int_{s}^{t}\int_{0}^{Y_{s,r-}(v)} W(\diffd r,\diffd u)  + \int_{s}^{t}\int_{0}^{Y_{s,r-}(v)}\int_{0}^{\infty}u \tilde{N}(\diffd r,\diffd \nu,\diffd u),
 \end{equation}
 then  for every $0\leq s\leq r$ we have a.s. $(Y_{s,t}(v): t\geq r)=(Y_{r,t}(Y_{s,r}(v)): t\geq r)$  (flow property).
\end{itemize}
The family of  solutions to equation \eqref{eqflow csbp} is thus called a {\it flow of continuous state branching processes} or, as in \cite{BookLi},  a {\it measure-valued branching process}. 
One upshot of such a construction is that, in a similar way as in  \cite{BLGI},  one can make sense of the genealogy of the family of CSBPs \eqref{eqflow csbp}: an individual $y$ at time $t\geq 0$ is a descendant of the individual $x$ at time  $s\leq t$ if and only if $y\in (Y_{s,t}(x-), Y_{s,t}(x)]$. 
 
 \smallskip

There is, however, a more explicit and  natural way to encode the genealogy of CSBPs, which  relies on   tree-like topological random objects known as  {\it continuum random trees} (CRT).  This point of view builds upon the pioneer works of Ray \cite{R} and Knight \cite{K} on the quadratic branching case, which imply that the full flow of the Feller diffusion can be  constructed from the local times, at different  heights, of  reflected Brownian motion. Aldous \cite{A,A1}  later showed how each of the   corresponding  Brownian excursions codes a  CRT. These results, together with   It\^o's  Poissonian  representation   of the Brownian  excursion process  \cite{I1} provide in that case a complete description of the continuum  genealogy of the evolving population. Later,  Le Gall and Le Jan \cite{LGLJ} and Duquesne and Le Gall \cite{DLG} extended such a genealogical  representation to a general class of $\psi$-CSBP.   More precisely,  under  assumptions \eqref{psi2} and \eqref{psi3},  those works provided  a  {\it Ray-Knight representation} result for a $\psi$-CSBP $Y$   in terms of  the height and local times processes of certain  measure valued Markov process,  the {\it exploration process},   defined using the corresponding L\'evy process  $X$  reflected at its running infimum (we recall this construction and the precise statement  in Section \ref{explo}).

\smallskip 

During the last decade,  the scope of mathematically tractable  population models  featuring branching behavior has been considerably enlarged to include models with interactions, immigration and density dependence. Negative density dependence, in particular, can represent  competition among individuals, due to limited resources or other mechanisms. The aim of the present work is to extend the Ray-Knight genealogical representation to branching-type populations with competition between individuals.

\subsection{Flows of branching processes with competition}

%%%%%%%%%%%%%%%%%%%%%%%%%%%%%%%%%%%%%%%

The prototypical example of a continuum  branching model with competition is the logistic branching process (LPB), introduced by Lambert  in \cite{L1} by  means of a Lamperti transform of spectrally positive Ornstein-Ulhenbeck processes. Alternatively, LBPs are defined as scaling limits of some discrete population models, where the death rate of each individual owed to competition, is proportional to the instantaneous population size. 
The aggregate competition rate results in that case into a negative drift, proportional to the squared population size (see \cite{L1} for details and further discussion). 

Generalizing that idea, we consider $g : [0,\infty) \to [0,\infty)$ a locally bounded measurable function, which we  call the {\it competition mechanism}. 
Heuristically, $g(z)$ is the rate at which an additional individual in a given population of size $z$ would be killed due to competition. We  introduce a competitive analogue of \eqref{eqflow csbp}:

\begin{proposition}  \label{flowLBP}
Let  $W(\diffd s, \diffd u)$  and $N(\diffd s,\diffd \nu,\diffd r)$ be the same processes considered in \eqref{eqflow csbp} and
define   a locally Lipschitz function $G: [0,\infty) \to [0,\infty)$ by $G(z):=\int_0^z g(u)\diffd u .$ There is, for each $v\ge 0$, a unique strong solution of the stochastic differential equation:
\begin{equation}\label{flow lb}
\begin{split}
                   & Z_{t}(v)  =\  v  - \alpha\int_{0}^{t}Z_{s}(v) \diffd s+\sigma\int_{0}^{t}\int_0^{Z_{s-}(v)}W(\diffd s,\diffd u) + \int_{0}^{t}\int_{0}^{Z_{s-}(v)}\int_{0}^{\infty}r\tilde{N}(\diffd s,\diffd \nu,\diffd r)  \\& 
                       \qquad \qquad - \int_{0}^{t} G(Z_s(v)) \diffd s, \quad 
                    t\ge 0.\\
\end{split}
\end{equation}
Moreover,  the process $(Z_t(v) : t \geq 0, v\geq 0)$ admits a (bi-measurable) version such that   
\begin{itemize}
 \item[i)] for each $v \geq 0$, $t \mapsto Z_t(v)$ is a c\`adl\`ag process on $[0, \infty)$; 
 \item[ii)] for each $t \geq 0$, $v \mapsto Z_t(v)$ is a non-negative and non-decreasing c\`adl\`ag process on $[0, \infty)$;
 \item[iii)]  for each $u\geq 0 $, the  conditional law of $(Z_t(v)- Z_t(u): t\geq 0, v\geq u )$ given $\left(  Z_t(x) : t \geq 0,0\leq  x\leq u \right)$ depends only on   $\left(  Z_t(u) : t \geq 0\right)$; and
 \item[iv)] defining  for each $s\geq 0$  the process $(Z_{s,t}(v): t\geq s, v\geq 0)$ solution to
 \begin{linenomath*}
 \begin{equation*}
 \begin{split}
  Z_{s,t}(v)  =\, &  v  - \alpha\int_{s}^{t}Z_{s,r}(v)\diffd r+\sigma\int_{s}^{t}\int_0^{Z_{s,r-}(v)}W(\diffd r,\diffd u) 
              + \int_{s}^{t}\int_{0}^{Z_{s,r-}(v)}\int_{0}^{\infty}u\tilde{N}(\diffd r,\diffd \nu,\diffd u) \\
                  & \quad  - \int_{s}^{t} G(Z_{s,r}(v)) \diffd r, \\
 \end{split}
 \end{equation*}
 \end{linenomath*}
 we have  for every $0\leq s\leq r$  a.s. that $(Z_{s,t}(v): t\geq r)=(Z_{r,t}(Z_{s,r}(v)): t\geq r)$.
\end{itemize}           
\end{proposition}

We call $(Z_t(v): t\geq 0, v\ge 0)$  a  stochastic flow of  branching processes  with branching mechanism $\psi$ and competition mechanism $g$, or simply {\it   stochastic flow of branching processes with competition}, if the mechanisms are known and fixed. 

In the case that $g= C> 0$ is   constant, one clearly recovers a flow of CSBP as in  \eqref{eqflow csbp} but, in general, the branching property is lost (notice the change in property (iii) between solutions of \eqref{eqflow csbp}  and solutions of \eqref{flow lb}). 
It is easily checked that in the linear case,  $g(x)=2cx$ say, for each $v\geq 0$ the process $t\mapsto Z_t(v)$ reduces to the usual logistic branching process studied in \cite{L1} (with competition intensity $c>0$). 
If moreover $\Pi=0$,  one gets the so-called logistic Feller diffusion 
\begin{equation}\label{EQDL}
     \diffd Z_t = \left(-\alpha Z_t - c  Z_t^2 \right) \diffd t + \sigma\sqrt{Z_t}\diffd B_t ,  \quad Z_0 = v,
\end{equation}
the genealogy of which was studied by  Le {\it et al.}  \cite{PWL} and Pardoux and Wakolbinger  \cite{PW}  when $-\alpha\geq 0$,  by  means of approximations  with discrete logistic  processes. Extensions of those works  to much general  drift terms  were developed in  \cite{BaP} and \cite{Px},  but still in the diffusive  branching case.

The competition dynamics implicit in equation \eqref{flow lb} is determined by  an ordering of the population. 
Indeed, if the complete population is identified with  the positive half-line, we are implicitly  assuming that an individual at $x>0$ can only be killed ``from below'', i.e. by the fraction of the population in $[0,x)$,  at rate  given by $g(x)$. 
The downward drift $G(Z_t(v))=\int_0^{Z_t(v)} g(u) \diffd u$ thus  corresponds to the total rate of killing  at time $t$ in the population started from size $v$ at time $0$.  
This  picture is similar to the point of view  adopted in \cite{PWL,PW} to establish a genealogical representation for the  logistic Feller diffusion, which we further discuss later, along with other  relations  between our work and \cite{PWL,PW,BaP,Px}.
 
 Although one can already read  genealogical structure in \eqref{flow lb}, our main aim  here is to describe this genealogy using a  L\'evy CRT, that is, in a  manner analogous to  the Ray-Knight theorem for CSBP.  Following \cite{LGLJ}  and \cite{DLG}, we next  recall in details  that result  and introduce some basic  objects we require. 
  
\subsection{L\'evy CRT and Ray-Knight Theorem for flows of CSBP}\label{explo}

We briefly review here the exploration process, introduced in \cite{LGLJ} to construct the L\'evy continuum random tree from a spectrally positive L\'evy process  and  used in \cite{DLG} to  define the continuum genealogy of  the associated CSBP. We follow \cite{DLG} and \cite{ADV} and refer to Chapter 1 of the former for detailed proofs and further background.

\smallskip 

Let $X$ be a L\'evy process with Laplace exponent $\psi$ satisfying conditions \eqref{psi2} and \eqref{psi3}. Zero is then regular for $X$ reflected  at its {\it running infimum}. 
We denote the latter process by $I_t: = \inf\limits_{ 0\leq s \leq t}X_s$ and recall that $-I$ is a local time at $0$ for the strong Makov process $X-I$. 
For $0 \leq s \leq t$ we denote by  $I_t^s= \inf\limits_{s\leq r \leq t} X_r$ the two parameter process known as \textit{future infimum} of $X$. 
The \textit{height process} $H^0=(H_t^0: t\geq 0)$ is defined for each $t\geq 0$ as 
$H_t ^0 = \liminf\limits_{\varepsilon \rightarrow 0} \frac{1}{\varepsilon} \int_0^t \I_{\{X_s < I_t^s + \varepsilon\}} \diffd s$ (which, by time reversal,   is the total local time of the dual L\'evy process on $[0,t]$ reflected  at its supremum) and measures for each $t\geq 0$ the size of the set $\{s \leq t : X_s = \inf_{[s,t]} X_r \}$. 
If  $X$ has no jumps,  $H$  is a reflected Brownian motion with drift but it might  in general not be Markovian nor a semimartingale.  However, $H$  does always have a version which is a measurable function of some strong Markov process, called exploration process. 
The \textit{exploration process} $\rho = (\rho_t: t \geq 0)$ takes values in the space of finite measures in $\R_+$ and, for each $t\geq 0$,  it is defined on bounded measurable functions $f$ by 
\[
\langle \rho_t , f \rangle = \int_0^t \diffd_s I_t^s f(H_s^0 ),
\]
where $\diffd_s I_t^s$ denotes Lebesgue-Stieltjes integration with respect to the nondecreasing map $s \mapsto I_t^s$. Equivalently, one can write
\begin{equation}\label{rhodef}
    \rho_t(\diffd r)=\frac{\sigma^2}{2} \I_{[0,H_t ^0]}(r)\diffd r \ + {\sum\limits_{0<s\leq t,X_{s-} < I^s_t}}(I_t^s - X_{s-}) \delta_{H_s^0}(\diffd r).
\end{equation}
In particular, the measure $\rho_t$ can be written as a function of the excursion of the reflected L\'evy process $(X_s - I_s: \, s\geq 0)$ straddling $t$. 
Furthermore,   $t \mapsto \rho_t$ is c\`adl\`ag in the variation norm, $\rho_t$ has total mass $\langle \rho_t , 1 \rangle = X_t - I_t$ for each $t\geq 0$  and   $\rho_t(\{0\})=0$. 
The process defined by $H_t : = \sup \mbox{supp} (\rho_t)$ (with  $\mbox{supp}(\mu)$  the topological support of $\mu$ and the convention  that $ \sup \emptyset = 0$) is a continuous modification of  $H^0_t$ and one has  $ \mbox{supp} (\rho_t)=[0,H_t]$ when $\rho_t\neq 0$.  
This also implies that the excursion intervals out of zero are the same for $X-I$, $\rho$ or $H$. 
We call $\mathbf{N}$ the excursion measure (away from zero) of the strong Markov process  $\rho$ or, equivalently,  of $X-I$.  

\smallskip 
      
Under $\mathbf{N}$, the process $s\mapsto H_s$  is a.e. continuous, non-negative  with compact support and we have $H_0=0.$ It therefore encodes a tree, as follows. 
Let $\zeta=\inf\{ s>0: H_s =0\}$ denote  the length of the canonical excursion under  $\mathbf{N}$. 
The  function $d_H$  on $[0,\zeta]^2$ given by 
\[
d_H(s, t) = H_s + H_t - 2 H_{s, t}  \quad \mbox{with} \quad 
H_{s, t} :  = \inf\limits_{s\wedge t\leq r \leq s\vee t} H_r\, ,
\] 
defines an equivalence relation $\sim_H$ whereby 
$
(s \sim_H t ) \Leftrightarrow   (d_H(s, t) = 0),$  $\forall s,t \le \zeta.$
Hence,  $d_H$   induces a distance on  the quotient space   $\T_H=[0,\zeta]/ \sim_H$  which is thus  a compact metric space, more precisely,   a ``real tree''. That is, any two points are joined by a unique, up to re-parametrization, continuous injective path,  isomorphic to a line segment  \footnote{The reader is referred to \cite{DLGI} for further topological background.}.

The {\it $\psi$-L\'evy random tree}  (or L\'evy CRT) is the real  tree $(\T_H,d_H)$ coded by $H$ under the measure $\mathbf{N}.$
 Hence, each $s \in [0, \zeta]$ labels a vertex at height $H_s$ in the tree and $d_H(s, t)$ is the distance between vertices corresponding to $s$ and $t$; accordingly, $H_{s, t}$ represents the height (or  generation) of the most recent  ancestor common to $s$ and $t$. The equivalence class  of $s=0$ is termed {\it root} and  the unique path  isomorphic to $[0,H_t]$ connecting it with  the class of $t\geq 0$, is interpreted as the  ancestral line or ``lineage'' of  the individual corresponding to $t$. 
Thus, $\rho_t$ can be seen as a measure on this  lineage, describing the mass of sub-trees grafted on its right. 
 
Under $\P$, the process of excursions of  $\rho$ is Poisson with respect to the local time clock at level $0$, with intensity measure $\mathbf{N}$  (much as  reflected Brownian motion is a Poisson point process of Brownian excursions).  The genealogy of the population is then described by the ``L\'evy forest''  $\T$ defined as the union of the  corresponding Poissonian collection of trees.
  
\smallskip 

 Since $H_t = 0$ if and only if $X_t-I_t = 0$,   the local time  at level $0$ of $H$ is naturally  defined as the process $L_t^0:=-I_t$. 
One way of defining the local time process $(L^a_t:t\geq 0)$ of  $H$ at  levels  $a>0 $ is through  the spatial Markov  property  of the exploration process:  if   for each $t\geq 0$ we set
\begin{equation}\label{tauat} 
    \tau_t^a : = \inf \{s \geq 0 : \int_0^s \I_{\{H_r >a\}}\diffd r > t\} 
             = \inf\{s \geq 0 : \int_0^s \I_{\{\rho_r((a,\infty))>0\}}\diffd r >t\} 
\end{equation}
and $\tilde{\tau}_t^a : = \inf\{s \geq 0 : \int_0^s \I_{\{H_r \leq a\}} \diffd r > t\}$, then the  measure valued process $(\rho^a_t : t \geq 0)$  defined by 
\begin{equation}\label{rhoa}
     \langle \rho_t^a , f \rangle = \int_{(a, \infty)} \rho_{\tau_t^a} (\diffd r) f (r - a)
\end{equation}  
has the same distribution as $(\rho_t : t \geq 0)$, and is independent of the sigma field generated by the   process $((X_{\tilde{\tau}_t^a} , \rho_{\tilde{\tau}_t^a}) : t \geq 0)$. 
Thus, for each $t\geq 0$, we define $L^a_t : = l^a \left( \int_0^t \I_{\{H_r >a\}} \diffd r \right)$ where $l^a :  = (l^a (s): s \geq 0)$ denotes the local time at $0$ of the  L\'evy process reflected at its infimum  $(\langle \rho^a_s , 1 \rangle  :s \geq 0)$.

The Ray-Knight theorem of Duquesne and Le Gall \cite[Theorem 1.4.1]{DLG} states that, for each $x\geq 0$,  a  $\psi$-CSBP starting at $x$ is equal in law to the processes $(L_{T_x}^a: a \geq 0)$, where 
\begin{equation}\label{Tx}
  T_x: = \inf\{t \geq 0: L_t^0= x\}.
\end{equation}
Thanks to the strong Markov property of the exploration process and to  the branching property iii) of process  \eqref{eqflow csbp}, this result obviously extends to a two-parameter processes as follows:
 
\begin{theorem}[{\bf Ray-Knight representation for flows of CSBP}]\label{RKDL}
The process $(L_{T_x}^a: a \geq 0, x\ge 0)$ and  the process $(Y_a(x): a\geq 0, x\ge 0)$ given by \eqref{eqflow csbp} have the same law. 
\end{theorem}  
   
\subsection{Pruning of L\'evy trees}\label{Pruning of Levy trees}

Following  \cite{ADV, ADH},  the {\it pruning}  of a real tree or forest  $\T$  at a discrete set of  points $\bar{\tau}\subset \T$ is  the subset  of $\T$ defined as the union of  the connected components of $\T\setminus \bar{\tau}$  containing the roots.  
According to those works, if conditionally on the corresponding L\'evy forest  $\T$ the point configuration  is randomly distributed as a Poisson point process of intensity $\theta>0$ (with respect to the natural length  measure on the skeleton of $\T$), then the resulting pruned random subforest  $\T^{\theta}$ has the same law as the random L\'evy forest associated with the branching mechanism 
\begin{equation}\label{psitheta}
  \psi_{\theta}(\lambda):= \psi(\lambda) + \theta \lambda \, ;
\end{equation}
see \cite{ADV} and Proposition \ref{markRK} below for rigorous statements in terms of  exploration and local time processes, respectively. 
In order to formulate a genealogical representation for process \eqref{flow lb} analogue to Theorem \ref{RKDL}, we will extend those ideas, by pruning the L\'evy forest  at variable rates. 
 
To that end, we need to first  formalize the notion of a Poissonian configuration of points on the product space $\T \times \R_+$,  with, as intensity, the product of  the respective length (Lebesgue) measures. 
The following notation will be used in the sequel:
\begin{itemize}
 \item   $ \M_f^0(\R_+) $  stands for  the space of compactly supported  Borel  measures on $\R_+$. For  $\mu \in  \M_f^0(\R_+) $, we set $H(\mu):=\sup \mbox{supp} (\mu)$, where  $\mbox{supp}(\mu)$  is the topological support of $\mu$ and  $ \sup \emptyset: = 0$.
 \item   $ \M_{at}(\R_+^2) $  denotes  the space of  atomic Borel measures  on $\R_+^2$ with unit mass atoms.   
 \item  ${\cal V}:=\big\{ (\mu, \eta)\in \M^0_f(\R_+)\times \M_{at}(\R^2_+): \,  \mbox{ supp } \mu=  [0, H(\mu)) $ and  $ \mbox{ supp } \eta\subseteq [0, H(\mu)) \times \R_+ \big\}$.
\end{itemize}

The object next  defined is an instance of the  snake processes introduced in \cite{DLG}. It extends, in a way,  a Poisson L\'evy-snake used in \cite{ADV} to prune a forest at constant rate, and it  is essentially a variant of the objects used  in \cite{AS,ADfrag} to define fragmentation processes by means of L\'evy trees. (See Section \ref{poissonsnake} below  for its construction  as a snake process and for the precise topology put on ${\cal V}$.)

\begin{definition} \label{def of 2d marked explo} 
We call exploration process with positive marks or simply  {\it marked exploration process} a  c\`adl\`ag strong Markov process $((\rho_s , \N_s)  : s \geq 0)$ with values in ${\cal V}$  such that
\begin{enumerate}
       \item $(\rho_s : s \geq 0)$ is the exploration process associated with the L\'evy process $X$;              \item conditionally on $(\rho_s : s \geq 0)$,  for each $s\geq  0$, $ \N_s $ is a Poisson point measure on $[0,H_s)\times \R_+ $   with  intensity the Lebesgue measure. Moreover,  for all  $0 \leq s \leq s'$,
       \begin{itemize}
       \item   $\N_{s'} (\diffd r, \diffd\nu) \I_{\{ r \leq H_{s,s'} \}} = \N_s ( \diffd r, \diffd \nu) \I_{\{ r \leq H_{s,s'} \}}$  a.s.     
        and\\
       \item  $\N_{s'} (\diffd r,  \diffd \nu) \I_{\{ r>H_{s,s'} \}}$ and $\N_s (\diffd r, \diffd \nu) \I_{\{ r > H_{s,s'} \}}$ are independent point processes.
       \end{itemize}
\end{enumerate}
\end{definition}    
The last two properties  in  point 2.\  naturally correspond to the random tree structure of $\T$  and are  classically referred to as the {\it snake property}. Thanks to them,   the Poisson snake $\N$ induces a point process in the obvious quotient space $\T \times \R_+$ of  $ \R_+ \times \R_+$.  Hence, the Poisson process $\N_s$ describes  the restriction of that point process  to  the unique path in $\T$   isometric to $[0,H_s]$  joining  the root and  the point  labeled   $s$.

 Note that given $\theta >0$, the point  process defined for  each $t\geq 0$  by 
\[
m_t^\theta(A) = \mathcal{N}_t(A\times [0,\theta]) \, \,  \forall  \mbox{ Borel set }A \subseteq [0,H_t),
\]   
 
  is  standard Poisson of rate $\theta$  on $ [0,H_t)$, conditionally on $\T$. The pair  $((\rho_t , m^\theta_t)  : t \geq 0)$  thus corresponds to the  exploration process ``marked  on the skeleton  at rate $\theta$'' of \cite{ADV}. According to that work, 
the ``pruned exploration process''  defined  by $  (\rho^\theta_t: t\geq 0) : = (\rho_{C^\theta_t}: t\geq 0) $, with  $C^\theta_t $ the  right-continuous inverse of  
\[
A^\theta_t : =   \int_0^t \I_{\{m^{\theta}_s  =0\}} \diffd s,
\]  
codes the genealogy of the population of individuals with no marked ancestors (cf. time changing by $C^\theta_t$ when exploring the tree amounts to skipping  individuals with an atom of $m^{\theta}$ in their lineages). It was then shown in   \cite[Theorem 0.3]{ADV} that $ (\rho_{C^\theta_t}: t\geq 0) $  has the same law as the exploration process of a CSBP with branching mechanism $\psi_{\theta}$ as in \eqref{psitheta}.  That result can be restated in a way that connects pruning, local times  and stochastic flows of CSBP,  which is  relevant to both motivate and prove our main results:

\begin{proposition}\label{markRK}
 Let $\theta>0$ be fixed and for each $ t,a \geq 0$ define 
 \[
 L_t^a(m^{\theta}):=  \int_0^t  \mathbf{1}_{\{m^\theta_s=0\}} \diffd L_s^a= \int_0^t  \mathbf{1}_{\{m^\theta_s([0,H_s))=0\}} \diffd L_s^a.
 \]
 Then, the process defined by $\displaystyle{\left( L_{T_x}^a (m^{\theta}) :  a \geq 0,x \geq 0\right)}$ with $T_x$ as in \eqref{Tx}, has the same law as the stochastic flow of CSBP 
 \begin{equation}\label{flow constant theta}
 \begin{split}
                    Y^{\theta}_{a}(x)=\ &  x - \alpha\int_{0}^{a} Y^{\theta}_{s}(x)\diffd s+\sigma\int_{0}^{a}\int_0^{ Y^{\theta}_{s-}(x)}W(\diffd s,\diffd u)\\
                     & + \int_{0}^{a}\int_{0}^{ Y^{\theta}_{s-}(x)}\int_{0}^{\infty}r\tilde{N}(\diffd s,\diffd \nu,\diffd r) -  \int_{0}^{a}  \theta   Y^{\theta}_s(x) \diffd s, \quad t\ge 0, x\ge 0.
                   \end{split}  
 \end{equation}
\end{proposition}
The proof of Proposition \ref{markRK} is given in  Appendix \ref{proofmarkRK}.
  In  words,   this result  states that pruning $\T$ at constant rate $\theta$ and measuring  the local times  $\displaystyle{\left( L_t^a (m^{\theta}) :  t\geq 0, a \geq 0\right)}$ in the obtained subtree, amounts to introducing at the level of stochastic flows an additional drift term that corresponds to a  constant competition mechanism $g(y)=\theta$.   This suggests us that, in order to obtain a pruned forest whose local times equal in law a flow of branching processes  with general  competition   $g$ (as in  \eqref{flow lb}),   one  should first be able to prune  $\T$ at a rate that can dynamically  depend on  the ``previous information'' in the exploration-time and height senses. 

\smallskip 

With this aim, denote by $ {\cal P}red(\N)$ (for predictable) the sigma-field generated by continuous  two-parameter processes, whose value at each point  $(t,h)$ is measurable with respect to the exploration process $\rho$ up to time $t$, and to atoms  of the Poisson snake  $\N$ below height $h$ up to time $t$ (see \eqref{Gtr} and Definition \ref{snakepredictable} in Section \ref{interprun} for the precise definitions).

\begin{definition}\label{adapintensity}
A real process $( \intens(t,h) : t\geq 0, h\geq 0)$ is called a height-time  adapted intensity process  or simply {\it adapted intensity process} if  the following properties holds:
\begin{enumerate}
 \item It  has a version which is  $ {\cal P}red(\N)$-measurable. 
 \item For each $t\ge0$, $h\mapsto \intens(t,h)$ is  a.s. locally integrable.  
 \item  For each pair $s, s' \ge 0$ one a.s. has  $\intens(s',h)  =\intens(s,h) $ for  $\diffd h$  a.e.    $h \leq H_{s,s'}$.
\end{enumerate} 
\end{definition}
Property  3.\ above will be refereed to as the  {\it semi-snake property of adapted intensity processes}. 
Using these objects and the Poisson snake,  we can  put marks on  the tree at height-time adapted rates as follows:  

\begin{definition}\label{def of prune with adapted process}
Let $((\rho_s, \mathcal{N}_s) :s\ge 0)$ be the marked exploration process and $\intens=(\intens(t,h): t\ge 0, h\ge0)$ be an adapted intensity process. 
The {\it  exploration process marked at adapted intensity $\intens$} is the process $((\rho_s , m^\intens_s)  : s \geq 0)$, with  $m_t^\intens \in \M_{at}(\R_+) $ the point process given for each $t\geq 0$  by 
\begin{equation}\label{markedintensity}
       m_t^\intens(A) = \mathcal{N}_t(\{  (h,\nu) : h\in A, \nu \le \intens(t,h) \})   \mbox{ for all Borels set } A \subseteq [0,H_t).
\end{equation}
\end{definition}

Last, in order to measure at each height the mass of individuals in the pruned forest  up to a given exploration time, and generalizing the notion in Proposition \ref{markRK}, we introduce:
 
\begin{definition}
  We call   local time process  pruned at rate  $\intens$, or  {\it $m^{\intens}$-pruned local time}, the   process  $(L_t^r(m^{\intens}):t\geq 0, r\geq 0)$ defined by
 \[
 L_t^r(m^{\intens}):=  \int_0^t  \mathbf{1}_{\{m^\intens_s=0\}} \diffd L_s^r= \int_0^t  \mathbf{1}_{\{m^\intens_s([0,H_s))=0\}} \diffd L_s^r \, , \quad r\geq 0, t\geq 0.
 \] 
\end{definition}  
                      
With these notions in hand, we now turn to the questions of what adapted intensity  process  we should look for,  and state our main results. 
           
\subsection{Main results}
 
A closer comparison of \eqref{flow constant theta} and \eqref{flow lb} indicates that, in order to obtain a subforest $\T^*$ encoding the genealogy of the process  \eqref{flow lb},  we should prune $\T$ at  each point  at rate equal to $g$ of the local time left of it  in $\T^*$. In terms of the objects we have just introduced, this  means  that the corresponding adapted intensity $\intens^*$  should solve,  in a certain sense, some fixed point equation. The next result gives a precise meaning to such an object and grants its existence:   

\begin{theorem}\label{main thm 1}
Let $g: [0,\infty) \to [0,\infty)$ be a competition mechanism and let  $F:\intens \mapsto F(\intens)$ denote the operator acting on adapted intensity processes by 
\[
F(\intens) =\intens' \, , \text{ with } \intens'(t,h) = g(L^h_t(m^{\intens})) \quad \text{for all  }(t,h)\in\R_+^2 \,.
\]
Then, $F(\intens)$ is also an adapted intensity process. 
Moreover, assume  the following condition  on $g$ holds: 
\begin{equation*}
 \mbox{{\bf (H)}}\quad \quad    g:\R_+\to \R_+   \mbox{  is non  decreasing  and  locally Lipschitz continuous. }  \hspace{4cm} \end{equation*} 
Then the operator $F$ has a fixed point $\intens^*(t,h)$, a.e unique with respect to  $\mathbb{P}(\diffd \omega)\otimes \diffd t \otimes \I_{[0,H_t(\omega))}(h) \diffd h$.
\end{theorem}

  Notice that under  {\bf (H)} (and  by definition of the competition mechanism $g$), the  function $G=\int_0^{\cdot} g(y)\diffd y$ in   \eqref{flow lb} is  non negative, non  decreasing, convex   and has a locally Lipschitz derivative.  \\
   The following theorem provides the genealogical representation of a flow of branching process with competition  we are looking for and is the main result of the paper (note that $T_x$ is the same as in  \eqref{Tx}).
   
\begin{theorem}\label{main thm 2}
Suppose that   {\bf (H)}   holds and let $\intens^*$ be the adapted intensity process given by  Theorem \ref{main thm 1}. 
Then, the process  $(L_{T_x}^a(m^{\intens^*}) : a\ge 0, x\ge 0)$ and the stochastic flow of branching processes with competition $   (Z_a(x) :  a\ge 0, x\ge 0)$  defined in \eqref{flow lb}  are equal in law.
\end{theorem}      
               
\subsection{Organization of the paper}

In Section {\bf \ref{prop flow}}, we prove Proposition \ref{flowLBP}, using variants of techniques introduced in \cite{DL} and some results  established therein. 
A Lamperti type representation for process \eqref{flowLBP} (for  fixed $v$) is also discussed. 

In Section {\bf \ref{explosnakeprun}} we will first settle  some technical issues required to prove Theorems  \ref{main thm 1} and  \ref{main thm 2}, concerning  in particular filtrations and measurability aspects, as well as  properties of the  operator $F$.  
We will then study the sequence of adapted intensities defined by $\vartheta^{(0)}\equiv 0$ and $\vartheta^{(n+1)}=F(\vartheta^{(n)})$, together with  the  pruned local times they  induce. 
Heuristically, noting that  $\vartheta^{(1)}$ yields a pruned local time process which is smaller than the (non pruned) one associated with $\intens^{(0)}$,  the fact that $g$ is nondecreasing implies  that  $\vartheta^{(2)}$  is smaller than $\vartheta^{(1)}$. Pruning according to $\vartheta^{(2)}$ then implies ``pruning less'' than when using $\vartheta^{(1)}$ and this yields that $\vartheta^{(2)} \le \vartheta^{(3)} \le \vartheta^{(1)}.$ 
By iterating this reasoning one gets  that $\vartheta^{(2)} \le \vartheta^{(4)} \le \ldots \le \vartheta^{(5)} \le \vartheta^{(3)} \le \vartheta^{(1)} $. 
The proof of Theorem  \ref{main thm 1} makes  this remark rigorous  and is achieved showing that the sequences $(\vartheta^{(2n)})$ and $(\vartheta^{(2n+1)})$ converge to the same limit, which is a fixed point of $F.$ 

In Section {\bf    \ref{secRKT}}  we will prove Theorem \ref{main thm 2} by  putting in place a careful approximation argument inspired by Proposition \ref{markRK}. 
More precisely, we  will first construct an approximation of the stochastic flow of  branching processes with competition \eqref{flow lb}  by means of  a flow of    CSBP with ``piecewise frozen'' killing rate,  specified by constant negative drifts  in a rectangular bi-dimensional grid of  constant  time-step  length  and  initial population size. 
We  will then construct  an approximation of the  pruned local time process given by Theorem  \ref{main thm 1}, pruning the original  forest at constant rates inside blocks of some height-local time grid. The blocks of this second  grid are defined in such a way that they are in  correspondence, {\it via}  Proposition \ref{markRK}, with the rectangles of the stochastic flow grid. 
The  proof of Theorem \ref{main thm 2} will then be deduced, noting  first that  these two grid approximations globally have  (as two parameter processes  defined in different spaces) the same law, and then showing that  when the  rectangle sizes go to $0$, each of them does actually converge (in a strong enough way) to the respective processes  they intend to  approximate.  
The technical proofs of these facts are differed to the last two sections.  
    
In Section  {\bf    \ref{consitstochflow}} we prove the convergence  of the  stochastic flow grid-approximation,  using stochastic calculus and some technical ideas  adapted from \cite{DL}. 
In Section   {\bf    \ref{ConsistLawGridTree}}, combining Proposition \ref{markRK} with global properties of stochastic flows and of local times, we  first identify the law of the local time  grid approximation with that of the stochastic flow grid  approximation. 
Then,  we prove the convergence of the former  to the pruned local time process constructed in Theorem \ref{main thm 1}, using mainly  techniques from snake-excursion theory. 
Finally, some technical or auxiliary results will  be proved in the Appendix. 
 
\smallskip 

Before delving in the proofs, we devote the next subsection to a  discussion of  our model and results and of their relations to related works in the literature.

\subsection{Discussion of related results}

In \cite{PWL}, (see also  \cite{PW1,PW}) a Ray-Knight representation theorem for the logistic Feller diffusion was obtained. 
These authors constructed a process $H$ which can roughly be described as a reflected Brownian motion with a negative drift, proportional to the amount of local time  to the left and  at the same height of the current state. More precisely, $H$ is defined as the (unique in law) solution of the SDE 
\begin{equation}\label{RKLPW}
         H_s = \frac{2}{\sigma}B_s + \frac{1}{2}L_s^0(H) - \frac{2\alpha}{\sigma}s - c \int_0^s L_r^{H_r}(H)\diffd r, \quad s \geq 0,
\end{equation}  
where $\alpha<0$, $B$ is a standard Brownian motion and  $L_s^a(H)$  the local time accumulated by $H$ at level $a\geq 0$ up to time $s\geq 0$. 
The main result is that the local time process of $H$, i.e. the process $((\sigma^2/4)L_{T_x}^a(H): a \geq 0)$, is a weak solution of \eqref{EQDL}  (here, $T_x$ is the stopping time   $ T_x = \inf\{s > 0:  L_s^0(H)  > x\}$ which is unchanged if $H$ is replaced by $B$). 
The downward drift of $H$ can thus be understood as a rate of killing which increases the farther to the right one looks, so that excursions  appearing  later with respect to  exploration time  tend to be smaller. 
Individuals are  thus represented as being  arranged ``from left to right'' and as though interacting  through ``pairwise fights'' which  are always won by the individual ``to the left'' (hence  lethal for the individual ``to the right''). 
Under another (more pathwise)  guise, we adopt the same ``trees under attack''  picture here using pruning. There is nevertheless no obvious  extension to the present  setting, of arguments and techniques introduced in  \cite{PWL} which strongly  relied on   the  Brownian motion based description   of the corresponding  height process, and  on the possibility of building  tractable  discrete weak approximations of it. 

 We must  also point out  that  our  arguments do not cover the case of  supercritical branching mechanism  treated in \cite{PWL,PW1,PW}  ($\alpha<0$,  in our notation), nor the general type of interactions considered  in \cite{BaP} and \cite{Px} (conditions $f\in C^1$ and (H1') in the latter would correspond  here  to  assuming $g$  only to be continuous and bounded from below).

\smallskip

In \cite{BDMZ}, Berestycki et al.\ study measure valued branching processes {\it with interactive immigration}, as  introduced and discussed by Li in \cite{BookLi}, in the spirit of the systems of stochastic differential equations (1.5) of  \cite{DL}. 
The main result in  \cite{BDMZ} is a Ray-Knight representation for those branching processes with immigration. More precisely, suppose that we want a representation of the solutions of 
\begin{equation}\label{(2.7)} %\tag{MBI} 
	\begin{cases} 
		X_t(v) = v +   \int_{0}^{t}\int_{0}^{X_{s-}(v)}\int_{0}^{\infty}r \tilde{N}(\diffd s,\diffd \nu,\diffd r) + v \int_0^t  g(X_{s}(1))\, \diffd s, \\	
		v \in [0,1], \ t\geq 0,
	\end{cases}
\end{equation}
where $g$ is now an {\it immigration} mechanism, supposed to be non-negative, monotone non-decreasing, continuous and locally Lipschitz continuous away from zero. 
Observe that this is very similar to equation \eqref{flow lb} with $\sigma=0$ and a positive drift term instead of a negative one, which is however also dependent on the current state of the population.
Let $(s_i, u_i, a_i, \exc_i)$ be a Poisson point process with  intensity measure $\diffd s \otimes \diffd u \otimes \diffd a \otimes \mathbf{N}(\diffd(\exc))$ (with $\mathbf{N}$ denoting here  the excursion measure of the height process $H$) and define $w_t^i$ to be the total local time at level $t$ of the excursion $\exc_i$. 
It is  shown in \cite{BDMZ} that, for all $v\in[0,1]$, there is a unique c\`adl\`ag process $(Z_t(v), t\geq 0)$ satisfying $\P$-a.s.
\begin{equation}\label{Z}
 \begin{cases} 
  Z_t(v) = \sum_{s_i>0} w_{t}^i \,\un{(a_i \le v)} \un{(u_i \le g(Z_{s_i-}(1)))},\qquad t> 0,\\
  Z_0(v)=v
 \end{cases}
\end{equation}
and $Z$ is a weak solution of \eqref{(2.7)}.\footnote{The full result is actually that the Poisson point process can be constructed to obtain a strong solution with respect to a given noise to equation  \eqref{(2.7)}.}
The existence and uniqueness of the process $Z$ is analogue to our Theorem \ref{main thm 1} (and is also proved by a  fixed point argument) while the fact that $Z$ is a solution of \eqref{(2.7)} is analogue to our Theorem \ref{main thm 2}. 
Let us underline, however, that in \cite{BDMZ} the positive drift represents an interactive {\it immigration} while here we have a negative drift corresponding to a {\it competition}. 
This is reflected of course in the fact that instead of {\it adding} excursions as in \cite{BDMZ} we must here {\it prune} the real forest coding the evolution of the population. 
In spite of the similarities,  adapting  techniques in \cite{BDMZ} to the present setting does not seem to be simple.
 
%%%%%%%%%%%%%%%%%%%%%%%%%%%%%%%%%%%%%%%%%%%%%%%%%%%%%%%%%%%%%%%%%%%%%%%%%          
%%%%%%%%%%%%%%%%%%%%%%%%%%%%%%%%%%%%%%%%%%%%%%%%%%%%%%%%%%%%%%%%%%%%%%%%%
%%%%%%%%%%%%%%%%%%%%%%%%%%%%%%%%%%%%%%%%%%%%%%%%%%%%%%%%%%%%%%%%%%%%%%%%%        

%%%%%%%%%        The rest is copy-paste from version_cartagena sent by Jo 8/1/2015

%%%%%%%%%%%%%%%%%%%%%%%%%%%%%%%%%%%%%%%%%%%%%%%%%%%%%%%%%%%%%%%%%%%%%%%%%          
%%%%%%%%%%%%%%%%%%%%%%%%%%%%%%%%%%%%%%%%%%%%%%%%%%%%%%%%%%%%%%%%%%%%%%%%%
%%%%%%%%%%%%%%%%%%%%%%%%%%%%%%%%%%%%%%%%%%%%%%%%%%%%%%%%%%%%%%%%%%%%%%%%%        
      
\section{Continuous state branching processes with competition  }\label{prop flow} 
   
Recall that  $G(x)=\int_0^x g(y) \diffd y$ denotes the primitive of the competition mechanism $g$. 
In this section, given a fixed function $z\in \mathbb{D}(\R_+,\R_+)$,  we  will moreover consider  a time inhomogeneous competition mechanism $g_{z}:\R_+\times \R_+\to \R_+$ defined as $g_z(t,y)=g(z_t+ y)$. 
We denote  its  primitive in $y$ by $G_z(t,x):=\int_0^x g(z_t+y) \diffd y$, which is locally Lipschitz in $x$, locally uniformly in $t$.  

Also, we will repeatedly use the following analogue  to  L\'evy's characterization of $n$-dimensional Brownian motion,  whose simple proof we give in Appendix \ref{proofLemmaweirdwnandpp} for completeness. 
 \begin{lemma}\label{weirdwnandpp}
Given $({\cal S}_t)_{t\geq 0}$  a filtration  in some probability space and $n,m\in \NN$,  let $(M^i_t: t	\geq 0)_{ i\leq n}$ and $(N^k_t: t	\geq 0)_{ k\leq m}$   respectively be  continuous $({\cal S}_t)$-local martingales and   $({\cal S}_t)$-adapted counting processes. Assume that, for some real numbers  $a_i>0,i\leq n,$ and $b_k>0, k\leq m$,  one has: 
\begin{itemize}
\item[i)]  for every $i,j\leq n$  and all $t\geq 0$,   $[ M^i,M^j]_t=a_i \delta_{ij} t $ (with $\delta_{ij} $ the Kronecker delta); and
\item[ii)]  for every $k,l\leq m$   the processes $(N^k_t)_{t\geq 0}$ and $(N^l_t)_{t\geq 0}$ have  no simultaneous jumps, and $(N^k_t)_{t\geq 0}$  has the predictable compensator $b_k t$. 
\end{itemize}
Then,  $(a^{-1/2}_i M^i_t: t\geq 0)_{ i\leq n}$ are standard $({\cal S}_t)$-Brownian motions,   $(N^k_t: t\geq 0)_{ k\leq m}$ are  $({\cal S}_t)$- Poisson processes with respective  parameters $(b_k)_{ k\leq m}$ and all these processes are independent. 
  \end{lemma}

\subsection{Basic properties of the stochastic flow}

\begin{proof}[Proof of  Proposition \ref{flowLBP} ]  

We first prove parts  i) and iii), leaving  ii) and  iv)  for the end of the proof.

i)   The tuple of functions $(b, \sigma, g_0, g_1 )$ defined by
 \begin{itemize}
      \item  $x \mapsto b(x):= - \alpha x - G(x)$ ;
      \item $(x, u) \mapsto \sigma(x, u):=\sigma\I_{\{u \leq x\}}$;
      \item $(x, \nu, r) \mapsto g_0 (x, \nu,r)=g_1 (x, \nu,r):= r\I_{\{\nu \leq x\}}$,
 \end{itemize}
 are admissible parameters in the sense of conditions $(2.a,b,c,e)$ in \cite[Section 2]{DL}. 
 Then, \cite[Theorem 2.5]{DL} readily provides for each $v \geq 0 $ the existence of a unique strong solution to \eqref{flow lb} and proves i).

iii) Observe first that for each $v \geq 0 $ and each $z\in \mathbb{D}(\R_+,\R_+)$  the same arguments of \cite[Theorem 2.5]{DL} can be used to prove strong existence and  pathwise uniqueness for the  stochastic differential equation
 \begin{equation}\label{flow lb nonhomog}
  \begin{split}
   Z^z _{t}(v)=\ &  v- z_0 - \alpha\int_{0}^{t}Z^z_{s}(v)\diffd s+\sigma\int_{0}^{t}\int_0^{Z^z_{s-}(v)}W(\diffd s,\diffd u)\\
                 & + \int_{0}^{t}\int_{0}^{Z^z_{s-}(v)}\int_{0}^{\infty}r\tilde{N}(\diffd s,\diffd \nu,\diffd r) - \int_{0}^{t} G_z(s,Z^z_s(v))    \diffd s,  \quad   t\ge 0. \\
  \end{split}
 \end{equation}
In particular, the solution $  Z^z _{t}(v)$ is adapted to the filtration generated by $W$ and $N$. 
We claim  that,  for each  fixed   $v \geq u \geq 0$  the process $\Upsilon_t := Z_t(v) - Z_t(u)$,  $t \geq 0$ satisfies a similar equation,  but with a randomized $z$. Indeed,  from \eqref{flow lb}  we have
 \begin{equation}\label{upsilon}
  \begin{split}   
    \Upsilon_t =\ & v-u - \alpha \int_0^t \Upsilon_s \diffd s + \sigma \int_0^t\int_0^{\Upsilon_{s-}} W' (\diffd s,\diffd w)  \\ 
    &     +  \int_0^t\int_0^{\Upsilon_{s-}}\int_0^{\infty} r\tilde{N}' (\diffd s,\diffd \nu, \diffd r) 
                    - \int_0^t G_{Z_{\cdot} (u)}( s,\Upsilon_s) \diffd s, \quad   t\ge 0 \, , \\
  \end{split}
 \end{equation} where now 
 \[
 W' (\diffd s,\diffd w) = W(\diffd s, \diffd w + Z_{s-}(u) )
 \] 
 is an orthogonal martingale measure (in the sense of Walsh \cite{Walsh})  and 
 \[
 N' (\diffd s,\diffd \nu, \diffd r)= N(\diffd s,\diffd \nu + Z_{s-}(u), \diffd r)
 \]
is a random point measure. Using standard properties of integration with respect to $W$ and $N$,  we can  then check that for  arbitrary  Borel sets $A_1,\dots A_n$ in $\R_+$ and $B_1,\dots,B_m$ in $\R_+^2$, disjoint and  of finite Lebesgue measure,   the processes  $M^i_t:=W'([0,t]\times A_i)$  and $N^k_t:=N'([0,t]\times B_k)$ satisfy the  assumptions of Lemma \ref{weirdwnandpp} with the filtration    given by $ {\cal S}_t:=\sigma\left( W([0,s],\diffd w), N([0,s],\diffd r,\diffd \nu),  s\leq t \right)$ and with constants  $a_i:=\int_{A_i}\diffd w  <\infty $ and  $b_k=\int\int_{B_k} \Pi(r)\diffd r \diffd \nu<\infty$. This implies that  $W'$ and $N'$ are respectively  a  Gaussian white noise process with intensity $\diffd s \otimes \diffd w$ and  a Poisson random measure with intensity $\diffd s \otimes \diffd \nu \otimes \Pi(\diffd r)$, both with respect to $ ({\cal S}_t)_{t\geq 0}$, and independent of each other. 

\smallskip

  In order to prove iii),  it is thus enough to show that the pair $(W',N')$ is independent from  ${\cal Z}^u:=\sigma(\left(  Z_t(x) : t \geq 0,0\leq  x\leq u \right))$ (indeed, we can then apply  pathwise uniqueness for  \eqref{flow lb nonhomog} conditionally on ${\cal Z}^u$ to get that $\Upsilon$ is a measurable function of $(W',N',Z(u))$, and immediately  deduce the desired property). 
  To that end, we enlarge the probability space with an independent  pair $(\hat{W},\hat{N})$  equal in law to $(W,N)$ and  define,  given $0\leq u_1\leq \dots\leq u_p\leq u$, the  processes
\[
\overline{W}(\diffd s, \diffd \nu): = \I_{[0,\tilde{Z}_{s-}]}(\nu) W(\diffd s, \diffd \nu)+\I_{(\tilde{Z}_{s-},\infty)}(\nu) \hat{W}(\diffd s, \diffd \nu - \tilde{Z}_{s-})\, \mbox{ and }$$
$$\overline{N}(\diffd s, \diffd \nu, \diffd r): = \I_{[0,\tilde{Z}_{s-}]}(\nu) N(\diffd s, \diffd \nu, \diffd r)+\I_{(\tilde{Z}_{s-},\infty)}(\nu) \hat{N}(\diffd s, \diffd \nu - \tilde{Z}_{s-}, \diffd r)\, 
\]
where $\tilde{Z}_s=\max\{Z_{s}(u_l): l=1,...,p\}$. Call   $({\cal S}^*_t)$ the   filtration generated by $(W,\hat{W},N,\hat{N})$ and, given  families of sets $(\overline{A}_i)_{i\leq n'}$ and  $(\overline{B}_k)_{k\leq m'}$ with similar properties as  $(A_i)_{i\leq n}$ and $(B_k)_{k\leq m}$ considered above, define   $\overline{M}^i_t:=\overline{W}([0,t]\times \overline{A}_i)$  and  $\overline{N}^i_t:=\overline{N}([0,t]\times \overline{B}_i)$. Thanks to Lemma   \ref{weirdwnandpp}  now applied to the families $\left((M^i)_{i\leq n}, (\overline{M}^i)_{i\leq n'}\right)$ and $\left((N^k)_{k\leq m}, (\overline{N}^k)_{k\leq m'}\right)$ and the filtration  $({\cal S}_t^*)$,  we readily get that the process $(\overline{W},\overline{N})$ has the same law as $(W,N)$  and  moreover, that $(W',N')$ and $(\overline{W},\overline{N})$ are independent.
Since the processes $Z(u_l),  l=1,...,p$  are adapted to $({\cal S}_t^*)$ and  they also  solve  a system of  SDEs as in \eqref{flow lb} but driven by  $(\overline{W},\overline{N})$, strong existence and pathwise uniqueness for such  equations imply they are measurable functions of $(\overline{W},\overline{N})$. Hence, their independence from  $(W',N')$ follows. 

% Since for $0\leq u_1\leq \dots\leq u_n\leq u$  the  cross-variations of the semimartingales  $Z_s(u_1),\dots, Z_s(u_n)$  with   the processes $W^1_t,\dots, W^m_t,N^1_t,\dots,N^m_t$ are null, we get using   It\^o's formula   that the   Markov  process $\bigg( (W^1_t,\dots, W^m_t,N^1_t,\dots,N^m_t), (Z_t(u_1),\dots, Z_t(u_n))\bigg)_{t\geq 0}$  has a generator ${\cal L}$   of the form 
% \begin{align*}
%  {\cal L}f( w_1,\dots,w_m,y_1,\dots,y_m,x_1,\dots, x_m)=\ & {\cal L}'_{ (w,y)}  f( w_1,\dots,w_m,y_1,\dots,y_m,x_1,\dots, x_m)\\
%                                               & + {\cal L}''_{ x}  f( w_1,\dots,w_m,y_1,\dots,y_m,x_1,\dots, x_m)
% \end{align*} 
% where  the first  and second operators in the r.h.s.  stand for the generator of $(W^1_t,\dots, W^m_t,N^1_t,\dots,N^m_t)$ acting on the variable  $(w,y)=( w_1,\dots,w_m,y_1,\dots,y_m)$ and the generator of $(Z_t(u_1),\dots, Z_t(u_n))$  acting on the variable $x=(x_1,\dots,x_n)$, respectively. That is, the generator of two independent processes with generators   $ {\cal L}'_{ (w,y)} $ and $ {\cal L}''_{ x} $. Since uniqueness  for the  martingale problem corresponding to ${\cal L}$ holds,  independence of $(W',N')$ and ${\cal Z}^u$  follow.

\smallskip 

 ii) 
 For each  $t \geq 0$, the fact that  $v \mapsto Z_t(v)$ is increasing can be deduced from the comparison property stated in \cite[Theorem 2.3]{DL}. 
 Moreover, it is easy to show using similar arguments as in the proof of \cite[Proposition 3.4]{DL}, that there is a locally bounded non-negative function $t \mapsto  C(t)$ on $[0, \infty)$ such  that
 \begin{equation}\label{boundcadlagmodif}
  \E \left\{  \sup\limits_{0 \leq s \leq t} |Z_s (v) - Z_s(u) | \right\} 
  \leq C(t)\left\{ (v - u) +  \sqrt{v -u} \right\}
 \end{equation}
 for $v \geq u \geq 0$. 
 Alternatively, one can use a simple extension of  \cite[Theorem 2.3]{DL}  to stochastic flows with  time inhomogeneous drift terms in order to show that, for each $v\geq 0$, the solution $(Z^z_t(v):t\geq 0)$  of \eqref{flow lb nonhomog} when $z_0=u$ is a.s.\ dominated on $[0,\infty)$  by  the solution $(Y_t(v-u):t\geq 0)$ of \eqref{eqflow csbp}, then deduce by conditioning first on  $\left(  Z_t(x) : t \geq 0,0\leq  x\leq u \right)$ (and using part iii)) that 
 \[  
 \E \left\{  \sup\limits_{0 \leq s \leq t} |Z_s (v) - Z_s(u) | \right\} \leq   \E \left\{  \sup\limits_{0 \leq s \leq t} |Y'_s (v-u) | \right\}, 
 \] 
 with $Y'$ defined in terms of $(W',N')$ in the same way as $Y$ was in terms of $(W,N)$, 
 and   conclude  the bound \eqref{boundcadlagmodif} from  \cite[Proposition 3.4]{DL}. 
 Thanks to \eqref{boundcadlagmodif}  and  \cite[Lemma 3.5]{DL}, we  can   follow the arguments  in the proof of  \cite[Theorem 3.6]{DL}  to deduce that the path-valued process $(Z (v) : v \geq 0)$ has a c\`adl\`ag modification, using when necessary  the Markov property stated  in  iii)  instead  of the independent increments property of  the process $(Y (v) : v \geq 0)$.
 
  iv) This part follows by similar arguments as in Theorem 3.7 of \cite{DL}.  
\end{proof}

\medskip

\begin{remark}  Recall that in \cite{L1}, logistic branching processes were constructed by a  Lamperti time-change of certain  L\'evy-driven Ornstein-Ulhenbeck processes. 
A similar approach also works here: for a fixed initial population $v>0$ it is possible to construct  the  process $Z_{\cdot}(v)$    by  a Lamperti transform of the solution $U$ to the  L\'evy-driven SDE
\begin{equation}\label{OHish}
 \diffd U_t= \diffd X_t  - \frac{G(|U_t| )}{U_t} \diffd t , \quad t\geq 0, \quad U_0=v. 
\end{equation}
A precise statement  (not required in the sequel)  and a sketch of  its proof are given in Appendix \ref{Lampertitransform}.
\end{remark}

%%%%%%%%%%%%%%%%%%%%%%%%%%%%%%%%%%%%
%%%%%%%%%%%%%%%%%%%%%%%%%%%%%%%%%%%%
%%%%%%%%%%%%%%%%%%%%%%%%%%%%%%%%%%%%

\section{Competitive pruning}\label{explosnakeprun}

\subsection{A Poisson L\'evy-snake}\label{poissonsnake}

%%%%%%%%%%%%%%%%%%%%%%%%%%%%%%%%%%%%
%%%%%%%%%%%%%%%%%%%%%%%%%%%%%%%%%%%%
%%%%%%%%%%%%%%%%%%%%%%%%%%%%%%%%%%%%

In order to  settle some basic properties  of the exploration process marked at height-time adapted rate $((\rho_s , m^\intens_s)  : s \geq 0)$ (cf. Definition\  \ref{def of prune with adapted process}), we need to first   make precise the way the marked exploration process $((\rho_s , \N_s)  : s \geq 0)$ of Definition  \ref{def of 2d marked explo} enters into the general framework of  L\'evy-snake processes constructed in  \cite[Chapter 4]{DLG}. 

\smallskip 

Recall that, given (the law of) a Markov process $(\xi_r)_{r\geq 0}$ with c\`adl\`ag paths and values in a Polish space $E$, which is furthermore continuous in probability, the snake with path-space component $\xi$ is a Markov process $((\rho_t, \xi^{(t)}):t\geq 0)$ with $\rho$  the exploration process  and such that,  conditionally on $\rho$,  for each $t\geq 0, \, \xi^{(t)}=(\xi^{(t)}_r ,r\le H_t)$ is a path of $\xi$ killed at $H_t$ and  for $t<s$, the  Markovian paths $\xi^{(t)}$ and $\xi^{(s)}$  are related by the snake property.  That is,   $\xi^{(t)}$ and $\xi^{(s)}$ are a.s. equal on $[0,H_{t,s}]$  and,   conditionally on  $\rho$ and one their  value  at time $H_{t,s}$,  their evolutions thereafter are independent
 (see  \cite[Proposition 4.1.1]{DLG}). 

 \smallskip 
 
For $j=1,2$, let us  respectively denote  by  $\M(\R^j_+)$ and  $\M_{at}(\R_+^j)$ the space of Borel measures in $\R_+^j$ and the corresponding subspace of atomic Borel measures with unit mass atoms. 
We endow the space  $\M(\R_+)$  with a complete metric inducing the topology of vague convergence (equivalent to weak convergence of the  restrictions to  every compact subset of $\R_+$). 
Given 
$\pi$ a Poisson point measure on $\R_+^2$ with intensity $\diffd x \otimes \diffd y$, seen as a random  element of $\M_{at}(\R_+^2)$, we  consider the  increasing $\M_{at}(\R_+)$-valued additive Markov process  $(\xi_r)_{r\geq 0}$  (issued from $0$) defined by 
\[
\xi_r(\cdot) = \pi ([0,r]\times \cdot), \, r\geq 0.
\]
Then,  $\xi$ clearly has the above dicussed properties and we can therefore construct the associated L\'evy snake $((\rho_t, \xi^{(t)}):t\geq 0)$.  The  marked exploration process $((\rho_t , \N_t)  : t \geq 0)$ is then defined setting for each  $t\geq 0$: 
\[
\N_t((a,b] \times A) = \xi^{(t)}_b(A) - \xi^{(t)}_a(A), \forall \, 0\leq a< b <H_t \mbox{ and } A\subseteq \R_+  \mbox{ Borel} 
\]
and  $\N_t(\{0\}\times \cdot)=0$. This uniquely determines  a random element or point process $\N_t\in \M_{at}(\R_+^2)$. It is  easy to see that  $((\rho_t,\N_t):t\geq 0)$ thus defined fulfills the distributional conditions   in  Definition  \ref{def of 2d marked explo}. 

\smallskip

The space ${\cal V}$   introduced in Section \ref{Pruning of Levy trees} is given a topology as follows. 
The set  $ \M^0_f(\R_+)$ is seen as a subset of the space $ \M_f(\R_+)$ of finite Borel measures, endowed with the weak topology and  an associated complete distance $D$. 
Consistently with the previous construction,  given a  pair $(\mu,\eta) \in {\cal V}$ we   identify $\eta$ with the  increasing and  killed  c\`adl\`ag  path $[0,H(\mu))\to \M(\R_+)$ given by 
\[
 r \mapsto  \eta([0,r]\times \cdot \, ).
\]
 Finally, as in  \cite[Section  4.1]{DLG},  we endow the space $ {\cal V}$  with the complete distance
\[
    d \left( (\mu, \eta), (\mu' , \eta' )\right) = D(\mu,\mu') + \int_0^{H(\mu) \wedge H(\mu')}\left( d_u (\eta_{(u)} , \eta'_{(u)} )\wedge  1  \right)  \diffd u + |H(\mu) - H(\mu')|\, +  d_0 (\eta_{(0)} , \eta'_{(0)} ), 
\]
for  all  $(\mu,\eta), (\mu' , \eta' ) \in {\cal V}$, where for each $u\in [0,H(\mu)) $, $d_u$ is  the Skorohod metric on  $\mathbb{D}([0, u],\M(\R_+))$ and $\eta_{(u)}$ is the restriction  to $[0,u]$  of the increasing path associated with $\eta$   \footnote{Notice that $d_0$ corresponds to the metric on $\M(\R_+)$.}.  The fact that  $((\rho_s , \N_s)  : s \geq 0)$  is a c\`adl\`ag  strong    Markov process  in  ${\cal V}$  is then granted by the general results in \cite[Chapter 4]{DLG}. 

\subsection{Marking at height-time adapted rates}\label{interprun}

In all the sequel, we will write $(\F_t)_{t \geq 0}$ for  the right continuous complete  filtration  generated by $(\rho_t , \N_t ) : t \geq 0)$,  and $(\F_t^{\rho})_{t \geq 0}$ for  the  one generated by  $(\rho_t   : t \geq 0)$.     
For each $t \geq 0$, we moreover  introduce $(\gG^{(t)}_r)_{r\geq 0}$  the right continuous completion of  the filtration given by 
\begin{equation}\label{Gtr}
        \sigma \left( \F_t^{\rho}, \left\{ \N_s|_{[0,r]\times \R_+}, s\leq t\right\}\right),\quad r\geq 0. 
\end{equation}
In particular, for each $r,t\ge 0$ we have
\[
\F_t^{\rho} \subseteq \gG^{(t)}_r \subseteq \F_t.
\]

\begin{remark}
Heuristically, $\gG^{(t)}_r$ contains all the information to the left of $t$ for $\rho $ and to the  left of $t$ and below $r$ for $\mathcal{N}.$   Notice that  conditionally on $ \gG_0^{(t)}=\F_t^{\rho}$, $\N_t$ is a Poisson process in the filtration $( \sigma \left( \F_t^{\rho}, \left\{ \N_t|_{[0,r]\times \R_+} \right\}\right))_{  r\geq 0}$. Hence,    $(\gG^{(t)}_r)_{r\geq 0}$  is obtained by progressively enlarging that  filtration with   the sigma-fields 
\[
\sigma \left( \F_t^{\rho}, \left\{ \N_s |_{[H_{s,t},r]\times \R_+} \I_{\{ H_{s,t}<r\}}: s\leq t\right\}\right) ,\quad r\geq 0
\]
which are  independent from it conditionally on  $ \gG_0^{(t)}$, by   the snake property of $((\rho_s , \N_s)  : s \geq 0)$ (cf. they contain  information of marks which are not in the lineage of $t$).   
\end{remark}

 The following property is then easily deduced:

\begin{lemma}\label{NPoissG} 
For each $t\geq 0$, conditionally on $ \gG_0^{(t)}$  the point process  $\N_t$  on $\R_+^2$ is  Poisson   of intensity $\I_{ [0,H_t)}(r)\diffd r \otimes \diffd \nu$, with respect to the filtration $(\gG^{(t)}_r)_{r\geq 0}$.
\end{lemma}

The use of Poisson calculus for $\N_t$ with respect to the filtration $( \gG_r^{(t)})_{r\geq 0}$  is thus possible but some care is needed regarding measurability issues. 
We  need to introduce

\begin{definition}\label{predictablet}
For each $t\geq 0$, ${\cal P}red^{(t)}$ denotes the sub sigma-field of  ${\cal B}(\R_+)\otimes  \F_{t}$  generated by $ (\gG^{(t)}_{r})_{r\geq 0}$-adapted  processes $(a (r) :  r\geq 0)$ which have continuous trajectories. 
\end{definition}

\begin{definition}\label{snakepredictable}
We denote by ${\cal P}red(\N)$ the sub sigma-field of ${\cal B}(\R_+)\otimes {\cal B}(\R_+)\otimes   \F_{\infty}$  generated by  processes $(b (s,r) : s\geq 0,  r\geq 0)$ such that 
\begin{enumerate}
 \item for each $(s,r)\in \R_+^2$, $b(s,r)$ is $\gG^{(s)}_{r}$  measurable, 
 \item the process $(s,r)\mapsto b(s,r)$  is continuous. 
\end{enumerate}
\end{definition}

\begin{remark}  A monotone class argument shows that if  $(b (s,r) : s\geq 0,  r\geq 0)$   is  ${\cal P}red(\N)$- measurable, then $(b (t,r) :   r\geq 0)$  is ${\cal P}red^{(t)}$-measurable  for each  $t\geq 0$   (the converse is   not true). 
\end{remark}

Recall that, given an adapted intensity process $\intens $,   in  Definition  \ref{def of prune with adapted process}  we introduced  for each $s\geq 0$  the  point process $
m^\intens_s$  in $ \M_{at}(\R_+)$ defined by 
\[
m^\intens_s(A) =  \N_s ( \{(h,u) :  u \le \intens(h,s) \text{ and } h \in A \}  ), \forall \mbox{ real Borel set } A.
\]
Thus,    $((\rho_s , m^\intens_s)  : s \geq 0)$ takes its values in the space of  ``marked finite measures''    
\[
\mathbb{S}:=\left\{  (\mu, w ) \in \M^0_f(\R_+)\times \M_{at}(\R_+) :  \, \mbox{ supp} (w)\subset  [0,H(\mu))\right\}
\] 
considered   in  \cite{ADV}. 
In a similar way as for elements of  ${\cal V}$, for   $(\mu,w)\in \mathbb{S}$  we now  identify  the measure $w$   with the increasing  element of $\D(\R_+,\R)$  given by the  path $ v \mapsto w([0,v])$  killed at  $H(\mu)$ (i.e. the cumulative distribution of $w$). Following again \cite[Chapter 4]{DLG},  $\mathbb{S}$ is endowed with the complete  metric 
\begin{equation}\label{distm}
\begin{split}
 \hat{d}((\mu,w),(\mu',w')):=  D(\mu,&\mu') +  \int_0^{H(\mu) \wedge H(\mu')}\left(\hat{d}_u(w_{(u)},w'_{(u)})\wedge 1 \right) \diffd u  \\
 & + |H(\mu)-H(\mu')|  +\hat{d}_0(w_{(0)},w'_{(0)}),  
 \end{split}
\end{equation}
with  $\hat{d}_u$  the Skorohod metric on $\mathbb{D}([0, u], \R_+)$ and  $w_{(u)}$  the restriction to $[0,u]$   of $w$. 
 
% \begin{remark} With a slight abuse of notation, we will accordingly  denote hereafter by
%$$a\mapsto  m_t^\intens(a) : =m_t^\intens([0,a])  \, ,   a < H_t $$
% the cumulative distribution function of $m_t^\theta$. Note that the  statement  ``$  m_t^\intens=0$'' has the same meaning whether $  m_t^\intens$ is seen as an element of  $ \M_{at}(\R_+)  $ or of   $\D(\R_+,\R)$.  Observe also that $ m_t^\intens(0)=0$ always. \end{remark}

 Next result gathers basic properties of the process $((\rho_t , m^\intens_t)  : t \geq 0)$ needed in the sequel: 
\begin{lemma}\label{lem: prop mark adapt intens}
Let  $((\rho_t , m^\intens_t)  : t \geq 0)$  be a  marked exploration process with adapted intensity $\intens$. 
\begin{itemize} 
 \item[i)]    For each  $t > 0$, the increasing  c\`adl\`ag  process  $h \mapsto m_t^\intens ([0,h)) $ is  $(\gG_h^{(t)})_{h\geq 0}$-adapted  and its   predictable compensator  in that filtration is  $h \mapsto \int_0^{h\wedge H_t} \intens(r,t)\diffd r$. 
 \item[ii)]  For each pair of  time instants $s$ and $s'$ one has, almost surely, 
 \[
 m^\intens_{s'} (\diffd r) \I_{\{ r \leq H_{s,s'} \}} = m^\intens_s (\diffd r) \I_{\{ r \leq H_{s,s'} \}}.
 \]
 \item[iii)]  $((\rho_t , m^\intens_t)  : t \geq 0)$ is an $(\F_t)$-adapted and  c\`adl\`ag  process with values in $(\mathbb{S},\hat{d})$.  
 \item[iv)]  A marked exploration processes with adapted intensity $((\rho_t , m^{\intens'}_t)  : t \geq 0)$    is indistinguishable from  $((\rho_t , m^\intens_t)  : t \geq 0)$  if and only if    a.s.   $\intens (t,h)=\intens' (t,h)$   a.e. with respect to $\diffd t \otimes \I_{[0,H_t)}(h) \diffd h$.
\end{itemize} 
\end{lemma}

We will refer to property ii) above as {\it the semi-snake property for marked exploration processes}. 

\begin{remark} Notice that  independence of $m^\intens_{s'} (\diffd r) \I_{\{ H_{s,s'}<r \leq H_{s'} \}} $ and $m^\intens_s (\diffd r) \I_{\{H_{s,s'}<r \leq H_{s} \}}$ for $s<s'$ (hence the full ``snake-property'') will in general fail to hold. 
Moreover, $((\rho_t , m^\intens_t)  : t \geq 0)$ is not Markovian in general  (an exception is the case of $\intens$ constant, where the marked exploration process is itself a snake process, with a standard Poisson process as path-space component). 
\end{remark}

\begin{remark}\label{equalintens}
Observe that for each $t\geq 0$ the process $h \mapsto  \intens(h,t)$   is  determined from $ m^{\intens}_t$ only $\diffd h$ a.e. (as the Lebesgue derivative of its compensator). 
In view of  this and of part iv) of Lemma \ref{lem: prop mark adapt intens}, two adapted intensities $\intens$ and $\intens'$ are identified  if any of the   two equivalent properties therein hold.
\end{remark}

\begin{proof}[Proof of Lemma \ref{lem: prop mark adapt intens}]
 Part i) is straightforward using  the compensation formula for ${\cal N}_t$ with respect to the filtration $(\gG_h^{(t)})_{h\geq 0}$,  conditionally on ${\cal F}^{\rho}_t=\gG_0^{(t)}$.  

 For part ii),  taking conditional expectation given $\gG_0^{(s)}$ in the inequality  
 \[
 |m_{s}^\intens ([0,h))- m_{s'}^\intens ([0,h)) | \I_{\{ h \leq H_{s,s'} \}} \leq \int_0^h \int_{\intens(s,r)\wedge \intens(s',r)}^{\intens(s,r)\vee \intens(s',r)} \mathcal{N}_s(\diffd \nu,\diffd  r)   \I_{\{ h \leq H_{s,s'} \}},
 \]  
 we deduce from  $\intens(s,\cdot)$ and  $\intens(s',\cdot)$  being ${\cal P}red^{(s)}$-measurable and from   the compensation formula  with respect to   $( \gG_t^{(t)})_{r\geq 0}$ that,  for each $h\geq 0$,  $m_{s}^\intens ([0,h))= m_{s'}^\intens ([0,h)) $ a.s. on $\{ h \leq H_{s,s'} \}$. 
 This yields $m_{s}^\intens ([0,r))= m_{s'}^\intens ([0,r)) $   for all  rational $r\leq h$  a.s. on $\{ h \leq H_{s,s'} \}$, and  then  $m_{s}^\intens= m_{s'}^\intens $ as measures on that event by  a.s. left continuity of $r\mapsto m_u^\intens ([0,r)) $, for $u=s,s'$. 

 Adaptedness  in  part iii) is clear. As for path regularity of $((\rho_t , m^\intens_t)  : t \geq 0)$,  from the definition  \eqref{distm}  and the facts that  $(\rho_s: s \geq 0)$ is weakly  c\`adl\`ag and $(H_s: s \geq 0)$ is  continuous, we easily get right continuity.  
 Existence of the limit in $(\mathbb{S},\hat{d})$ of $(\rho_{s_n}, m^{\intens}_{s_n})$   when $s_n\nearrow t$ is slightly more subtle. 
 Indeed, by path properties of $\rho$ and $H$,  $\rho_{s_n}$ clearly converges to a limit $\rho_{t-}$ such that  $H(\rho_{t-})=H_t$, while  convergence of  $ m^{\intens}_{s_n}$ to a limit  $ m^{\intens}_{t-}$   follows from completeness of the space of killed  c\`adl\`ag  paths in $\R_+$  under the metric
 \[
  (w,w')\mapsto  \int_0^{\zeta(w) \wedge\zeta(w')}\left(\hat{d}_u(w_{(u)},w'_{(u)})\wedge 1 \right) \diffd u  + \hat{d}_0(w_{(0)},w'_{(0)})+ |\zeta(w)- \zeta(w')|, 
  \]
 where $\zeta(w)$ is the lifetime of $w$  (see \cite[Section 4.1.1]{DLG} for the general property). Last, $\mbox{supp } m^{\intens}_{t-}\subseteq [0,H(\rho_{t-}))$  follows from the Portmanteau theorem and implies that  $(\rho_{t-}, m^{\intens}_{t-})\in {\cal V}$. 
   
 For the direct implication in  part iv),  we  first obtain from  indistinguishability and part i) the fact that, a.s. for every $t\geq 0$, the processes $h \mapsto \int_0^{h	\wedge  H_t} \intens(r,t)\diffd r$ and $h \mapsto \int_0^{h	\wedge  H_t} \intens'(r,t)\diffd r$ are equal. Then, we conclude by Lebesgue derivation.  
 In the converse implication, starting from the inequality
 \[
 |m_{s}^\intens ([0,h))- m_{s}^{\intens'} ([0,h)) |  \leq \int_0^h \int_{\intens(s,r)\wedge \intens'(s,r)}^{\intens(s,r)\vee \intens'(s,r)} \mathcal{N}_s(\diffd \nu,\diffd r) \, ,  
 \] 
 we obtain  (with similar arguments as in proof of point ii) above) the equality  $m_{s}^\intens =m_{s}^{\intens'}$ a.s. for all $s$ in a set of full Lebesgue measure. 
 Indistinguishability follows then from the path regularity stated in  iii). \end{proof}

Finally, we  establish a useful  regularity  property  regarding  the process  of marks in each  lineage: 
 
\begin{lemma}\label{total marks lsc}
Let  $((\rho_t , m^\intens_t)  : t \geq 0)$  be a  marked exploration process with adapted intensity $\intens$.  Then, the function $s\mapsto m^{\intens}_s([0,H_s))$ is a.s.\ lower semi-continuous. 
As a consequence,  a.s.  the function  $s\mapsto \I_{\{ m^{\intens}_s=0\}}= \I_{\{ m^{\intens}_s([0,H_s))=0\}}$ has at most countably many discontinuity points.
\end{lemma}
   
\begin{proof}   
 By the semi-snake property of  $((\rho_t , m^\intens_t)  : t \geq 0)$ and the continuity of $s\mapsto H_s$, it holds for each  $\varepsilon >0$ and every $t\geq 0$  that $m^{\intens}_t([0,H_t-\varepsilon))= m^{\intens}_s([0,H_t-\varepsilon)) \leq m^{\intens}_s([0,H_s) )$, for all $s$ close enough to $t$ so that $H_t-\varepsilon \leq H_{s,t}$. 
 Taking $  \liminf_{s\to t}$  and letting $\varepsilon\to 0$ gives us the first statement. Also, we  deduce that
 \begin{equation}\label{usc}
  \limsup_{s\to t} \I_{\{ m^{\intens}_s([0,H_s))=0 \}}\leq  \I_{\{  m^{\intens}_t([0,H_t))=0\}}
 \end{equation}
 for all $t\geq 0$,  implying that $s\mapsto \I_{\{ m^{\intens}_{s}([0,H_{s}))=0 \}} $ is continuous at points $t\geq 0$ such that  $m^{\intens}_t([0,H_t))>0$.  
 Now, if  $t\geq 0$ is such that $m^{\intens}_t([0,H_t))=0$, either we have  for any sequence $s_n\to t$ and  all  large enough  $n$ that  $\I_{\{ m^{\intens}_{s_n}([0,H_{s_n}))=0 \}}=1$, or we can construct  a sequence $\hat{s}_n\to t$ such that  $\I_{\{ m^{\intens}_{\hat{s}_n}([0,H_{\hat{s}_n}))=0 \}}=0$ for all $n\geq 0$. 
 In the first case,   $s\mapsto \I_{\{ m^{\intens}_{s}([0,H_{s}))=0 \}}$ is continuous in $t$. In the second one,  $t$ lies in the boundary of the  open set $\{s\geq 0:  m^{\intens}_s([0,H_s))>0 \}$. 
 The latter being a disjoint countable union of  open intervals, the conclusion  follows. 
\end{proof}

\subsection{An operator on adapted intensities}  
The following well known approximations of local time at a given level $a\geq 0$,  by  the limiting  normalized time  that  the height process spends  in a small neighborhood,   will be very useful in the sequel: 
\begin{equation}\label{aprox1}
 \lim\limits_{\epsilon \rightarrow 0}\sup\limits_{a \geq 0} \E\left[\sup\limits_{s\leq t}\left| \epsilon^{-1}\int_0^s \I_{\{a<H_r\leq a+\epsilon\}}\diffd r - L^a_s\right|\right]=0 \, , \
\end{equation}
\begin{equation}\label{aprox2}
 \lim\limits_{\epsilon \rightarrow 0}\sup\limits_{a \geq \epsilon} \E\left[\sup\limits_{s\leq t}\left| \epsilon^{-1}\int_0^s \I_{\{a-\epsilon<H_r\leq a\}}\diffd r - L^a_s\right|\right]=0 
\end{equation}
(see \cite[Proposition 1.3.3]{DLG}). 
We will also need the  occupation times formula therefrom deduced or, better, its immediate  extension to time-height functions: 
\begin{equation}\label{occupation formula generalized}
           \int_0^t \varphi(r, H_r)\diffd r =  \int_0^{\infty} \int_0^t \varphi(s,a)  \diffd L^a_s da
\end{equation}
a.s. for  every $t\geq 0$ and all   bounded  measurable functions $\varphi:\R_+^2 \to \R$. 
 
\smallskip 
           
Next result will prove  the first statement of Theorem \ref{main thm 1}.  
\begin{proposition}[{\bf Pruned local times yield adapted intensities}]\label{lem: loc time of pruned is adapted}
Let $f:\R_+\mapsto \R_+$ be  a locally bounded measurable function  and $\intens$ an adapted intensity process. Then $(f(L_t^r(m^{\intens})):  t\geq 0, r\geq 0)$ is also an adapted intensity process.
\end{proposition}
  
\begin{proof}
 Since $L_t^r(m^{\intens})\leq L_t^r\leq L^r_{T_{x}}$ on the event $\{t\leq T_x\}$ and, for each $x\geq 0$ the process $(L^r_{T_{x}}:r\geq 0)$ is a  (sub)critical CSBP issued from $x$, the fact that $T_x\nearrow \infty$ when $x\to \infty$ implies that a.s.\ for each $t\geq 0$ the function  $r\mapsto L_t^r(m^{\intens})$   is  a.s.  locally bounded. 
 This  ensures that  property  2 in Definition \ref{adapintensity} is satisfied. 
 
 In order to check the first point of Definition \ref{adapintensity}, observe first that for any  $\varepsilon ,t, r \geq 0$ 
 \[
 \varepsilon^{-1} \int_0^t \I_{\{ -\varepsilon < H_u- r\leq 0     \}} \diffd u  =   
 \lim_{k\to \infty} \varepsilon^{-1}  \int_0^{\infty}  (1-\phi_{ k}(u-t+k^{-1} ))  \left(\phi_{ k}( H_u-r+\varepsilon)-    \phi_{ k}(H_u-r) \right) \, \diffd u  \, ,
 \]
 almost surely,  for $ (\phi_{ k}:\R\to [0,1])_{k\geq 1}$  continuous functions  vanishing on $(-\infty,0]$ and equal to $1$ on $(1/k,\infty)$. 
 From this and the  fact that, for all $t,r\geq0$ and $k\in \NN\backslash \{0\}$,  $(H_u\I_{u\leq t-k^{-1}} : u\geq 0)$  is  measurable with respect to ${\cal B}(\R_+)\otimes {\cal G}^{(t)}_0\subset {\cal B}(\R_+)\otimes{\cal G}^{(t)}_r$,  we deduce using \eqref{aprox2} that $(L_t^r:  t\geq 0, r\geq 0)$ is  ${\cal P}red(\N)$-measurable. Since  for each $s\geq 0$ the process $r\mapsto m^{\intens}_s([0,r))$ is c\`agl\`ad and $(\gG_r^{(s)})_{r\geq 0}$ adapted,  we next observe that 
 \begin{linenomath}
  \begin{equation*}
  \begin{split}
   (t,r)\longmapsto  & \sum_{k=0}^{\infty} \left( L^r_{  (k+1)2^{-n} \wedge t} -  L^r_{  k 2^{-n}\wedge t}\right) \I_{\{ m^{\intens}_{k 2^{-n}}([0,r))=0\} } \\
   & = \int_0^t \sum_{k=0}^{\infty} \I_{\{ k 2^{-n}< s\leq (k+1)2^{-n},\, m^{\intens}_{k 2^{-n}}([0,r))=0\} }\diffd L_s^r \, 
  \end{split}
 \end{equation*}
 \end{linenomath}
 is   for each $n\in \NN$ a ${\cal P}red(\N)$--measurable  process too. 
 Thanks to the  a.s. equality $\diffd  L_s^r=\I_{\{H_s=r\}} \diffd L_s^r$ for each  $r\geq 0$, in order to conclude that $(L_t^r(m^{\intens}): r\geq 0, t\geq 0)$ is ${\cal P}red(\N)$-measurable  it is enough to check the convergence 
 \begin{equation}\label{aprox2_m}
  \lim\limits_{n\to \infty} \int_0^t \sum_{k=0}^{\infty} \I_{\{ k 2^{-n}< s\leq (k+1)2^{-n}\}  }
  \I_{\{ m^{\intens}_{k 2^{-n}}([0,r))=0\} }\diffd L_s^r =  \int_0^t   \I_{\{m^{\intens}_s([0,r))=0\}} \diffd L_s^r 
 \end{equation}
 for each $t\geq 0$. 
 By Lemma \ref{total marks lsc} and  continuity of  the  measure $\diffd L_s^r$,  the function  $s\mapsto  \I_{\{m^{\intens}_s=0\}}$ is  continuous $\diffd L^r_s$ almost everywhere. 
 Hence, the  functions  $s\mapsto  \sum_{k=0}^{\infty}  \I_{\{ k 2^{-n}< s\leq (k+1)2^{-n}\}} \I_{\{ m^{\intens}_{k 2^{-n}}=0\} }$ converge $\diffd L^r_s$  a.e. as $n\to \infty$  to $s\mapsto  \I_{\{m^{\intens}_s=0\}}$. 
 Thus, dominated convergence yields \eqref{aprox2_m} and the desired measurability property.
 
 Let us finally check that property 3 in Definition \ref{adapintensity} holds for  $(f(L_t^r(m^{\intens})):t\geq 0, r\geq 0)$. 
 Since for all  $s'\geq s \ge 0$ and every  $r,\varepsilon>0$ we trivially have  $\E\left[  \int_s^{s'} \I_{\{r-\varepsilon<H_u\leq r\}}\diffd u \I_{\{ r < H_{s,s'}\}} \right]=0$, the approximation  \eqref{aprox2} implies that $L_s^r=L_{s'}^r$ for each $r\geq 0$, a.s. on the event $\{ r < H_{s,s'}\}$, whence $\int_0^{\infty}\E (|L_s^r-L_{s'}^r|\I_{\{ r < H_{s,s'}\}})\diffd r=0$. 
 Fubini's theorem then ensures that   $(L_t^r:t\geq 0, r\geq 0)$  satisfies the semi-snake property for adapted intensities, and the conclusion follows using the  approximation \eqref{aprox2_m} and the semi-snake property of  $(s,r)\mapsto m^{\intens}_s([0,r))$.  
\end{proof}

\medskip
    
We now establish a continuity-type estimate for the operator $F$ under condition  {\bf (H)}, that will be useful at several points in the sequel. 
Introduce  for each $M>0$ the $({\cal F}^{\rho}_t)_{t\geq 0}$-stopping time
\begin{equation}\label{TM}
 T^M:= \inf \{  t\ge 0: \sup_{h\ge 0} L^h_t \ge M \}
\end{equation}
and let $c(M)\geq 0$  be a finite constant such that
\begin{equation}\label{c(M)}
|g(x)-g(y)|\leq c(M) |x-y| \quad \forall \, 0\leq x,y\leq M \, ,
\end{equation}
 that is,  a Lipschitz constant of $g$ on $[0,M]$.   We have 
\begin{lemma}\label{lemma pre-Gronwall}
Assume condition {\bf (H)}  holds and let $\intens_1$, $\intens_2$  be two adapted intensity processes. Define for each $t\geq 0$, 
\[
\Delta_t := \int_0^{H_t} | \intens_1(t,h)-\intens_2(t,h)| \diffd h
\] 
and 
\[
\Delta'_t :=\int_0^{H_t} | F(\intens_1)(t,h)-F(\intens_2)(t,h)| \diffd h \, ,
\]
Then, for any $M,a\geq 0$ and  every  $(\F_t^{\rho})$-stopping  time $\tau$ such that $\tau\leq T^M$ a.s, we have
\begin{equation}\label{key lemma for convergence of picard scheme}
 \E[\I_{\{t\le \tau,\, H_t\leq a\}} \Delta'_t ] \le c(M) \int_0^t \E[\I_{\{s\le \tau, \, H_s\leq a\}} \Delta_s  ] \diffd s \quad \mbox{ for all }t\geq 0. 
\end{equation}
\end{lemma}

\begin{proof}
 Notice  first that  for each $s\geq 0$, $ \big|\I_{\{m_s^{\intens_1}=0\}} -\I_{\{m_s^{\intens_2} =0\}} \big|\leq  \big|m_s^{\intens_1}([0,H_s)) - m_s^{\intens_2} ([0,H_s))  \big|$, hence 
 \[
  \big| L_t^h(m^{\intens_1}) -L_t^h(m^{\intens_2}) \big| \le \int_0^t  \big|m_s^{\intens_1}([0,H_s)) - m_s^{\intens_2} ([0,H_s))  \big|  \diffd L_s^h
 \]
a.s  for all $t\geq 0$.  Since  for each $h\geq 0$  and for $i=1,2$,  a.s. $  L_{t\wedge \tau} ^h(m^{\intens_i}) \leq M$  for all $t\geq 0$, we get
 \begin{linenomath}
 \begin{align*}
 \I_{\{t\le \tau,\, H_t\leq a\}} \Delta'_t &\le c(M) \I_{\{t\le \tau\} }\int_0^{a} |  L_t^h(m^{\intens_1}) -L_t^h(m^{\intens_2}) | \diffd h   \\
 &\le c(M) \I_{\{t\le \tau\} }\int_0^a  \int_0^t 
 \big|m_s^{\intens_1}([0,H_s)) - m_s^{\intens_2} ([0,H_s))  \big|  \diffd L_s^h \diffd h    \\
 &\leq  c(M)  \int_0^a   \int_0^t \I_{\{s\le \tau\}}  
 \big|m_s^{\intens_1}([0,h)) - m_s^{\intens_2} ([0,h))  \big|   \diffd L_s^h \diffd h    \Big] \\
 &\le  c(M) \int_0^t    \I_{\{s\le \tau, \, H_s\leq a\}}  
 \big|m_s^{\intens_1}([0,H_s)) - m_s^{\intens_2} ([0,H_s))  \big|  \diffd s  
 \end{align*}	
 \end{linenomath}
 using that $\diffd  L_s^h=\I_{\{H_s=h\}} \diffd L_s^h$ in the third  line and  the  inhomogeneous occupation-time formula \eqref{occupation formula generalized}  in the fourth. 
 Thus,
 \begin{linenomath}
 \begin{align*}
 \E[\I_{{t\le \tau}}& \Delta'_t ] \leq  c(M) \int_0^t   \E\Big[ \I_{\{s\le \tau, \, H_s\leq a\}}  \big|m_s^{\intens_1}([0,H_s)) - m_s^{\intens_2} ([0,H_s))  \big|  \Big]  \diffd s  \\
 &\leq  c(M)  \int_0^t  \E\left[\I_{\{s\le \tau, \, H_s\leq a\}}   \E\left( 
 \int_0^{H_s}  \int_0^\infty |  \I_{\{\nu < \intens_1(s,h)}\} - \I_{\{\nu< \intens_2(s,h)}\} |\cN_s(\diffd h , \diffd \nu) \bigg| \cF_s^\rho\right)  \right] \diffd s  \\
 &\leq  c(M)  \int_0^t  \E\left[\I_{\{s\le \tau, \, H_s\leq a\}}   \E\left( 
 \int_0^{H_s}  |\intens_1(s,h) - \intens_2(s,h)| \diffd h \bigg| \cF_s^\rho\right)  \right] \diffd s
 \end{align*}	
 \end{linenomath}
 and the statement follows.
\end{proof}

\subsection{Iterative scheme and fixed point}

The remainder of this section is devoted to the proof of the second part of Theorem \ref{main thm 1} so from now on we assume that the condition {\bf (H)}  holds. 
The construction of  the adapted intensity $\intens^*$ will be achieved  by  an  iterative procedure. 
For each $t\geq 0$, we define $\intens^0(t,h)=0$ for all $h$,  so that $m^{\intens^0}$ is the null measure. 
Then we let 
\[
\intens^{n+1} =F(\intens^n), \forall n\ge 0
\]
and write  $L^h_t(n) = L^h_t(m^{\intens^n})$ to simplify notation.

Observe that, by construction  of the sequence $\intens^{n}$, the continuity of $g$ and Lemma \ref{lem: loc time of pruned is adapted},  for all $n\in \NN$ the process $(t,h)\mapsto \intens^n(t,h)$ has a  version which  a.s.\ is  continuous in $t\geq 0$ for each $h\geq 0$ and  is an adapted intensity process. 
Moreover, since $g$ is increasing,   for any pair $\intens,\tilde{\intens}$ satisfying a.s.   $\intens  \leq \tilde{\intens} $ (pointwise), one has $\I_{\{ m_t^{\intens}=0\}}\geq \I_{\{ m_t^{\tilde{\intens}}=0\}}$ for all $t\geq 0$.
We can then check  by induction in $n$ that a.s. for all  $(t,h)\in \R_+^2$, 
\begin{equation}\label{orderL}
 L^h_t(0) \geq L^h_t(2) \geq  \dots  \geq L^h_t(2n) \geq \ldots   \geq L^h_t(2n+1) \geq  \dots  \geq L^h_t(3)   \geq L^h_t(1),
\end{equation} 
and 
\begin{equation}\label{orderm}
 \intens^0(t,h) \leq \intens^2(t,h) \leq \dots \leq \intens^{2n}(t,h) \leq \ldots \leq  \intens^{2n+1}(t,h)  \leq  \dots  \leq \intens^3(t,h) \leq \intens^1(t,h).
\end{equation}
Relevant consequences of these inequalities are next highlighted:
  
\begin{lemma}\label{convml}
There exist   two adapted intensities  $\intens^e $ and  $\intens^o$ and two ${\cal P}red(\N)$-measurable processes  $(L_t^r(e): t\geq 0,r\geq 0)$,  $(L_t^r(o): t\geq 0,r\geq 0)$   such that, almost surely,  
\begin{itemize}
\item[i)] 
\begin{linenomath}
\begin{align*}
\intens^{2n}   & \nearrow  \intens^e   \quad \text{ and } \quad      L(2n) \searrow      L(e),     \\
\intens^{2n+1} & \searrow  \intens^o   \quad \text{ and } \quad      L(2n+1) \nearrow      L(o)     
\end{align*} 	
\end{linenomath}   pointwise and   moreover,
   \[
     \intens^e \le \intens^o   \mbox{  and  }  L(e)\geq L(o) \, ,
   \]
\item[ii)]  for every $(t,h)\in \R_2^+$ one has   $
L^h_t(e) = L^h_t(m^{\intens^e}) $ and $ L_t^h(o)=L^h_t(m^{\intens^o}) $  and

\item[iii)]  for every $(t,h)\in \R_2^+$ it holds that
\[
\begin{cases} \intens^e(t,h) = g(L^h_t(m^{\intens^o})) \\ \intens^o(t,h) = g(L^h_t(m^{\intens^e}))\end{cases}. 
\]
\end{itemize}
\end{lemma}

\begin{proof} 
 i) From local boundedness of the process $(t,h)\mapsto L_t^h$ (cf. the proof of Lemma \ref{lem: loc time of pruned is adapted}) and the fact  that $\intens^n(h,t)\leq g(L^h_t(0))$ for all $n\in \NN$ and $(t,h)\in \R_2^+$, the increasing sequence $\intens^{2n}$ has a pointwise  limit  $\intens^e$. 
 The limit clearly inherits the properties of $\intens^{2n}$   making  it an adapted intensity process.  
 Similar  arguments apply to the decreasing sequence $\intens^{2n+1}$. 
 Pointwise convergence of the sequences of pruned  local times and their measurability  properties are  also easily obtained by monotone limit arguments. 
 The two inequalities  at the end of part i) are immediate consequences of the inequalities \eqref{orderL} and  \eqref{orderm}. 
 
 We now move to  point ii).  For fixed $a\geq 0$ and $t\geq 0$, we have
 \begin{equation*}
  L_t^a(e) = \inf\limits_{n \in \NN} \int_0^t \I_{\{m_s^{\intens^{2n}}([0,H_t)) =0\}}\diffd L_s^a
           = \int_0^t \inf\limits_{n \in \NN} \I_{\{m_s^{\intens^{2n}} =0\}}\diffd L_s^a  = \int_0^t \I_{\{m_s^{\intens^{e}}=0\}}\diffd L_s^a \, ,           
 \end{equation*}  
 using in the third equality the fact that for each $s\geq 0$ the increasing and integer-valued sequence $m_s^{\intens^{2n}}([0,H_s))$ is upper bounded by $m_s^{\intens^{1}} ([0,H_s)) $, hence constant for large enough $n$. 
 A similar argument applies for $ L_t^h(o)$. 
     
 Finally, iii) follows from $\intens^e(h,t) =\lim_n \intens^{2n}(h,t) =\lim_n F(\intens^{2n-1})(h,t) =\lim_n g(L^h_t(2n-1)).$   
\end{proof}

\smallskip

Concluding the proof of Theorem \ref{main thm 1} is now easy:    
\begin{proposition}
With $\mathbb{P}(d\omega)\otimes \diffd t \otimes \I_{[0,H_t(\omega))}(h) \diffd h$  as reference measure, we have 
$\intens^e=\intens^o =F(\intens^e)=F(\intens^o)$ a.e. 
Moreover, $\intens^*:=\intens^e$ is the  a.e. unique solution of the equation $\intens^*=F(\intens^*)$ a.e. and the  corresponding marked exploration processes are unique, up to indistinguishability. 
\end{proposition} 

\begin{proof} By Lemma \ref{convml} iii) 
 we have $\intens^e=F(\intens^o)$  and $\intens^o =F(\intens^e)$. Applying Lemma \ref{lemma pre-Gronwall}    to   $\intens_1= \intens^e$ and  $\intens_2= \intens^o$, 
 using  Gronwall's lemma  and letting  $M$ and $a$ therein go to $+\infty$ we get,   for each $t\geq 0$, that 
 \[
 \int_0^{H_t} |\intens^e(t,h) -\intens^o(t,h)| \diffd h =0 \quad  \text{ a.s.}
 \]    
 The first statement follows from statement iii) in Lemma \ref{convml} and integration with respect to $\diffd t$. 
  
 Now, let $\intens^*$ and $\intens^{**}$ be two fixed points. 
 By applying Lemma \ref{lemma pre-Gronwall}  to $\intens_1= \intens^*$ and  $\intens_2= \intens^{**}$, we  conclude  uniqueness of fixed points using again Gronwall's lemma. 
 Last, point iv) of Lemma \ref{lem: prop mark adapt intens} provides the indistinguishability of the associated marked exploration processes. 
\end{proof}

To simplify notation,   we will write in the sequel 
\[
((\rho_t , m^*_t)  : t \geq 0):= ((\rho_t , m^{\intens^{*}}_t)  : t \geq 0)
\]
and refer to this process as the {\it competitively marked exploration process}. 
Accordingly, we will call $\intens^*$ and $L(m^*)$ the {\it competitive intensity} and {\it competitively pruned local times}, respectively. 
Notice that 
\begin{equation}\label{eqL*}
  L_t^h(m^*):=  \int_0^t  \mathbf{1}_{\{m^*_s=0\}} \diffd L_s^h 
\end{equation}
for each $h\geq 0$, a.s. for all $t\geq 0$  and 
\begin{equation} \label{eqm*}
 m^* _t([0,h])  =\int_0^{h} \int_0^{\infty} \I_{\{\nu <g(L_t^{r}(m^*))\}}\N_t(\diffd \nu,\diffd r)
\end{equation}
for each $t\geq 0$, a.s. for all $h\in [0,H_t)$. 
The rest of the paper is devoted to the proof of Theorem \ref{main thm 2}.  

\section{A Ray-Knight representation  }\label{secRKT}

\subsection{Strategy  of proof of Theorem \ref{main thm 2}}
      
Our next goal is to prove that the process of   $(L_{T_v}^a(m^*): a\ge 0, v\ge 0)$, with  $(L_t^a(m^*):t\geq 0, a\geq 0) $  constructed in the previous section and $T_x$ as in \eqref{Tx},  has the same law as the stochastic flow  \eqref{flow lb}. \smallskip 

To that end, we first construct  a family $( Z_a^{\varepsilon, \delta}(x) : x\ge 0,  a\ge 0)$  of approximations of the solution of \eqref{flow lb}, with $(\varepsilon,\delta)\in (0,\infty)^2$,  coupled  to each other and to the process \eqref{flow lb}  by using the same Gaussian and Poisson noise, and with negative competition drifts that are suitably discretized versions of the drift in \eqref{flow lb}. More precisely,  every time instant $a>0$ of the form $k\varepsilon$, $k\in \NN$,  we  split the  population into  blocks of  size $\delta$. For the next   $\varepsilon>0$ time units,   the evolution  of the $n-$th block, $n\in \NN$, is given by   a  flow like \eqref{eqflow csbp}, but with a  constant additional  negative drift  $g(n\delta)$ and suitably ``shifted'' noises. This construction must be done by means of a nested bi-recursive  (in $k$ and $n$) procedure. The  process $(Z_a^{\varepsilon, \delta}(x) : x\ge 0,  a\ge 0)$  will  then be defined by composing the dynamics  of consecutive  time strips.

Secondly, paralleling the previous  construction, we prune $(L_t^a: a\geq 0, t\geq 0)$ at a rate  that is constant inside rectangles of some two-dimensional grid, now defined in a bi-recursive way in terms of height and local time units. 
More precisely, at  heights of the form $k\varepsilon$, $k\in \NN$,  we split  the exploration times in blocks of  $\delta$  units of pruned local time at that level. Then, between height $k\varepsilon$ and $(k+1)\varepsilon$ and  when in the $n-$th exploration-time block, we prune the  original local time process   at rate $g(n\delta)$. 
\smallskip

The next step, crucially relying on   Proposition \ref{markRK}, will be to  prove that  the resulting punned local times, denoted  $(L_{T_x}^a(\varepsilon,\delta): a\geq 0, x \geq 0)$,  has the same law   as the  process $(Z_a^{\varepsilon, \delta}(x) : a\ge 0,  x\ge 0)$. 
The final and most technical  steps will be proving that, when $(\varepsilon, \delta)\to (0,0)$, for each $x\ge 0,  a\ge 0$, both the  convergences
 \[
 L_{T_x}^a(\varepsilon,\delta)\to L_{T_x}^a(m^*)\quad \mbox{ and }\quad Z_a^{\varepsilon, \delta}(x)\to Z_a(x)
 \]
hold in probability (in the respective probability spaces). This is enough to grant  finite dimensional convergence in distribution and, by the previous step, that   $(L_{T_x}^a: a\geq 0, x \geq 0)$   and $(Z_a(x) : x\ge 0,  a\ge 0)$ are equal in law as desired.

In the remainder of this section the two approximations are constructed and precise statements on their convergence and on their laws are given (proofs  are deferred to Sections 5 and 6).   We then end the present section deducing  the proof of Theorem \ref{main thm 2} from those results. 
      
\subsection{Grid approximation of  the stochastic flow of branching processes with competition}\label{gridapproxflow}

As in  \eqref{flow lb}, consider $W(\diffd s, \diffd u)$ a white noise process on $(0, \infty)^{2}$ based on the Lebesgue measure $\diffd s\otimes \diffd u$ and $N$ a Poisson random measure on $(0, \infty)^{3}$ with intensity $\diffd s\otimes \diffd \nu \otimes \Pi(\diffd r)$. 
Let $\varepsilon , \delta > 0 $  be fixed  parameters.  The process $(Z^{\varepsilon,\delta}_t(x): x\geq 0, t\geq 0)$ is constructed  in the same probability space  by means of the following bi-recursive procedure:
\begin{itemize}
 \item{\textbf{Time-step $k=0$ :}} For all $x \geq 0$, we set
 \[
 Z_0^{\varepsilon, \delta}(x): = Z_0 (x) = x;
 \] 
 \item{\textbf{Time-step $k+1$ :}} Assuming that $(Z^{\varepsilon,\delta}_t(x) : t \leq k\varepsilon, x\geq 0)$ is already constructed,  we define  
 \[
 (Z^{\varepsilon,\delta}_t(x) : t\in (k\varepsilon, (k+1)\varepsilon], x\geq 0)
 \] 
 as
  \begin{equation}\label{prunedflow}
  Z^{\varepsilon,\delta}_t(x):= Z_{k\varepsilon,t}^{\varepsilon,\delta}(Z^{\varepsilon,\delta}_{k\varepsilon}(x) ).
 \end{equation}         
where $(Z_{k\varepsilon,t}^{\varepsilon, \delta}(z): t\in (k\varepsilon, (k+1)\varepsilon] , z\geq 0)$ is an auxiliary process defined as follows:
 \begin{itemize}
  \item[*] {\textbf{Space-step $n=0$ :}}    For each $ t\in (k\varepsilon, (k+1)\varepsilon]$ we set
  \[
  Z_{k\varepsilon,t}^{\varepsilon, \delta}(0)=0.   
  \]  
  \item[*] {\textbf {Space-step $n+1$ :}}   Assuming that a process $(Z^{\varepsilon,\delta}_{k\varepsilon, t}(z) :  t\in (k\varepsilon, (k+1)\varepsilon], z\leq n\delta)$  has already been constructed, we set
  \[
  W^{(k , n)}(\diffd s,\diffd u) := W(\diffd s + k\varepsilon,\diffd \nu +Z_{k\varepsilon,s-}(n	\delta) ) 
  \]
  and 
  \[
  N^{(k , n)}(\diffd s,\diffd \nu, \diffd r) :=  N(\diffd s + k\varepsilon,\diffd \nu + Z_{k\varepsilon, s-}(n	\delta),\diffd r)
  \]
  which respectively are a Gaussian white noise with intensity $\diffd s\otimes \diffd u$ in $ \R_+ \times \R_+$ and a Poisson random measure with intensity $\diffd s \otimes \diffd \nu \otimes \Pi(\diffd r)$  in $ \R_+\times \R_+^2$, independent from each other (this can be checked using Lemma  \ref{weirdwnandpp} as in the proof of Proposition \ref{flowLBP}, iii)).  
  We then consider the stochastic flow of $\psi_{g( n\delta)}$-CSBP
  \[
  (Z^{(k, n)}_{t}(v): t\geq 0, v\geq 0)
  \]  
  defined by 
  \begin{equation}\label{Zkn}
   \begin{split}
     Z_t^{(k , n)}(v) = & v  - \alpha \int_0^t Z^{(k , n)}_s(v) \diffd s
                          + \sigma\int_0^t\int_0^{Z^{(k , n)}_{s-}(v)} W^{(k , n)}(\diffd s,\diffd u)\\
                        & + \int_0^t\int_0^{Z_{s-}^{(k , n)}(v)} \int_0^{\infty} r\tilde{N}^{(k , n)}(\diffd s , \diffd \nu,\diffd r)- \int_0^t g(n\delta) Z^{(k , n)}_s(v) \diffd s ;
   \end{split} 
  \end{equation}
  Here, $\psi_{g( n\delta)}$ is the  branching mechanism   \eqref{psitheta} with $\theta= g(n\delta)$, that is,
  \begin{equation}
  \label{psin}\psi_{g( n\delta)}(\lambda)=\psi(\lambda) + \lambda  g(n\delta). 
  \end{equation}
 We then define  for all $(t,z) \in (k\varepsilon, (k+1)\varepsilon]\times  (n\delta,(n+1)\delta]$ \, : \begin{equation}\label{Zept}
   Z_{k\varepsilon,t}^{\varepsilon,\delta} (z) := Z_{k\varepsilon,t}^{\varepsilon,\delta} (n\delta ) +   Z_{t-k\varepsilon}^{(k, n)}\left(z - n\delta \right) = \sum\limits_{l =0}^{n-1} Z_{t-k\varepsilon}^{(k, l)}(l \delta)+  Z_{t-k\varepsilon}^{(k, n)}\left(z - n\delta \right) . 
  \end{equation}
 \end{itemize}   
\end{itemize}

\begin{remark}\label{indepnoises}
It is immediate that  $( W^{(k,n)}(\diffd s,\diffd u),  N^{(k,n)}(\diffd s,\diffd \nu,\diffd r)  )_{n\in \NN}$ are independent as $k\in \NN$ varies. 
Moreover, reasoning as  proof of Proposition \ref{flowLBP}, iii)  inductively, one can also check that for fixed 
$k \in \NN$, $( W^{(k,n)}(\diffd s,\diffd u), N^{(k,n)}(\diffd s,\diffd \nu,\diffd r))$ are independent  as $n\in \NN$  varies.  Thus, the processes 
 \[
  \left( Z_h^{(k,n)}(v)=Z^{\varepsilon,\delta}_{k\varepsilon, k\varepsilon+h}(n\delta +v)- Z^{\varepsilon,\delta}_{k\varepsilon, k\varepsilon+h}(n\delta):  h\in [0,\varepsilon], v\in [0,\delta]  \right)_{k,n \in \NN }
 \]
are independent too. 
\end{remark}
  
\begin{definition}\label{gridflowepsdel}  We  refer to process $(Z_t^{\varepsilon,\delta}(x): t \geq 0 , x\geq 0)$   as the $(\varepsilon,\delta)$- grid approximation of the stochastic flow of continuous state branching processes with competition, or simply   {\it grid approximation of the stochastic flow}.  
\end{definition}
  
The parameters $(\varepsilon,\delta)$ being  fixed, for each $s,x\geq 0$ and $k\in \NN$ we introduce the notation 
  \begin{equation}\label{k}
 k_s=k_s(\varepsilon):=  \sup\{j \in \NN:  \, j\varepsilon < s\}
 \end{equation} 
 and 
  \begin{equation}\label{n}
 n_x^{k}=n_x^{k}(\varepsilon,\delta):=\sup\{m \in \NN: m\delta < Z^{\varepsilon,\delta}_{k\varepsilon} (x) \}.
 \end{equation}

By induction in $k\in\NN$, summing at each step over $n\in \NN$, $(Z_t^{\varepsilon,\delta}(x): t \geq 0 , x\geq 0)$ is seen to solve a sort of stochastic-flow equation, but with a  ``locally frozen'' drift term. 
More precisely, it is not hard to check
\begin{lemma} For each $x\geq 0$, 
the process $Z^{\varepsilon,\delta}(x) = (Z^{\varepsilon,\delta}_t(x):t\geq 0)$ is a solution to the  stochastic differential equation:
\begin{equation}\label{prunedSDE}
 \begin{aligned}
  Z_t^{\varepsilon, \delta} (x) = x &- \alpha \int_0^t Z^{\varepsilon,\delta}_s(x) \diffd s
                   + \sigma \int_0^t \int_0^{Z_{s-}^{\varepsilon,\delta}(x)}W(\diffd s,\diffd u)
                   + \int_0^t \int_0^{Z_{s-}^{\varepsilon, \delta}(x)}\int_0^{\infty} r \tilde{N}(\diffd s,\diffd \nu,\diffd r)\\
                                    &-\int_0^t \sum\limits_{n=0}^{n^{k_s}_x} g(n\delta) \left(
                                     Z_{k_s\varepsilon,s}((n+1)\delta \wedge Z^{\varepsilon,\delta}_{k_s\varepsilon}(x)) - Z_{k_s\varepsilon,s}(n\delta \wedge Z^{\varepsilon,\delta}_{k_s\varepsilon}(x))\right) \diffd s
%                                                     Z^{(k_s, n)}_{s-k_s\varepsilon}\left((n+1)\delta \wedge Z^{\varepsilon,\delta}_{k_s\varepsilon}(w) - n\delta \wedge Z^{\varepsilon,\delta}_{k_s\varepsilon}(w)\right)  \diffd s\,
                   , \quad  t\geq 0.\\
 \end{aligned}
\end{equation}
Moreover, by construction and properties of the flows  \eqref{Zkn} of CSBP, $ (Z^{\varepsilon,\delta}_t(x):t\geq 0,x\geq 0)$ has a bi-measurable version such that,   almost surely, for each $x\geq 0$ the process $t \mapsto Z_t^{\varepsilon, \delta} (x)$ is non-negative and c\`adl\`ag and, for every $t \geq 0$, $x \mapsto Z_t^{\varepsilon, \delta} (x)$ is  non-negative,  non-decreasing and c\`adl\`ag. 
\end{lemma}

The proof of the following result is deferred to Section  \ref{consitstochflow}.
   
\begin{proposition}[{\bf Convergence of the  stochastic flow grid approximation}]\label{convergenceflow}
For each $x\geq 0$ and $t \geq 0$, the r.v.  $Z_{t}^{\varepsilon,\delta}(x)$ converges  to $Z_{t}(x)$  in probability when $(\varepsilon,\delta)\to (0,0)$.
\end{proposition}     
            
\subsection{Grid approximation of  competitively marked local times}\label{gridapproxloctime}

 Given fixed parameters $\varepsilon , \delta > 0 $, we now introduce an approximation $\intens^{\varepsilon,\delta}$ of the intensity process $\intens^*$, which will define a pruning at a piecewise constant rate, the regions where it is constant being now  the interior of  rectangles in a  suitably defined height/local-time   grid. 
We will denote by $((\rho_s, m_s^{\varepsilon,\delta}): s \geq 0)$ the corresponding marked exploration process     and by
\[
L(\varepsilon,\delta)=(L_t^a(\varepsilon,\delta) : t \geq 0, a \geq 0)
\] 
the associated pruned local time process, that is, $L_t^a(\varepsilon,\delta) =\int_0^t \mathbf{1}_{\{ m_s^{\varepsilon,\delta} =0\} }\diffd L_s^s. $

\medskip 

\begin{figure}[h]
    \centering
   \includegraphics[width=1\textwidth]{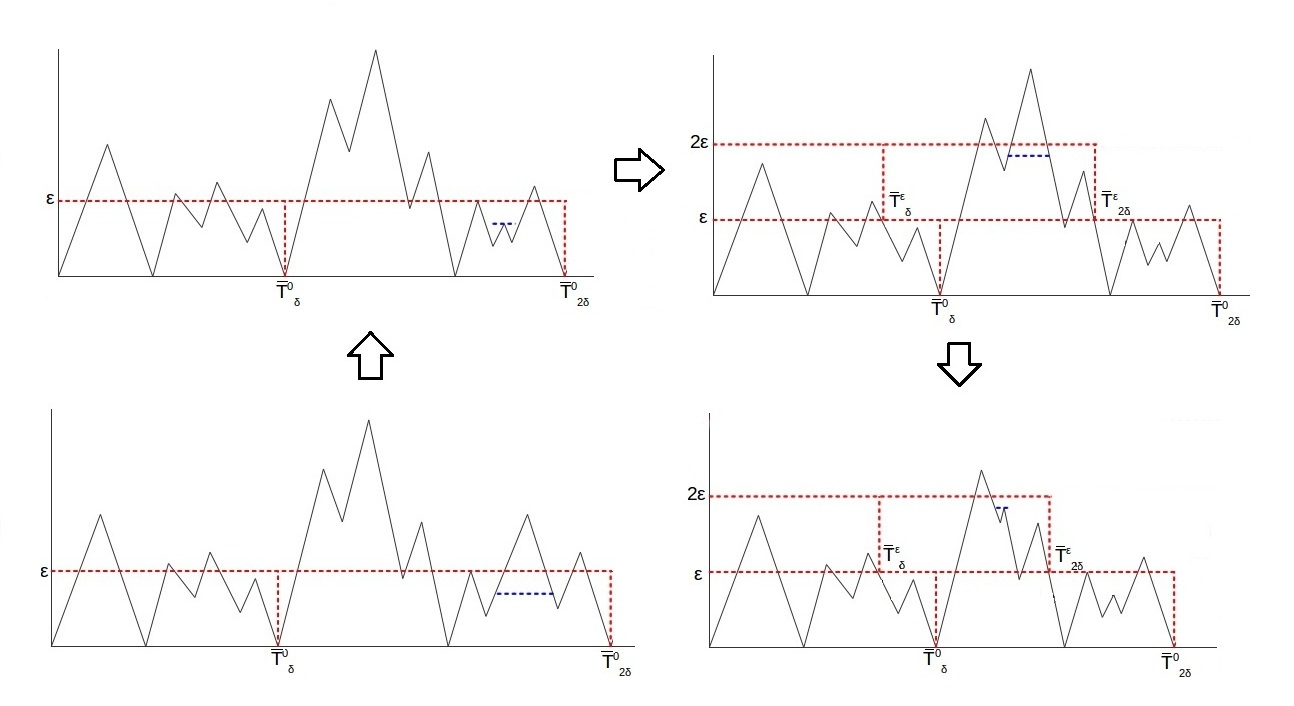}
    \caption{Construction of the local time grid approximation.}
\end{figure} 

Again, the construction consists in  a doubly recursive procedure.  
\begin{itemize}
 \item {\textbf{Height-step $0$ :}} we set
 $\intens^{\varepsilon,\delta}(t,0):=0 $ for all $t\geq 0$,   and $ \overline{T}_{x}^{0}=T_{x} $ for all $x\in \R. $
 \item {\textbf{Height-step $k+1$ :}} Assuming  that   $\left(\left( \intens^{\varepsilon,\delta}(t,h),  L^h_t(\varepsilon,\delta) \right): t \geq 0,  h \in [0, k\varepsilon]\right)$ is  constructed,  we set for all $x\in \R$,   
 \[
 \overline{T}^{k\varepsilon}_{x}:=\inf\{t\geq 0: L^{k\varepsilon}_t(\varepsilon,\delta) = x\}\, , \, 
 \mbox{    with } \inf \emptyset =\infty.
 \]
   The process
 $\left(\left(L^h_t(\varepsilon,\delta),\intens^{\varepsilon,\delta}(t,h)\right): t \geq 0,  h \in (k\varepsilon, (k+1)\varepsilon]\right)$ is then defined
 as follows:
 \begin{itemize}
   \item[*] {\textbf{Time-step $0$:}} For each $h \in (k\varepsilon, (k+1)\varepsilon] $ we set  $\intens^{\varepsilon,\delta}(0,h)=g(0).$
   \item[*] {\textbf{Time-step $n+1$:}} Assume that  $\overline{T}^{k\varepsilon}_{n\delta}<\infty$ and that  
   \[
   \left(\left(L^h_t(\varepsilon,\delta), \intens^{\varepsilon,\delta}(t,h)\right):  t \in (0,\overline{T}^{k\varepsilon}_{n\delta}] , h \in (k\varepsilon, (k+1)\varepsilon]) \right) \, \text{ is already defined. }
   \] 
  If  $ \overline{T}^{k\varepsilon}_{(n+1)\delta}< \infty$,  for all $(t ,h)\in (\overline{T}^{k\varepsilon}_{n\delta}, \overline{T}^{k\varepsilon}_{(n+1)\delta} ] \times (k\varepsilon,(k+1)\varepsilon]$  we  set  
   \begin{equation}\label{defintensepdel}
     \intens^{\varepsilon,\delta}(t,h):= g(n\delta) \,  ,
   \end{equation}
   \begin{equation} \label{defmepdel} 
     m^{\varepsilon,\delta}_t([0,h]):= m^{\varepsilon,\delta}_t([0,k\varepsilon]) + \int_{(k\varepsilon, \, h]} \int_0^{\infty} \I_{\{\nu <  g( n\delta)\}}\N_t(\diffd r, \diffd \nu) \, ,
   \end{equation} 
   and
   \begin{equation} \label{defLepdel0}
    L_t^h(\varepsilon,\delta): =   L_{\overline{T}^{k\varepsilon}_{n\delta}}^h(\varepsilon,\delta) +  \int_{\overline{T}^{k\varepsilon}_{n\delta}}^{t} \I_{\{m_s^{\varepsilon,\delta}=0\}}\diffd L_s^h .  \\
   \end{equation} 
        If  $\overline{T}^{k\varepsilon}_{(n+1)\delta}=\infty$, we apply this definitions for all $(t ,h)\in (\overline{T}^{k\varepsilon}_{n\delta}, \infty) \times (k\varepsilon,(k+1)\varepsilon]$, and the inductive procedure in $n$ stops. Note that in this case, we have $ n\delta \leq L_{\infty}^{k\varepsilon}(\varepsilon,\delta) \leq (n+1)\delta$.  
                 \end{itemize} 
\end{itemize}

 It is easily seen that $\intens^{\varepsilon,\delta} $ is ${\cal P}red(\N)$-measurable and that  $\intens^{\varepsilon,\delta} (s,h)<\intens^{\varepsilon,\delta} (s',h)$ can happen for $s\leq s'$ only if $L_s^{k_h\varepsilon}<L^{k_h\varepsilon}_{s'}$  and hence $h> H_{s,s'}$. 
It therefore is an adapted intensity process too.  
   
   \medskip 
   
 \begin{remark}\label{Tn+1deltaTndelta} Note that, for each  $\varepsilon,\delta >0$ and $(k,n) \in  \NN^2 $,   $L^{k\varepsilon}_{\overline{T}_{(n+1)\delta }^{k\varepsilon}} (\varepsilon,\delta)- L^{k\varepsilon}_{\overline{T}_{n\delta }^{k\varepsilon}} (\varepsilon,\delta)$ takes the value $\delta$ if $\overline{T}_{(n+1)\delta }^{k\varepsilon}<\infty$, the value $0$ if $\overline{T}_{n\delta }^{k\varepsilon}=\infty$ and some value in $[0,\delta]$ if  $\overline{T}_{n\delta }^{k\varepsilon}<\infty$ and   $\overline{T}_{(n+1)\delta }^{k\varepsilon}=\infty$.
\end{remark}
\medskip

For fixed   parameters $(\varepsilon,\delta)$ and for $t,h\geq 0$ and $k\in \NN$ we introduce the notation 
\begin{itemize}
 \item $\overline{n}_t^{k}=\overline{n}_t^{k}(\varepsilon,\delta):=\sup\{n \in \NN: n\delta <L_t^{k\varepsilon}(\varepsilon,\delta)\},$
 \end{itemize}
and   recall that we have set $k_{h}=k_{h} (\varepsilon)  = \sup \{k \in \NN : k\varepsilon <  h \}$ (see Subsection \ref{gridapproxflow}). 
Hence, by construction,    a.s  for each $t\geq 0$ and  all $h\in [0,H_t)$ it holds that  
\[
\intens^{\varepsilon,\delta} (t,h) =g(\overline{n}_t^{k_h}\delta ).
\]
 Moreover, we have
\begin{equation} \label{Lepsdeleq}
 L^h_t(\varepsilon, \delta) = \int_0^t \I_{\{m^{\varepsilon, \delta}_s=0\}}\diffd L_s^h=L^h_t(m^{\varepsilon, \delta}) 
\end{equation} 
 and 
\begin{equation} \label{mepsdeleq}
 m^{\varepsilon, \delta} _t([0,h])  =\int_0^{h} \int_0^{\infty} \I_{\left\{\nu < \intens^{\varepsilon,\delta} (t,r)\right\}}\N_t(\diffd \nu,\diffd r) =  m^{\intens^{\varepsilon, \delta} }_t(h) ,     
\end{equation}
 which should be compared to the system of equations \eqref{eqL*}--\eqref{eqm*}.

\begin{definition}\label{gridfloctime}  We  refer to the  processes  $(L_t^a (\varepsilon,\delta): a \geq 0 , t\geq 0)$ as the $(\varepsilon,\delta)$- grid approximation of pruned local time or simply {\it local time grid approximation}.  
\end{definition}

\medskip

In what follows, for each fixed height $a\geq 0$ we will denote by ${\cal E}_a$ the sigma-field 
\begin{equation}\label{calEa}
 {\cal E}_a:=  \left( (\rho_{\tilde{\tau}_t^a},{\cal N}_{\tilde{\tau}_t^a})\, : t\geq 0 \right)
\end{equation}
where $\tilde{\tau}_t^{a}$ is the right continuous inverse of $t\mapsto \tilde{A}^a_t:=\int_0^t \I_{\{ H_s \leq a\}} \diffd s$   (notice that  ${\cal E}_a$ increases with $a$). 

\medskip

The following results rely on  excursion theory for snake processes and  will be  proved in Section \ref{ConsistLawGridTree}:
 
\begin{proposition}[{\bf Law of the local time grid approximation}] \label{lawgrid}
For each fixed  $\varepsilon,\delta >0$ and $(k,n) \in  \NN^2 $,  the random variable $L^{k\varepsilon}_{\overline{T}_{(n+1)\delta }^{k\varepsilon}} (\varepsilon,\delta)- L^{k\varepsilon}_{\overline{T}_{n\delta }^{k\varepsilon}} (\varepsilon,\delta) $ is $ {\cal E}_{k\varepsilon}$-measurable, and the  process
\begin{equation}\label{L-L}
 \left( L^{k\varepsilon+h}_{\overline{T}_{n\delta +z}^{k\varepsilon}} (\varepsilon,\delta)- L^{k\varepsilon + h}_{\overline{T}_{n\delta}^{k\varepsilon}} (\varepsilon,\delta):  h\in (0,\varepsilon], z\in [0,\delta] \right)
\end{equation}
is ${\cal F}_{\overline{T}_{(n+1)\delta}^{k\varepsilon}}\bigwedge  {\cal E}_{(k+1)\varepsilon}$ measurable. 
Moreover, the conditional law  of the  process \eqref{L-L}  given   ${\cal F}_{\overline{T}_{n\delta}^{k\varepsilon}}\bigvee {\cal E}_{k\varepsilon}$  is equal  to the (unconditional) law of  
\[
 (Z^{(k,n)}_h(z \wedge y):  h\in [0,\varepsilon], z\in [0,  \delta]),
\] 
when
 $y=L^{k\varepsilon}_{\overline{T}_{(n+1)\delta }^{k\varepsilon}} (\varepsilon,\delta)- L^{k\varepsilon}_{\overline{T}_{n\delta }^{k\varepsilon}} (\varepsilon,\delta)$ (here, $Z^{(k,n)}$ is a stochastic flow of continuous state branching processes 
with branching mechanism  $ \psi_{g(n\delta) }$ as in \eqref{psin}).
\end{proposition}

\begin{lemma}[{\bf Flow property of pruned  local-times}]\label{flowloctime} Let $\intens$ be an adapted intensity and for each $a>0$ denote by $\breve{T}_ x^a:=\inf\{ t>0: \, L_t^a(m^{\intens})>x \}\leq \infty, x\geq 0$,  the right-continuous inverse  of the pruned local   processes $(L_t^a(m^{\intens}): t\geq 0)$. Then, for every $a,b,c,x>0$ we a.s. have  
\[
 L^{a+b+c}_{\breve{T}^a_x}(m^{\intens})=L^{a+b+c}_{\breve{T}^{a+b}_{L^{a+b}_{\breve{T}^a_x}(m^{\intens})}}(m^{\intens})\quad. 
 \]
\end{lemma}

\begin{proposition}[{\bf Consistency of the local-time grid approximation}] \label{convgrid}
For each $x\geq 0$ and $a \geq 0$, the random variable $L_{T_x}^a(\varepsilon,\delta)$ converges in probability to $L_{T_x}^a(m^*)$ as $(\varepsilon,\delta)\to 0$.
\end{proposition}

\subsection{Proof of the Ray-Knight representation}

\begin{corollary}\label{identlawgrid}
$(Z^{\epsilon,\delta}_a(x): x\geq 0, a\geq 0)$ and $(L^a_{T_{x}}(\epsilon, \delta): x\geq 0, a\geq 0)$ have the same law.
\end{corollary}
                
\begin{proof}  
Notice that for every $a,x\geq 0$  
\begin{linenomath}
\begin{equation*}
\begin{split}
Z^{\varepsilon, \delta}_a (x)&=  Z^{\varepsilon, \delta}_{k_a\varepsilon,a} (Z^{\varepsilon,\delta}_{k_a\varepsilon}(x))= Z^{\varepsilon, \delta}_{k_a\varepsilon,a} \circ Z^{\varepsilon,\delta}_{k_a\varepsilon}(x)\\ & = Z_{k_a\varepsilon,a}^{\varepsilon,\delta} \circ Z_{(k_a-1)\varepsilon,k_a\varepsilon}^{\varepsilon,\delta} \circ Z_{(k_a - 1)\varepsilon}^{\varepsilon,\delta}(x) =\cdots\\
& =Z_{k_a\varepsilon,a}^{\varepsilon,\delta} \circ  Z_{(k_a-1)\varepsilon,k_a\varepsilon}^{\varepsilon,\delta} \circ \cdots \circ Z_{0,\varepsilon}^{\varepsilon,\delta}(x)
\end{split}
\end{equation*} 	
\end{linenomath}
  a.s.   On the other hand,
  \begin{linenomath}
  	\begin{equation*}
  	\begin{split}
  	L^a_{T_{x}}(\varepsilon, \delta)& = L^{a}_{\overline{T}^{ k_a\varepsilon}_{L^{k_a\varepsilon}_{T_x} (\varepsilon,\delta)}}( \varepsilon,\delta) =  L^{a}_{\overline{T}^{ k_a\varepsilon}_{\bullet }}( \varepsilon,\delta) \circ  L^{k_a\varepsilon}_{T_x} (\varepsilon,\delta)\\
  	&=L^{a}_{\overline{T}^{ k_a\varepsilon}_{\bullet }}( \varepsilon,\delta) \circ  L^{k_a\varepsilon}_{\overline{T}^{(k_a -1)\varepsilon}_\bullet} (\varepsilon,\delta) \circ L^{(k_a-1)\varepsilon}_{T_x} (\varepsilon,\delta) = \cdots\\
  	&= L^{a}_{\overline{T}^{ k_a\varepsilon}_{\bullet }}( \varepsilon,\delta) \circ  L^{k_a\varepsilon}_{\overline{T}^{(k_a -1)\varepsilon}_\bullet} (\varepsilon,\delta) \circ L^{(k_a-1)\varepsilon}_{\overline{T}^{(k_a -2)\varepsilon}_\bullet} (\varepsilon,\delta) \circ \cdots
  	\circ L^{\varepsilon}_{T_x} (\varepsilon,\delta)
  	\end{split}
  	\end{equation*}
  \end{linenomath}  
for every $a,x\geq 0$ a.s.   by Lemma \ref{flowloctime}. 
  It is thus enough to show that, for each $m\in\NN$, the  equality in law  
  \begin{equation}\label{LZeqlaw}
  \left( L^{k\varepsilon+h}_{\overline{T}_{y}^{k\varepsilon}} (\varepsilon,\delta):  h\in [0,\varepsilon], y\geq 0 \right)_{k\leq m} \overset{d}{=}\left( Z^{\varepsilon,\delta}_{k\varepsilon, k\varepsilon+h}(y \wedge Z^{\varepsilon,\delta}_{k\varepsilon}(\infty)  ):  h\in [0,\varepsilon], y\geq 0 \right)_{k\leq m}
 \end{equation}
holds. To that end we proceed by induction on $k$, with help of  the two families  of processes
  \[
  \left( A^{(k,n)}(h,z):= L^{k\varepsilon+h}_{\overline{T}_{n\delta +z}^{k\varepsilon}} (\varepsilon,\delta)- L^{k\varepsilon +h }_{\overline{T}_{n\delta}^{k\varepsilon}} (\varepsilon,\delta):   h\in [0,\varepsilon], z\in [0,\delta] \right)_{k,n \in \NN }
 \]
 and 
  \[
  \left(  B^{(k,n)}(h,z):= Z_h^{(k,n)}\left(z \wedge\left(   Z^{\epsilon,\delta}_{k\varepsilon }(\infty)-n\delta\right)_+ \right):  h\in [0,\varepsilon], z\in [0,\delta]  \right)_{k,n \in \NN }.
 \]
 Since $L^0_{\overline{T}^0_{(n+1)\delta}}(\varepsilon,\delta)-L_{\overline{T}^0_{n\delta}}^0(\varepsilon,\delta) =L^0_{T^0_{(n+1)\delta}} -L_{T^0_{n\delta}}=\delta$ a.s.\  for all $n\in \NN$,  Proposition \ref{lawgrid}    implies  that  the processes  $\left( A^{(0,n)} \right)_{n \in \NN }$  are independent. Thus, 
   $\left( A^{(0,n)} \right)_{n \in \NN }$  and $\left( B^{(0,n)} \right)_{n \in \NN }$ have the same law by Remark \ref{indepnoises},  and the equality in law  \eqref{LZeqlaw}  follows in the case $m=k=0$ by summing over $n\leq n^0_z$ (note that $Z^{\epsilon,\delta}_{0}(\infty)=\infty$) for each $x\geq 0$. 
   
 We assume next that for $m\geq 1$  the subfamilies  
  $\left( A^{(k,n)} \right)_{k\leq m-1, n \in \NN }$  and $\left( B^{(k,n)} \right)_{k\leq m-1,n \in \NN }$ have the same law. By  summing over $n_x^k$ for each $x\geq 0$  and every $k\leq m-1 $,    this also grants that  the equality in law \eqref{LZeqlaw} for $m-1$ (instead of $m$) holds.
  Let now $(\hat{Z}^{(m,n)})_{n\in \NN}$ be a family of processes equal in law to the sequence  of   processes  $(Z^{(m,n)})_{n\in \NN}$ defined in \eqref{Zkn}, which are independent by Remark  \ref{indepnoises}. Set $\tilde{n}=\inf\{ n\in \NN: \overline{T}^{m\varepsilon}_{(n+1)\delta}=\infty\} \leq \infty $ and, for each $m,n \in \NN$ and $y\geq 0$,  define $\tilde{Z}^{(m,n,y)}=\left(  \hat{Z}_h^{(m,n)}\left(z \wedge y \right):  h\in [0,\varepsilon], z\in [0,\delta]  \right)$.  By Proposition \ref{lawgrid}, the conditional law  of    $\left( A^{(m,n)} \right)_{ n \in \NN }$  given  $\left( A^{(k,l)} \right)_{k\leq m-1, l \in \NN }$ is equal to the (unconditional) law of
  $ \left( \tilde{Z}^{(m,n,y_n)} \right)_{n\in \NN}$
 with $(y_n)_{n\in \NN} \in \R_+^{\NN}$ given by 
 \[
  y_n= \delta \wedge( L^{m\varepsilon}_{\infty}(\varepsilon,\delta)-n\delta)_+ = \begin{cases} \delta & \mbox{if }  n <\tilde{n} \\   L^{m\varepsilon}_{\infty}(\delta,\varepsilon)-\tilde{n}\delta   & \mbox{if }n= \tilde{n}  \\
0 & \mbox{if }  n >\tilde{n}. \\ 
  \end{cases} 
  \]
 In particular, for every $z\in [0,\delta]$ and $n\in \NN$ one has $z\wedge y_n= z\wedge( L^{m\varepsilon}_{\infty}(\varepsilon,\delta)-n\delta)_+ $.  The desired equality in law follows then  for $k\leq m$  from the induction hypothesis and the fact that 
 \[
  L^{m\varepsilon}_{\infty}(\varepsilon,\delta)= \sum\limits_{n\in \NN} A^{(m-1,n)}(\varepsilon, \delta)\quad  \text{ and } \quad Z_{m\varepsilon}^{\varepsilon,\delta}(\infty)=\sum\limits_{n\in \NN} B^{(m-1,n)}(\varepsilon, \delta). 
  \]
\end{proof}
  
\begin{proof}[Proof of Theorem \ref{main thm 2}] 
 By Proposition \ref{convergenceflow} (resp. Proposition \ref{convgrid}), when $(\varepsilon,\delta)\to (0,0)$ the process $(Z^{\varepsilon,\delta}_a(x): x\geq 0, a\geq 0)$ (resp. $(L_{T_x}^a(\varepsilon,\delta): x\geq 0, a\geq 0)$) converges to $(Z_a(x): x\geq 0, a\geq 0)$ (resp. $(L_{T_x}^a(m^*): x\geq 0, a\geq 0)$) in the sense of finite dimensional distributions. 
 From these convergences and Corollary \ref{identlawgrid} we conclude that $(Z_a(x): x\geq 0, a\geq 0)$ and  $(L_{T_x}^a(m^*): x\geq 0, a\geq 0)$ have the same finite dimensional distributions.
\end{proof}

\begin{remark}
Although for all $\varepsilon>0,\delta>0$, the process $(Z^{\varepsilon,\delta}_a(x): x\geq 0, a\geq 0)$ has the same law as $(L_{T_x}^a(\varepsilon,\delta): x\geq 0, a\geq 0)$ and $(Z_a(x): x\geq 0, a\geq 0)$ the same as $(L_{T_x}^a(m^*): x\geq 0, a\geq 0)$, it is not clear whether   the couplings   $\left((Z^{\varepsilon,\delta}_a(x), Z_a(x)  ) : x\geq 0, a\geq 0\right)$ and  
$\left((L_{T_x}^a(\varepsilon,\delta), L_{T_x}^a(m^*)  ) : x\geq 0, a\geq 0 \right)$ are equal in law or not.
\end{remark}

\section{Consistency of the stochastic flow approximation}\label{consitstochflow}

\subsection{Some auxiliary results}

In this section our goal is to prove Proposition \ref{convergenceflow}. 
Recall that the solutions of the various equations \eqref{eqflow csbp}, \eqref{flow lb}  and \eqref{prunedSDE} that will be used in its proof are constructed on the same probability space and with the same driving processes.  The  following comparison property  will be useful. 
   
\begin{lemma}[{\bf Comparison property}]\label{comparison property}
Let $(Y_t (y):t \geq 0,y\geq 0) $  be the solution of equation \eqref{eqflow csbp}. 
For all  $\varepsilon,\delta \geq 0$, the process  $(Z_t^{\varepsilon, \delta}(z): t\geq 0,z\geq 0)$  defined by  \eqref{prunedSDE} and the solution $(Z_t(z): t\geq 0,z\geq 0)$ to equation \eqref{flow lb} satisfy
\begin{equation}\label{eqcomparison property} 
 \p \{Z^{\varepsilon,\delta}_t(v) \leq Y_t (w)\, , \forall t\geq 0\}= 1 \quad \mbox{ and } \quad  \p \{Z_t (v)\leq Y_t (w)\, , \forall t\geq 0\}= 1,
\end{equation}
for all $ 0\leq v \leq w$. 
In both cases we say that the ``comparison property'' holds.
\end{lemma}
                     
\begin{proof}  
 The comparison property for $Z(v)$ with respect to $Y(w)$ follows directly from  \cite[Theorem 2.2]{DL}.   Whenever  $y\geq z$, the same result implies a comparison property for each  of the $\psi_{g(n\delta)}$-CSBP $( Z^{(k,n)}_{t}(z): t \geq 0)$ defined in \eqref{Zkn},  with respect to a $\psi$-CSBP  $(  Y_t^{(k , n)}(y): t \geq 0)$  constructed using  {\it the same  noises}  $( W^{(k,n)}(\diffd s,\diffd u),  N^{(k,n)}(\diffd s,\diffd \nu,\diffd r)  )_{n\in \NN}$. 
 Since $\left( Z^{\varepsilon,\delta}_{k\varepsilon, k\varepsilon+h}(z):  h\in [0,\varepsilon] \right)$ is build for each $k$ by  summing  processes $( Z^{(k,n)}_{t}(x): t \geq 0)$ defined  in \eqref{Zept},  over finitely many $n$  and for  suitable values of $x$,  an inductive argument in $n$ yields the comparison property of that process with respect to  the process $\big( Y_{k\varepsilon, k\varepsilon+h}(y):  h\in [0,\varepsilon]   \big)$  defined in \eqref{flow property CSBP}, which can be obtained  by similarly  summing over $n$ some of the above described  processes $(  Y_t^{(k , n)}(x): t \geq 0)$. 
 Last,   $(Z_t^{\varepsilon,\delta}(v): t\geq 0,v\geq 0)$ and $(Y_t(w): t\geq 0,w\geq 0)$ can   be  respectively obtained  from  the families  $( Z^{\varepsilon,\delta}_{k\varepsilon, k\varepsilon+\cdot})_{k\in \NN}$  and $( Y_{k\varepsilon, k\varepsilon+\cdot})_{k\in \NN}$  by composition, so an induction argument in $k$  completes the proof of the desired property.    
\end{proof}

To unburden the proof of Proposition \ref{convergenceflow}, we next prepare  a technical lemma.
 
\begin{lemma}\label{boundz-v}
Let $x\geq 0$ be fixed and for $m\geq 0$ let $\tau_{m} =\tau_{m}(x): = \inf\{t \geq 0 : Y_{t}(x) > m \}$. 
Then,  for all $s\geq 0$, 
\begin{linenomath}
\begin{equation*}
\begin{split}
\E \left[\I_{\{s \leq \tau_m\}} \sup\limits_{0 \leq v \leq Z^{\varepsilon,\delta}_{k_s\varepsilon}(x)} \left|Z^{\varepsilon,\delta}_{k_s\varepsilon,s}(v) - v\right| \right] \leq m\varepsilon &\left(|\alpha_0| +\int_1^{\infty} r \Pi(\diffd r)+g(m)\right)  \\
& + 2 \sqrt{m\varepsilon} 
\left(\sigma + \sqrt{\int_0^1 r^2\Pi(\diffd r)}\right),
\end{split}
\end{equation*}	
\end{linenomath}
where $\alpha_0 =  - \alpha - \int_1^{\infty} x\Pi(\diffd x) $.
\end{lemma}  

\begin{proof} 
 It follows  from \eqref{Zkn} and \eqref{Zept} that 
\begin{linenomath}
\begin{align*}
\E & \left[\I_{\{s \leq \tau_m\}}  \sup\limits_{0 \leq v \leq Z^{\varepsilon,\delta}_{k_s\varepsilon}(x)} \left|Z^{\varepsilon,\delta}_{k_s\varepsilon,s}(v) - v\right| \right] \\ 
&\leq \E\left[ \I_{\{s \leq \tau_m\}} \left( |\alpha_0| \int_{k_s\varepsilon}^s Z^{\varepsilon,\delta}_{\theta}(x)\diffd \theta +
\int_{k_s\varepsilon}^s \int_0^{Z^{\varepsilon,\delta}_{\theta}(x)}\int_1^{\infty} rN (\diffd \theta, \diffd \nu,\diffd r)\right)\right] \\
& \quad + \sigma \E\left[\I_{\{s \leq \tau_m\}} \sup\limits_{0 \leq v \leq Z^{\varepsilon,\delta}_{k_s\varepsilon}(x)} \left|\int_{k_s\varepsilon}^s \int_0^{Z^{\varepsilon,\delta}_{k_s\varepsilon,\theta}(v)} W(\diffd \theta, \diffd u)\right| \right] \\
& \quad +  \E\left[\I_{\{s \leq \tau_m\}} \sup\limits_{0 \leq v \leq Z^{\varepsilon,\delta}_{k_{s}\varepsilon}(x)} \left| \int_{k_s\varepsilon}^s \int_0^{Z^{\varepsilon,\delta}_{k_s\varepsilon,\theta}(v)}\int_0^1 r\tilde{N}(\diffd \theta, \diffd \nu,\diffd r)\right| \right] \\
& \quad + \E\left[\I_{\{s \leq \tau_m\}} g(Z^{\varepsilon,\delta}_{k_s\varepsilon} (x)) \int_{k_s\varepsilon}^{s} Z^{\varepsilon,\delta}_{\theta} (x)\diffd \theta \right]. 
\end{align*}	
\end{linenomath}
 Thanks to Lemma \ref{comparison property}, on $\{s \leq \tau_m\}$ it holds a.s. for every $(t,v) \in (k_s\varepsilon, s]\times [0,Z^{\varepsilon,\delta}_{k_s\varepsilon}(x))$ that $$Z_{k_s\varepsilon,t}^{\varepsilon,\delta}(v) = Z_{k_s\varepsilon,t}^{\varepsilon,\delta}(v)\wedge m. $$
 We deduce that
 \begin{equation}\label{bounddif}
  \begin{aligned}
   \E\left[\I_{\{s \leq \tau_m\}}  \sup\limits_{0 \leq v \leq Z^{\varepsilon,\delta}_{k_s\varepsilon}(x)} \left|Z^{\varepsilon,\delta}_{k_s\varepsilon,s}(v) - v\right| \right]\leq & \,  \sigma \E\left[  \sup\limits_{0 \leq v \leq Z^{\varepsilon,\delta}_{k_{s}\varepsilon}(x)\wedge m} |M_v^W(s)| \right]\\
   &  +   \E\left[ \sup\limits_{0 \leq v \leq Z^{\varepsilon,\delta}_{k_{s}\varepsilon}(x) \wedge m} |M_v^{N}(s)| \right]  \\
   & + m\varepsilon \left(|\alpha_0| + \int_1^{\infty} r \Pi(\diffd r)+ g(m)\right)    \, 
  \end{aligned}
 \end{equation}
 where, for each $s\geq 0$, $(M_v^W(s ))_{v \geq 0}$ and $(M_v^{N}(s ))_{v \geq 0}$ denote the processes respectively defined as
 \[
  M_v^W(s) = \int_{k_s\varepsilon}^s \int_0^{Z^{\varepsilon,\delta}_{k_s\varepsilon,\theta}(v) \wedge m} W(\diffd \theta, \diffd u) \quad \mbox { and } \quad  M_v^{N}(s) = \int_{k_s\varepsilon}^s \int_0^{Z^{\varepsilon,\delta}_{k_s\varepsilon,\theta}(v) \wedge m}\int_0^1 r\tilde{N}(\diffd \theta, \diffd \nu,\diffd r). 
 \]
 Standard It\^o  calculus for  stochastic integrals with respect to $W$ and $\tilde{N}$ show for each $s,v\geq 0$ that 
  \begin{equation*}
 \E\left[{M_{v}^{W}(s)}^2\right] =  \E \left[ \int_{k_{s}\varepsilon}^{s}(Z^{\varepsilon,\delta}_{k_{s}\varepsilon,\theta}(v) \wedge m)\diffd \theta\right]\leq m\varepsilon
 \end{equation*}
 and 
  \begin{equation*}
   \E\left[{M_{v}^{\tilde{N}}(s)}^2\right] =  \E\left[\int_{k_{s}\varepsilon}^{s} 
   \int_0^{Z_{k_{s}\varepsilon,\theta}(v) \wedge m}\int_0^{1} r^2\Pi(\diffd r)\diffd \theta \diffd \nu \right] \leq   m\varepsilon \, \int_0^{1} r^2\Pi(\diffd r) .
  \end{equation*}
 Moreover,  $(M_v^W(s ))_{v \geq 0}$ and $(M_v^{N}(s ))_{v \geq 0}$ are martingales in the variable $v$, issued from $0$, with respect to the filtration $(\mathcal{S}^{k_s\varepsilon,s}_{v})_{v \geq 0}$ given by 
\[
 \mathcal{S}^{k_s\varepsilon,s}_{v} := \sigma \left(\I_{[ 0,Z^{\varepsilon,\delta}_{k_s\varepsilon,t}(w)]}(u)  W(\diffd \theta, \diffd u) ,  \I_{[ 0,Z^{\varepsilon,\delta}_{k_s\varepsilon,t}(w)]}(\nu)  N(\diffd \theta, \diffd \nu , \diffd r)    :  \theta \in [k_s\varepsilon,s] , 0\leq  w \leq v \right).
\]
 Indeed, we have
 \begin{linenomath}
 \begin{align*}
 \E \left[\int_{k_{s}\varepsilon}^{s} \right. &\left. \int_0^{Z^{\varepsilon,\delta}_{k_{s}\varepsilon,\theta} (v+h) \wedge m} W(\diffd \theta, \diffd u) \Bigg|{\mathcal{S}^{k_s\varepsilon,s}_v} \right]\\
 & =  \E \left[\int_{k_{s}\varepsilon}^{s} 
 \int_0^{Z^{\varepsilon,\delta}_{k_{s}\varepsilon,\theta}(v) \wedge m} W(\diffd \theta, \diffd u) \bigg|{\mathcal{S}^{k_s\varepsilon,s}_v} \right] + \E \left[\int_{k_{s}\varepsilon}^{s} \int_{Z^{\varepsilon,\delta}_{k_{s}\varepsilon,\theta}(v) \wedge m}^{Z^{\varepsilon,\delta}_{k_{s}\varepsilon,\theta}(v+h) \wedge m} W(\diffd \theta, \diffd u) \bigg|{\mathcal{S}^{k_s\varepsilon,s}_v} \right]\\
 & = \int_{k_{s}\varepsilon}^{s} 
 \int_0^{Z^{\varepsilon,\delta}_{k_{s}\varepsilon,\theta}(v)\wedge m} W(\diffd \theta, \diffd u) + \E \left[\int_{k_{s}\varepsilon}^{s} \int_{Z^{\varepsilon,\delta}_{k_{s}\varepsilon,\theta}(v) \wedge m}^{Z_{k_{s}\varepsilon,\theta}(v+h) \wedge m} W(\diffd \theta, \diffd u)\right] \\
 & = \int_{k_{s}\varepsilon}^{s} 
 \int_0^{Z^{\varepsilon,\delta}_{k_{s}\varepsilon,\theta}(v)\wedge m } W(\diffd \theta, \diffd u) ,
 \end{align*}	
 \end{linenomath}
 where the second equality can be checked using Remark \ref{indepnoises} and  the flow property of the  $\psi_{g( n\delta)}$-CSBPs \eqref{Zkn}. 
 In a similar way, we get that
 \begin{linenomath}
 \begin{equation*}
 \begin{aligned}
 \E\left[\int_{k_{s}\varepsilon}^{s} \int_0^{Z^{\varepsilon,\delta}_{k_{s}\varepsilon,\theta}(v+h)\wedge m}\int_0^{1} r \tilde{N}(\diffd \theta, \diffd \nu,\diffd r) \bigg|\mathcal{S}^{k_s\varepsilon,s}_v \right]= \int_{k_{s}\varepsilon}^{s} \int_0^{Z^{\varepsilon,\delta}_{k_{s}\varepsilon,\theta}(v) \wedge m} \int_0^{1} r \tilde{N}(\diffd \theta, \diffd \nu,\diffd r).
 \end{aligned}
 \end{equation*}	
 \end{linenomath}
 
 Using Jensen's inequality and then Doob's inequality  for the martingale $(M_v^W(s))_{v \geq 0}$, the first term on the right hand side of \eqref{bounddif} can thus be bounded as follows:
 \begin{equation}\label{BDGW}
    \sigma \E\left[ \sup\limits_{0 \leq v \leq Z^{\varepsilon,\delta}_{k_{s}\varepsilon}(x) \wedge m} |M_v^W(s)| \right]  \leq   \sigma\sqrt{ \E\left[ \sup\limits_{0 \leq v \leq  m} M_v^W(s)^2 \right] }  \leq  2 \sigma   \sqrt{m\varepsilon}.
 \end{equation}
 In a similar way, for the second term on the right hand side of \eqref{bounddif} we get 
 \begin{linenomath}
 \begin{equation}\label{BDGN}
 \begin{split}
 \E\left[\sup\limits_{0 \leq v \leq Z^{\varepsilon,\delta}_{k_{s}\varepsilon}(x) \wedge m} |M_v^{\tilde{N}}(s)| \right] & \leq 2 \sqrt{m\varepsilon}\sqrt{\int_0^1 r^2\Pi(\diffd r)}. \\ 
 	\end{split}
 	\end{equation}	
 \end{linenomath}
 Using the bounds \eqref{BDGW} and \eqref{BDGN} in  \eqref{bounddif} the desired result follows.

\end{proof}

\subsection{Proof of Proposition \ref{convergenceflow}} 
             
\begin{proof}[Proof of Proposition \ref{convergenceflow}]
 The difference $\zeta_t^{\varepsilon,\delta}(x):= Z_t (x) - Z^{\varepsilon,\delta}_t(x)$ satisfies 
 \begin{equation}\label{eq:y-z}
  \begin{aligned}
   \zeta_t^{\varepsilon,\delta}(x) = -\alpha & \int_0^t \left(Z_s (x) - Z^{\varepsilon,\delta}_s (x)\right)\diffd s
   + \sigma \int_0^t\int_0^{\infty}\left(\I_{\{u <Z_{s-} (x)\}} - \I_{\{u <Z^{\varepsilon,\delta}_{s-} (x)\}}\right) W(\diffd s,\diffd u) \\
                                             & +  \int_0^t \int_0^{\infty}\int_0^\infty \left(\I_{\{\nu <Z_{s-} 
   (x)\}} - \I_{\{\nu <Z^{\varepsilon,\delta}_{s-} (x)\}}\right)r\tilde{N}(\diffd s,\diffd \nu,\diffd r) \\
                                             & -\int_0^t \left[G(Z_s(x)) - G(Z^{\varepsilon,\delta}_s(x)) 
   \right]\diffd s - \int_0^t \int_0^{Z^{\varepsilon,\delta}_s(x)} g(v)\diffd v \diffd s\\
                                             & +  \int_0^t  \sum\limits_{n =0}^{ 
   n^{k_s}_x} g(n\delta) \left(Z_{k_s\varepsilon,s}((n+1)\delta \wedge Z^{\varepsilon,\delta}_{k_s\varepsilon}(x)) - Z_{k_s\varepsilon,s}(n\delta \wedge Z^{\varepsilon,\delta}_{k_s\varepsilon}(x)) \right)\diffd s. \\  \end{aligned}
 \end{equation}  
  (where $n^{k_s}_x$ corresponds to $n_x^k$ as  defined in \eqref{n}, with $k=k_s$ as in \eqref{k}).  Observe that
 \[
  Z_{k_s\varepsilon,s}((n+1)\delta \wedge Z^{\varepsilon,\delta}_{k_s\varepsilon}(x)) - Z_{k_s\varepsilon,s}(n\delta \wedge Z^{\varepsilon,\delta}_{k_s\varepsilon}(x))
                                              = \int_{n\delta \wedge Z^{\varepsilon,\delta}_{k_s\varepsilon}(x)} 
  ^{(n+1)\delta \wedge Z^{\varepsilon,\delta}_{k_s\varepsilon}(x)} \diffd_uZ_{k_s\varepsilon,s}(u),
  \]
 which we can substitute in the last term in the right-hand side of \eqref{eq:y-z} to get              \begin{equation}\label{eq1:y-z}
  \begin{aligned}
   \zeta_t^{\varepsilon,\delta}(x)= & - \alpha \int_0^t \left(Z_s (x) - Z^{\varepsilon,\delta}_s (x)\right)\diffd s\\
   & + \sigma \int_0^t\int_0^{\infty}\left(\I_{\{u <Z_{s-} (x)\}} - \I_{\{u < Z^{\varepsilon,\delta}_{s-} (x)\}}\right) W(\diffd s,\diffd u) \\
                                    & + \int_0^t \int_0^{\infty}\int_0^\infty \left(\I_{\{\nu <Z_{s-} (x)\}} 
    - \I_{\{\nu <Z^{\varepsilon,\delta}_{s-} (x)\}}\right)r\tilde{N}(\diffd s,\diffd \nu,\diffd r)\\ 
                                    & -\int_0^t \left[G(Z_s(x)) - G(Z^{\varepsilon,\delta}_s(x)) \right]\diffd s\\
                                    & - \int_0^t \sum\limits_{n =0}^{ n^{k_s}_x}
    \int_0^{Z^{\varepsilon,\delta}_{k_s\varepsilon}(x)}  \I_{\{n\delta<u\leq (n+1)\delta\}} \left[ g(u) - g(n\delta) \right]  d_u\, Z^{\varepsilon,\delta}_{k_s\varepsilon,s}(u)\diffd s\\
                                    & - \int_0^t \left[ \int_0^{Z^{\varepsilon,\delta}_s(x)} g(v)\diffd v 
    - \int_0^{Z^{\varepsilon,\delta}_{k_s\varepsilon}(x)}g(u)\diffd_u Z^{\varepsilon,\delta}_{k_s\varepsilon,s}(u)  \right]\diffd s. 
  \end{aligned}
 \end{equation}  
 Applying It\^o's formula to an approximation of the absolute value as in the proof of  \cite[Theorem 2.1]{DL} and dealing with second order terms \eqref{eq1:y-z} in a similar way, we deduce (see Appendix  \ref{proofboundFuLi} for details)
 that for each $m\in \NN$, 
 \begin{equation}\label{boundFuLi}
  \begin{split}
   \E|\zeta^{\varepsilon,\delta}_{t\wedge\tau_{m}}(x)| \leq & \left(|\alpha_0|+ \int_{1}^{\infty} r\Pi(dr) 
   + g(m)\right)\E\left[\int_{0}^{t} |\zeta^{\varepsilon,\delta}_{s\wedge\tau_{m}}(x)|\diffd s\right] \\
                                                            & + \E \left[ \int_0^{t\wedge \tau_m} \sum\limits_{n = 
   0}^{ n^{k_s }_x} \int_0^{Z^{\varepsilon,\delta}_{k_s\varepsilon}(x)}\I_{\{n\delta<u\leq (n+1)\delta\}}
   \left| g(u) - g(n\delta)\right| d_uZ^{\varepsilon,\delta}_{k_s\varepsilon,s}(u)\diffd s \right]\\
                                                            & + \E \int_0^{t\wedge\tau_{m}} 
   \left|\int_0^{Z^{\varepsilon,\delta}_{s}(x)} g(v)\diffd v - \int_0^{Z^{\varepsilon,\delta}_{k_s\varepsilon}(x)} g(v)\diffd v\right| \diffd s  \\
                                                            & + \E\int_0^{t\wedge\tau_{m}} \left| 
   \int_0^{Z^{\varepsilon,\delta}_{k_s\varepsilon}(x)}g(u)\diffd u 
   - \int_0^{Z^{\varepsilon,\delta}_{k_s\varepsilon}(x)}g(u)\diffd_u Z^{\varepsilon,\delta}_{k_s\varepsilon,s}(u)  \right| \diffd s.\\
  \end{split}
 \end{equation}
  Since $ \left| u- n\delta \right|   \I_{\{n\delta <u \leq (n+1)\delta\}}\leq \delta$ and  $ Z^{\varepsilon,\delta}_{k_s\varepsilon,s}( Z^{\varepsilon,\delta}_{k_s\varepsilon}(x) )= Z^{\varepsilon,\delta}_{s}(x)\leq m$ for $s\leq \tau_m$, the second term on the right hand side  of  \eqref{boundFuLi}  is bounded by $c(m)m\delta t$.  
Thanks to the a.e. differentiability of $u\mapsto g(u)$ (following from Hypothesis {\bf (H)} and Rademacher's theorem),  we can use integration by parts inside the last term of \eqref{boundFuLi} in order to rewrite
 \[
 \int_0^{Z^{\varepsilon,\delta}_{k_s\varepsilon}(x)} g(u)\diffd u = g(Z^{\varepsilon,\delta}_{k_s\varepsilon}(x)) Z^{\varepsilon,\delta}_{k_s\varepsilon}(x) - \int_0^{Z^{\varepsilon,\delta}_{k_s\varepsilon}(x)} u \, g'(u)\diffd u,
 \]
and 
 \begin{linenomath}
 \begin{align*}
 - \int_0^{Z^{\varepsilon,\delta}_{k_s\varepsilon}(x)} g(u) \diffd_uZ^{\varepsilon,\delta}_{k_s\varepsilon,s}(u) & =  - g(u)  Z^{\varepsilon,\delta}_{k_s\varepsilon,s}(u) \bigg|_{u=0}^{u=Z^{\varepsilon,\delta}_{k_s\varepsilon}(x)}  +  \int_0^{Z^{\varepsilon,\delta}_{k_s\varepsilon}(x)} Z^{\varepsilon,\delta}_{k_s\varepsilon,s}(u)g'(u)\diffd u\\
 & = - g(Z^{\varepsilon,\delta}_{k_s\varepsilon}(x))Z^{\varepsilon,\delta}_{s}(x) 
 + \int_0^{Z^{\varepsilon,\delta}_{k_s\varepsilon}(x)} Z^{\varepsilon,\delta}_{k_s\varepsilon,s}(u) g'(u)\diffd u.
 \end{align*}	
 \end{linenomath}
 This yields the following upper bound for the sum of the third and fourth lines in  \eqref{boundFuLi}:
 \begin{linenomath}
 \begin{equation*}
 \begin{aligned}
 \E & \int_0^{t\wedge\tau_{m}} \left(g(m) + g(Z^{\varepsilon,\delta}_{k_s\varepsilon}(x))\right) \left|Z^{\varepsilon,\delta}_{k_s\varepsilon}(x) - Z^{\varepsilon,\delta}_{s}(x)\right| \diffd s \\
 & + \E\int_0^{t\wedge\tau_{m}} \int_0^{Z^{\varepsilon,\delta}_{k_s\varepsilon}(x)}\left| u - Z^{\varepsilon,\delta}_{k_s\varepsilon,s}(u)  \right| g'(u)\diffd u  \diffd s .
 \end{aligned}
 \end{equation*}	
 \end{linenomath}
 
 Summarizing, we have 
 \begin{equation}\label{boundzeta}
  \begin{aligned}
   \E|\zeta^{\varepsilon,\delta}_{t\wedge\tau_{m}}(x)| \leq  \bigg(|\alpha_0| + & \int_{1}^{\infty} r\Pi(\diffd r) + g(m)\bigg)\E\left[ \int_{0}^{t} |\zeta^{\varepsilon,\delta}_{s\wedge\tau_{m}}(x)|\diffd s\right]  + c(m)m\delta t \\
                                                                             + \E & \int_0^{t\wedge\tau_{m}} 
   \left(g(m) + g(Z^{\varepsilon,\delta}_{k_s\varepsilon}(x))\right) \left| Z^{\varepsilon,\delta}_{k_s\varepsilon}(x)- Z^{\varepsilon,\delta}_{s}(x)\right| \diffd s \\     
                                                                             + \E & \int_0^{t\wedge\tau_{m}} 
   \int_0^{Z^{\varepsilon,\delta}_{k_s\varepsilon}(x)}\left| u - Z^{\varepsilon,\delta}_{k_s\varepsilon,s}(u)  \right| g'(u)\diffd u  \diffd s \\ 
                                                           \leq \bigg(|\alpha_0|+ & \int_{1}^{\infty} r\Pi(\diffd 
   r)+g(m)\bigg) \E\left[ \int_{0}^{t}|\zeta^{\varepsilon,\delta}_{s\wedge\tau_{m}}(x)|\diffd s\right]  + c(m)m\delta t\\
                                                                                + & 2g(m)\E 
   \int_0^{t\wedge\tau_{m}}\left|Z^{\varepsilon,\delta}_{k_s\varepsilon}(x) -  Z^{\varepsilon,\delta}_{k_s\varepsilon,s} (Z^{\varepsilon,\delta}_{k_s\varepsilon}(x))\right|\diffd s \\     
                                                                                + & c(m)\E\int_0^{t\wedge\tau_{m}} 
   \int_0^{Z^{\varepsilon,\delta}_{k_s\varepsilon}(x)}\left| u - Z^{\varepsilon,\delta}_{k_s\varepsilon,s}(u)  \right|\diffd u\diffd s. \\ 
  \end{aligned}
 \end{equation}
 We then use Lemma \ref{boundz-v} to bound the last two terms  in \eqref{boundzeta} and get that
 \begin{linenomath}
 \begin{equation*}
  \begin{aligned}
   \E|\zeta^{\varepsilon,\delta}_{t\wedge\tau_{m}}(x)| \leq & \left(|\alpha_0|+ \int_{1}^{\infty} r\Pi(\diffd r)+g(m)\right)\E\left[ \int_{0}^{t} |\zeta^{\varepsilon,\delta}_{s\wedge\tau_{m}}(x)|\diffd s\right] + c(m)m\delta t\\
                                                            & + \left(2g(m) + c(m)  \right)tm\varepsilon \left( 
   |\alpha_0| +\int_1^{\infty} r \Pi(\diffd r)  + g(m) \right) \\
                                                            & + \left(2g(m) + c(m) \right) 2 t \sqrt{m\varepsilon} 
    \left(\sigma + \sqrt{\int_0^1 r^2\Pi(\diffd r)}\right).
  \end{aligned}
 \end{equation*}
\end{linenomath}
 Since $\zeta_{s}^{\varepsilon,\delta}(x) \leq 2 m$ when  $0 < s \leq \tau_{m}$ by Lemma \ref{comparison property},  we can use  Gronwall's lemma and get
 \begin{linenomath}
 	\begin{equation*}
 	\begin{split}
 	\E|\zeta^{\epsilon,\delta}_{t\wedge\tau_{m}}(x)| \leq & \left [\delta c(m)m  + \varepsilon \left(2g(m) + c(m) \right)m\left( |\alpha_0| +\int_1^{\infty} r \Pi(\diffd r)  + g(m) \right)\right. \\
 	& + \left. \sqrt{\varepsilon} \left(2g(m) + c(m)\right) 
 	2\sqrt{m}\left(\sigma + \sqrt{\int_0^1 r^2\Pi(\diffd r)}\right)\right] te^{(|\alpha_0|+ \int_{1}^{\infty} r\Pi(\diffd r)+g(m))t}.
 	\end{split}
 	\end{equation*}
 \end{linenomath}
 Hence, $\E|\zeta^{\epsilon,\delta}_{t\wedge\tau_{m}}(x)|$ goes to zero when $(\delta,\varepsilon) \rightarrow ( 0,0)$ for each fixed $m>0$. 
 The desired result follows since $\tau_{m} \rightarrow \infty$ a.s. as  $m \rightarrow \infty$ 
\end{proof}
          
\section{Law and consistency of the local time  approximation} \label{ConsistLawGridTree}

The proofs of Propositions \ref{lawgrid} and \ref{convgrid} are based on several technical results mainly relying on excursion theory and on Proposition \ref{markRK}.
For easier reading, some proofs are differed to the Appendix. 

We will need in what follows the $(\F_t^{\rho})$-stopping times defined by 
\[
T_x^a : = \inf\{t \geq 0: L_t^a \geq x \}
\]
so that  $T_x^0=T_x$ for all $x \geq 0$.
\subsection{Poisson-snake excursions}\label{explosnakeexcurs}

We first  briefly  recall  some  facts of excursion theory in the case of  the snake process  $((\rho_t, {\cal N}_t) : \, t\geq 0) $.  
The reader is referred to \cite{DLG, DLGI} for details and general  background.  

Recall that the processes  $ (X_s - I_s: \, s\geq 0)=(\langle \rho_s , 1\rangle  \, : s\geq 0) $ and  $\rho$ share  the same excursion intervals $(\alpha_j,\beta_j)_{j\in J}$ away from their respective zero elements. 
Let us denote by $(\rho^j, {\cal N}^j)$ the excursion away from $(0,0)$ of $((\rho_t, {\cal N}_t ): \, t\geq 0) $ in  the interval $(\alpha_j,\beta_j)$:
\begin{linenomath}
\begin{equation*}
 \left\{
 \begin{aligned}
  \rho_s^j= & \rho_{\alpha_j + s} \quad & 0<s<  \beta_j- \alpha_j\\
  \rho_s^j= & 0 \quad & s\geq   \beta_j- \alpha_j
 \end{aligned}
 \right., \quad   \quad
 \left\{
 \begin{aligned}
  {\cal N}_s^j = & {\cal N}_{\alpha_j + s} \quad  & 0<s<  \beta_j- \alpha_j\\
  {\cal N}_s^j = & 0  \quad  & s\geq   \beta_j- \alpha_j\\
 \end{aligned}
 \right. \quad .
\end{equation*}
\end{linenomath}
Then, the  Markov process $((\rho_t, {\cal N}_t ): \, t\geq 0) $ can be rewritten as
\begin{equation}\label{rhoexcursion}  
  (\rho_t, {\cal N}_t)=\sum_{j\in J} \I_{\{ \alpha^j< t <\beta^j\}} (\rho^j_{t- \alpha^j} , {\cal N}^j_{t- \alpha^j}) \, , \quad  t\geq 0. 
\end{equation} 
The  point process  in $\R_+ \times \mathbb{D} (  \R_+, \M_f(\R_+)\times \M(\R^2_+)) $ given by  
\begin{equation}\label{PEX}
 \mathbf{ M} (\diffd \ell,\diffd  \rho \, , \diffd {\cal N}) :=  \sum_{j\in J} \delta_{(\ell^j,  \rho^j,\,  {\cal N}^j)} (\diffd \ell,\diffd  \rho \, , \diffd {\cal N}) ,
\end{equation}
where $(\ell^j= L^0_{\alpha_j})_{j\in J}$,  is Poisson of  intensity $\diffd x \otimes \mathbb{N}(\diffd  \rho \, , \diffd {\cal N})$ with $\mathbb{N}(\diffd  \rho \, , \diffd {\cal N}) = \mathbf{N}(\diffd  \rho) Q^{H(\rho)}( \diffd {\cal N} )$, $\mathbf{N}(\diffd  \rho)$  the excursion measure of the exploration process and $Q^{H(\rho)}$  the conditional (probability) law of  $({\cal N}_t:t\geq 0)$  given $\rho$. 
These facts follow from standard excursion theory, or are established in Section 4.1.4 in  \cite{DLG} (in particular in what concerns the description of $\mathbb{N}$). 
\footnote{ Notice that  results in \cite{DLG} are stated in terms of the process $((\rho_t, W_t): t\geq 0)$  with   $W_t$ corresponding here to the increasing ${\cal M}_{at}(\R_+)$-valued path $W_t =(h\mapsto {\cal N}_t([0,h], d\nu):  h\in [0,H_t) \, )$, which was denoted $\xi^{(t)}$ in Section 3.} 

\medskip
 
Reciprocally, given a Poisson point process $\mathbf{M}$ of intensity $\diffd x \otimes \mathbb{N}(\diffd \rho \, , \diffd {\cal N})$ and atoms $(\ell^j, \rho^j,\, {\cal N}^j)_{j\in J}$, a snake process $((\rho_t, {\cal N}_t ): \, t\geq 0)$ is uniquely determined by the relation \eqref{rhoexcursion}, with  $(\alpha_j,\beta_j)$ defined in terms of $\mathbf{M}$ by     
\[
 \alpha_j:= \sum_{k\in J:\ell^k < \ell^j } \zeta_k  \mbox{ and }\beta_j:= \sum_{k\in J:\ell^k \leq \ell^j } \zeta_k  \quad  \,, 
\]
with $\zeta_j:= \inf \{s\geq 0 :\, \rho^j_s=0\}$ the length of excursion $j$ for each $j\in J$ 
(the fact that the measure $\diffd L_t^0=\mathbf{1}_{\{ \rho_t=0 \}}\diffd L_t^0$ is singular with respect to $\diffd t$ is used to check this).  

\medskip
 
For a fixed height $a\geq 0$ we will also consider the process $(\rho^a_t,{\cal N}_t^a)_{t\geq 0}$ describing  ``what happens above  $a$'',  that is,  $\rho_t^a $ is defined as  in \eqref{rhoa}  and  
\[
{\cal N}_t^a (\diffd r, \diffd \nu)  := {\cal N}_{\tau_t^a}  (a+\diffd r, \diffd \nu),
\]
with  $\tau_t^a$ the stopping time   defined in \eqref{tauat}.  
We moreover denote by $(\rho^{(i)})_{i\in I}$ the excursions of  the process $\rho$ above height $a$ and by  $(\alpha^{(i)},\beta^{(i)})_{i\in I}$ the corresponding excursion intervals. 
More precisely, for each $i\in I$, we set
\begin{linenomath}
\begin{equation*}
\left\{
\begin{aligned}
\langle \rho_s^{(i)}, f\rangle = &  \int_{(a, \infty)} \rho_{ \alpha^{(i)}+s} (\diffd r) f (r - a)  \quad   & 0<s< \beta^{(i)}-  \alpha^{(i)} \\
\rho_s^{(i)}= & 0 \quad    & s\geq   \beta^{(i)}-  \alpha^{(i)}.
\end{aligned}
\right.
\end{equation*}	
\end{linenomath} 
These excursions are in one-to-one correspondence to the excursions away from $0$ of the process $\rho^a$ occurring at cumulated local times $L_{\alpha^{(i)}}^a=L_{\beta^{(i)}}^a$  at level $a$. 
We also introduce the excursions of ${\cal N}$ above level $a$, ``relative to  $a$'', given by  \footnote{ In the notation of \cite{DLG}, ${\cal N}^{(i)} $ would correspond to the ``increment''  of an excursion of the path-valued process $W_t$.}
\begin{linenomath}
\begin{equation*}
\left\{
\begin{aligned}
{\cal N}_s^{(i)} (\diffd r, \diffd \nu)= & {\cal N}_{  \alpha^{(i)}+s}  (a+\diffd r, \diffd \nu)  \quad  & 0 < s < \beta^{(i)}-   \alpha^{(i)}\\
{\cal N}_s^{(i)}= & 0  \quad  & s\geq   \beta^{(i)}-  \alpha^{(i)} \\
\end{aligned}
\right. .  
\end{equation*}	
\end{linenomath}         
\begin{remark}\label{relativexcursion}
Each excursion ${\cal N}^{(i)}$ away from $0$ corresponds to a segment of a unique ``parent excursion'' ${\cal N}^{j} $ such that $(\alpha^{(i)},  \beta^{(i)})\subset (\alpha^{j},  \beta^{j})$.
\end{remark}
   
Thus, $(\rho^{(i)},{\cal N}^{(i)})_{i\in I}$ are exactly the excursions of the process $\left( (\rho^a_t,{\cal N}_t^a)\, : t\geq 0 \right)$ away from $(0,0)$.  
The next result follows from \cite[Proposition 3.1]{DLGI} and the snake property (but see Appendix \ref{proofPPPa} for the sketch of a direct proof,  adapted from the one  of Proposition 4.2.3 in \cite{DLG}) :
   
\begin{lemma}[{\bf Snake-excursion process above a given level}]\label{PPPa}
For each $a\geq 0$, conditionally on  the sigma field ${\cal E}_a$ defined in \eqref{calEa}, the  point process  in $\R_+ \times \mathbb{D} (  \R_+, {\cal V}) $ given by  
\begin{equation}\label{PEXa}
 \mathbf{M}^a  (\diffd \ell,\diffd  \rho \, , \diffd {\cal N})  :=  \sum_{i\in I} \delta_{(\ell^{(i)},  \rho^{(i)},\,  {\cal N}^{(i)})}  (\diffd \ell,\diffd  \rho \, , \diffd {\cal N})  ,
\end{equation}
where $\ell^{(i)}=L^a_{\alpha^{(i)}}$ for all $i\in I$,  has the same law as the process \eqref{PEX} and is independent of ${\cal E}_a$. 
Consequently, $\left( (\rho^a_t,{\cal N}_t^a)\, : t\geq 0 \right)$ has the same law as $\left( (\rho_t,{\cal N}_t)\, : t\geq 0 \right)$ and  is independent of  ${\cal E}_a$.   
\end{lemma}
  
\subsection{Proofs of Proposition  \ref{lawgrid} and Lemma \ref{flowloctime}}\label{proofproplawgrid}
     
 The following result extending Lemma \ref{PPPa} is a cornerstone of this paper. Consider  an intensity  process $\intens$  and for a given  height $a>0$ define:
\[
\intens^a(t,h):= \I_{[0,a]}(h)\intens(t,h)  \, \mbox{ for all } (t,h)\in \R_+^2.
\] 
Denote  also 
\[ I^{ \intens^a } := \{ i\in I \, :  \,  m^{\intens^a}_{\alpha^{(i)}}= 0\}= \{ i\in I \, :  \,  m^{\intens^a}_{\alpha^{(i)}} ([0,a))= 0\},
\] 
where $(\alpha^{(i)}, \beta^{(i)})_{i \in I}$ are the excursion intervals above height $a$ defined in Section 6.1.
By the semi-snake property of $ m^{\intens^a}$, for every $i\in I$ one has $ m^{\intens^a}_s=0 \, \forall s\in [\alpha^{(i)},\beta^{(i)}]= 0$, thus $I^{\intens^a}$ is  the set of excursions above   $a>0$ whose parent excursion  ${\cal N}^{j} $  has no marks below that level  in the time interval $(\alpha^{(i)}-\alpha^j,\beta^{(i)}-\alpha^j)$. We have 
\begin{lemma}[\bf Pruning below a fixed level]\label{prunPEX}  The process  $(L^a_{T_x^a }(m^{\intens^a}): x\geq  0)$ is measurable with respect to the sigma-field ${\cal B}(\R_+)\otimes {\cal E}_a $, where ${\cal E}_a$ was defined in \eqref{calEa}.   Moreover,  conditionally on  ${\cal E}_a$, the point process in $\R_+ \times \mathbb{D}(\R_+, {\cal V}) $ given by 
\begin{equation}\label{PEXtheta}
 \sum_{i\in I^{\intens^a}} \delta_{(L_{\alpha^{(i)}}^a(m^{\intens^a}) ,  \rho^{(i)},\,  {\cal N}^{(i)})}  (\diffd \ell,\diffd  \rho \, , \diffd {\cal N}), 
\end{equation}
with $L_t^a(m^{\intens^a})$  the $ m^{\intens^a}$-pruned local time at height  $a>0$ and time $t\geq 0$,   is equal in law to the  process \eqref{PEX} stopped at $\ell=L^a_{\infty}(m^{\intens^a})$. 
\end{lemma}

\begin{remark}\label{timechangeAintens}
It is not hard to see that   the marked exploration process  associated with the excursion process \eqref{PEXtheta} is given by
\[
\left(\widehat{\rho}_t,\widehat{\N}_t  \ : t\in [0,A_{\infty}^{\intens^a})\right)= \left(\rho_{\widehat{C}_t },{\cal N}_{\widehat{C}_t }  \ : t\in [0,A_{\infty}^{\intens^a})\right) \, ,
\]
where
$\widehat{C}_t:=\inf \left\{s>0 \, : \,  A_{s}^{\intens^a}  >t \right\}$ and $ A_{s}^{\intens^a}: = \int_0^{s} \I_{\{m^{\intens^a}_{r}([0,a)) =0 \}} \diffd A_r^a $. 
Indeed, by the semi-snake property, $ \I_{\{ m^{\intens}_{u}([0,a)) =0\}}$ is constant on the increase  intervals of  $A_u^a= \int_0^u \I_{\{H_r>a\}} \diffd r $,  hence the  only excursions above $a$ contributing to the integral  in the definition of $\widehat{C}_t$  are those in the set $I^{\intens^a}$.\end{remark}

Lemma \ref{prunPEX} will follow from Lemma \ref{PPPa} and  an elementary fact about Poisson processes in $\R_+$, proved in  Appendix  \ref{proofnvarphi}  for completeness: 
\begin{lemma}\label{nvarphi}
Let $(N_{\ell}:\ell \geq 0)$ be a  standard Poisson process of parameter $\lambda>0$ with respect to a given filtration $({\cal K}_{\ell})_{\ell \geq 0}$  and  $E \subset \R_+$ be a random set such that  $ \I_{E}  $  is ${\cal B}(\R_+) \otimes  {\cal K}_0 $-measurable. 
For each  $x,y\geq 0$ set $\varphi_y := \int_{0}^{y} \I_{E}(\ell)\diffd \ell $   and define the stopping time $\phi_x:= \inf\{ y\geq 0:  \varphi_y>x\}\leq \infty$.
Then, conditionally on $ {\cal K}_0 $\, ,  $N^{E}:=  \left(\int_0^{\phi_x}\I_{E}(\ell)N(\diffd \ell): \, x \geq 0\right)$ is a  standard $({\cal K}_{\phi_{x}})_{x \geq 0}$-Poisson process of  parameter $\lambda$ stopped at  $ \int_{0}^{\infty} \I_{E}(\ell)\diffd \ell $. 
\end{lemma}

\begin{proof}[Proof of Lemma \ref{prunPEX}] 
In the proof we write $\intens^a=\intens$ for simplicity.  As in the proof of Proposition 4.2.3 in \cite{DLG} we introduce $\tilde{L}_t^a:= L_{\tilde{\tau}^a_t }^a$ with $\tilde{\tau}^a_t$ defined  above  after \eqref{calEa},    and  its right-continuous inverse $\gamma^{a}(r):= \inf\{s\geq 0 : \, \tilde{L}_s^a  >  r\}.$ 
 Let us rewrite the $m^{\intens}$-pruned local time at level $a$ in terms of the corresponding local time units. 
 Using the semi-snake property of $m^{\intens}$ in the second equality, we have
 \begin{linenomath}
 \begin{equation*}
  \begin{split}
   L_{T^a_x} ^a(m^{\intens} )=  &  \int_0^{T^a_x}  \I_{\{m^{\intens}_{v}=0\}}\diffd L^a_v
   =   \int_0^{\tilde{A}_{T^a_x}^a} \I_{\{m^{\intens}_{\tilde{\tau}_u^a}=0\}}\diffd \tilde{L}^a_u \\
   = & \int_0^{\tilde{L}^a_{\tilde{A}_{T^a_x}^a}} \I_{\{m^{\intens}_{\tilde{\tau}_{\gamma^a(\ell)}^a}([0,a)) =0\}}\diffd \ell
   = \int_0^x \I_{\{\mathbf{m}^{\intens}_{\ell} =0 \}}\diffd \ell,\\
  \end{split}
 \end{equation*}
\end{linenomath}
 where for all $\ell\geq 0$ we have set $\mathbf{m}^{\intens}_{\ell}:= m^{\intens}_{\tilde{\tau}^a_{\gamma^{a}(\ell)}}([0,a))$.  
  We have also used  the changes of variables $v= \tilde \tau_u^a$ and  $\ell = \tilde L^a_u$.
 The last equality above stems from the fact that, by definition of  $\tilde{\tau}^a_s$ (as the  right inverse   of $\tilde{A}^a_t=\int_0^t\I_{\{H_s\le a\}} \diffd s$) and since $L^a_s$ does not vary on intervals where $H_s>a$, one has $\tilde{L}^a_{\tilde{A}_{T^a_x}^a}= L^a_{\tilde{\tau}^a_{\tilde{A}_{T^a_x}^a}}= L_{T^a_x}^a.$ Given that $\tilde{\tau}^a_t<\infty$ a.s. for each $t\geq 0$, \eqref{aprox2} implies that 
 \[
  L_{\tilde{\tau}^a_s }^a=\lim_{\epsilon\to 0} \epsilon^{-1}\int_0^s \I_{\{a-\epsilon<H_{\tilde{\tau}^a_r }\}}\diffd r \quad  \text{ in probability uniformly on $s$ in compact sets.}
  \]   
  It follows  that the process $( \I_{\{\mathbf{m}^{\intens}_{\ell} =0 \}} :\ell \geq 0) $ is ${\cal B}(\R_+) \otimes  {\cal E}_a$-measurable. Thus, $(L_{T^a_x} ^a(m^{\intens} ):x\geq 0)$ is  ${\cal B}(\R_+) \otimes  {\cal E}_a$-measurable  as claimed.
 
 We now introduce the filtration  $\sigma \left(\{ \mathbf{M}^a  \left( [0,x ],  \diffd \rho, \diffd \N \right) : 0\leq x \leq   \ell \}\vee {\cal E}_a\right)_{\ell \geq 0}$  (with $\mathbf{M}^a$ given in   \eqref{PEXa}) and   its  right-continuous completion denoted $({\cal K}^a_{\ell})_{\ell \geq 0}$. 
 For each  Borel set  $S \subset  \mathbb{D} (  \R_+, {\cal V}) $ with $\mathbb{N}(S)<\infty$,   define a    $({\cal K}^a_{\ell})_{\ell \geq 0}$-  Poisson process in $\R_+$ by 
 \[
  N^{a,S} \left([0,\ell]\right) := \mathbf{ M}^a \left([0,\ell] \times S\right)\,, \quad \ell \geq 0 \, .
 \] 
 Setting $\phi_y =\inf \{x>0: \int_0^x \I_{E^{\intens}(\ell) }d\ell >y \}$ with $E^{\intens}:=\{\ell \in \R_+ :   \mathbf{m}^{\intens}_{\ell}=0\}$, Lemma \ref{nvarphi} yields that  conditionally on ${\cal K}^a_0$,
 \[
 N^{E^{\intens}}(S):= \left(\int_0^{\phi_x} \I_{E^{\intens}}(\ell)N^{a,S}(\diffd \ell): x \geq 0\right)
 \] 
 is a Poisson process in $\R_+$  of parameter $ \mathbb{N}(S)$, with respect to the time changed filtration  $({\cal K}^a_{\phi_l \geq 0})_{\ell \geq 0}$. 
 Moreover, the processes $N^{E^{\intens}}(S_1),..,N^{E^{\intens}}(S_n)$ do not have  simultaneous jumps if  the sets $S_1 ,..., S_n$ are  disjoint, hence they are independent from each other conditionally on ${\cal K}^a_0$. We conclude that, conditionally on ${\cal K}^a_0$, the point process  $\mathbf{M}^{a,\intens}$  on  $\R_{+} \times \mathbb{D}(\R_+,\cal{V})$ defined  by 
 \[
 \mathbf{M}^{a,\intens}  \left( [0,x ] \times S  \right) =  \int_{[0,\phi_x]\times S }  \I_{E^{\intens}}(\ell)  \mathbf{M}^a   (\diffd \ell, \diffd \rho, \diffd \N ) 
 \]
 is Poisson with intensity $\diffd x \otimes \mathbb{N}(\diffd \rho,\diffd \N)$. 
 It is not hard to see that this is exactly the  process \eqref{PEXtheta}.
\end{proof}

 \medskip

We are ready to give  the 
  
\begin{proof}[Proof of Proposition \ref{lawgrid}]

The fact that  for each $(k,n) \in  \NN^2 $  the process \eqref{L-L} is  ${\cal F}_{\overline{T}_{(n+1)\delta}}$-measurable  is immediate.
For all the remaining properties, we proceed by induction on $k\in \NN$. 

 In the case $k=0$, the first assertion is obvious since for every $n\in \NN$ one has  $
  L^{0}_{\overline{T}_{(n+1)\delta}^{0}}(\varepsilon,\delta)- L^{0}_{\overline{T}_{n\delta}^{0}}(\varepsilon,\delta) = L^{0}_{T_{(n+1)\delta}^{0}}(\varepsilon,\delta)- L^{0}_{T_{n\delta}^{0}}(\varepsilon,\delta)  =\delta
 $ a.s. and ${\cal E}_0$ is trivial.  For the second assertion, we observe first that in a similar way as  in the proof of Lemma \ref{prunPEX}, using \eqref{aprox2} one shows that 
$ L_{\tilde{\tau}^\varepsilon_s }^h=\lim_{\eta\to 0} \eta^{-1}\int_0^s \I_{\{\varepsilon-\eta<H_{\tilde{\tau}^\varepsilon_r }\leq h\}}\diffd r$ in probability uniformly on $s$ in compact sets, for each $h\in [0, \varepsilon].   $
Furthermore, in a similar way  as in that proof we also get that
\[  L_{T_x}^h(\varepsilon,\delta )=   \int_0^{\check{L}^h_{\tilde{A}_{T_x}^{\varepsilon}}}\I_{\{m^{\varepsilon,\delta}_{\tilde{\tau}_{\check{\gamma}(\ell)}^{\varepsilon}}([0,h)) =0\}}\diffd \ell \, ,
\]  
where $s\mapsto \check{L}_s^h:= L_{\tilde{\tau}^{\varepsilon}_s }^h$ and  its right-continuous inverse denoted $\check{\gamma}(\ell)$ are both ${\cal B}(\R_+) \otimes {\cal E}_{\varepsilon}$-measurable. Since  $\tilde{\tau}^{\varepsilon}_{\tilde{A}_{u}^{\varepsilon}}\geq  u$  with equality  if $u\geq  0$ is a  strict right-increase point of  $\tilde{A}_{u}^{\varepsilon}$,  we deduce that  
\[
 L^0_{\tilde{\tau}^{\varepsilon}_{\tilde{A}_{T_x}^{\varepsilon}}}=L^0_{T_x}=x \quad \text{ and that } \quad \tilde{A}_{T_x}^{\varepsilon}=\inf\{s\geq 0:  L_{\tilde{\tau}^{\varepsilon}_s }^0 >x\}. 
 \] 
 The latter random variable is thus ${\cal E}_{\varepsilon}$-measurable. It follows that  $L_{T^0_x}^h(\varepsilon,\delta)$ is $ {\cal E}_{\varepsilon}$-measurable  for every $x\geq 0$  and $h\in [0, \varepsilon]$ as required. 
The fact that  $
   \left( L^{h}_{T_{n\delta+z}^{0}} (\varepsilon,\delta)- L^{h}_{T_{n\delta}^{0}} (\varepsilon,\delta) :  h\in [0,\varepsilon], z\in [0,\delta]\right)
 $ is  ${\cal F}_{T_{(n+1)\delta}^{0}}$-measurable is obvious, and it has the asserted  conditional law thanks Proposition \ref{markRK} and the strong Markov property of  $(\rho,\N)$  applied at  ${\cal F}_{T_{n\delta}^{0}}$. This completes the  case $k=0$.  

 For the inductive step, assume the statements are valid for all integer smaller than or equal to $(k-1)$. Since for each $n\in \NN$, 
 \[
 L^{k\varepsilon}_{\overline{T}_{(n+1)\delta }^{k\varepsilon}} (\varepsilon,\delta)- L^{k\varepsilon}_{\overline{T}_{n\delta }^{k\varepsilon}} (\varepsilon,\delta)  = \delta\wedge( L^{k\varepsilon}_{\infty}(\varepsilon,\delta)-n\delta)_+   
 \]
and $ L^{k\varepsilon}_{\infty}=\lim\limits_{l\to \infty} L^{k\varepsilon}_ {\overline{T}_{l\delta }^{(k-1)\varepsilon}} $ is ${\cal E}_{k\varepsilon}$-measurable by the induction hypothesis, the first property  for the integer $k$  is immediate. We consider next the   marked exploration process   $((\rho_t , m^{\intens^{ k\varepsilon}}_t)  : t \geq 0)$  with adapted intensity $\intens^{ k\varepsilon}(t,h):=\intens^{\varepsilon,\delta}(t,h)\I_{\{h\leq k\varepsilon\}}$, that is,  $\intens^{ k\varepsilon}= (\intens^{\varepsilon,\delta})^{k\varepsilon}$   in the notation of the beginning of Section \ref{proofproplawgrid}  with $a=k\varepsilon$. Also, denote respectively  by $\widehat{L}^h_t $, $\widehat{T}_x$ and $\widehat{H}_t$ the local time processes  at level $h\geq 0$,  the inverse local time  process at level $0$ and the height process associated with the (possibly stopped)  process  $(\widehat{\rho}, \widehat{\N})= \left(\rho_{\widehat{C}_{\cdot} },{\cal N}_{\widehat{C}_{\cdot} } \right)$ defined in Remark \ref{timechangeAintens} (i.e. the marked exploration process pruned below level $a=k\varepsilon$). 
By already developed arguments using the approximation  \eqref{aprox2},  we can check that for every $t,h\geq 0$, on the event that $\widehat{C}_t<\infty$ we have
  \begin{equation}\label{hatLhatC}
  \widehat{L}_t^h=\int_0^{\widehat{C}_t}  \I_{\{m_s^{\varepsilon,\delta}([0,k\varepsilon)) =0 \}}  \diffd L^{k\varepsilon+h}_s
  \end{equation}
a.s. 
 (see  the proof of Proposition \ref{markRK} in \ref{proofmarkRK} for the details of a similar computation based on  \eqref{aprox1}).  Thus,  for each $y\in [0,L^{k\varepsilon}_{\infty}(\varepsilon,\delta))$ we have
  \[
  L^{k\varepsilon+h}_{\overline{T}^{k\varepsilon}_y}(\varepsilon,\delta) = \int_0^{\overline{T}^{k\varepsilon}_y} \I_{\{m_s^{\varepsilon,\delta}([k \varepsilon,k \varepsilon+h)) =0 \}}   \I_{\{m_s^{\varepsilon,\delta}([0,k\varepsilon)) =0 \}}  \diffd L^{k\varepsilon+h}_s =  \int_0^{A^{\intens^{k\varepsilon}}_{\overline{T}^{k\varepsilon}_y}} \I_{\{\widehat{m}^{\varepsilon,\delta}_u([0,h)) =0 \}}  \diffd \widehat{L}^{h}_u  
 \]
using the change of variable $s=\widehat{C}_u$, \eqref{hatLhatC}
  and  the notation $A^{\intens^{k\varepsilon}}_s=   \int_0^{s} \I_{\{m^{\intens^{k\varepsilon}}_{r}([0,k\varepsilon)) =0, \, H_r>k\varepsilon \}} \diffd r $ and 
\begin{equation} \label{widehatm}
\widehat{m}^{\varepsilon,\delta}_u([0,h)): = m^{\varepsilon,\delta}_{\widehat{C}_u}([ k\varepsilon, k\varepsilon+ h )) \, .
\end{equation} 
  Notice now that since $H_{\overline{T}^{k\varepsilon}_y}=k\varepsilon$, with the notation used in Lemma \ref{prunPEX} one has
  \begin{equation}\label{AintensbarThatT}
  A^{\intens^{k\varepsilon}}_{\overline{T}^{k\varepsilon}_y}=    \sum_{i\in I^{\intens^{k\varepsilon}}: \, L_{\alpha^{(i)}}^{k\varepsilon}(m^{\intens^{k\varepsilon}})\, \leq y}(\beta^{(i)}-\alpha^{(i)}) =\inf\{u\geq 0 : \widehat{L}^0_u>y\}= \widehat{T}_y. 
  \end{equation}
Thus, we have 
  \begin{equation}\label{LhatL}  
  L^{k\varepsilon+h}_{\overline{T}^{k\varepsilon}_y}(\varepsilon,\delta) =  \int_0^{ \widehat{T}_y} \I_{\{\widehat{m}^{\varepsilon,\delta}_u([0,h)) =0 \}}  \diffd \widehat{L}^{h}_u  . 
 \end{equation} 

 Therefore,  thanks to  Lemma \ref{PPPa} (with $a=k\varepsilon$),  the construction of the processes \eqref{L-L}  for the integer $k$   can be done conditionally on ${\cal E}_{k\varepsilon}$  using the process  $(\widehat{\rho}, \widehat{\N})$  up to the time  $\widehat{T}_y$ for every  $  y\leq  L^{k\varepsilon}_{\infty}(\varepsilon,\delta)$,  in the same way as in the case $k=0$  the processes  \eqref{L-L}   were constructed using the process $(\rho,{\cal N})$ up  to each time $T_y$. In particular,    the process \eqref{L-L}  is  measurable with respect to the  sigma-field 
   $ \widehat{{\cal E}}_{\varepsilon}$,  defined as  in  \eqref{calEa}  in terms of the right-continuous inverse $\tilde{\widehat{\tau}}^\varepsilon_t$ of the process $\tilde{\widehat{A}}^\varepsilon_t:=\int_0^t \I_{\{\widehat{ H}_s \leq \varepsilon\}} \diffd s$. We can then check that $ \widehat{{\cal E}}_{\varepsilon}\subseteq   {\cal E}_{(k+1)\varepsilon}$, noting that $\widehat{C}_{\tilde{\widehat{\tau}}^\varepsilon_t}$ is the right-continuous inverse of  the process 
   \[\tilde{\widehat{A}}^\varepsilon_{A^{\intens^{k\varepsilon}}_t} =\int_0^t
\I_{\{H_{\widehat{C}_{A^{\intens^{k\varepsilon}}_u}} < (k+1)\varepsilon, \, m^{\intens}_u([0,k\varepsilon))=0,\,  H_u >k\varepsilon\}}\diffd u
=\int_0^t  \I_{\{H_u < (k+1)\varepsilon, \,  m^\intens_u([0,k\varepsilon))=0, \,  H_u >k\varepsilon\}}\diffd u,\] which yields $
\widehat{C}_{\widehat{\tilde{\tau}}_t^\varepsilon}=  \tilde{\tau}^{(k+1)\varepsilon}_{\theta_t},
$ with $\theta_t: = \inf \left\{r> 0 :\int_0^r \I_{\{m^\intens_{\tilde{\tau}^{(k+1)\varepsilon}_v}([0,k\varepsilon))=0,\,  H_{\tilde{\tau}^{(k+1)\varepsilon}_v} >k\varepsilon\}}\diffd v > t \right\}.$ Moreover,   thanks to  Lemma \ref{PPPa}, the conditional law  of the  process \eqref{L-L}  given   ${\cal F}_{\overline{T}_{n\delta}^{k\varepsilon}}\bigvee {\cal E}_{k\varepsilon}$  is measurable with respect to $\widehat{{\cal F}}_{\widehat{T}_{(n+1)\delta}}\bigvee {\cal E}_{k\varepsilon}$, where $(\widehat{{\cal F}}_t)_{t\geq 0}$ is the filtration $({\cal F}_{\widehat{C}_t})_{t\geq 0}$. Thus,  the identification of that conditional law is done in a similar way as in the case $k=0$, reasoning in terms of $(\widehat{\rho}, \widehat{\N})$  conditionally on $ {\cal E}_{k\varepsilon}$. This achieves the  inductive step and concludes the proof. \end{proof}

   We end this subsection with the
   
   \begin{proof}[Proof of Lemma \ref{flowloctime}]   
  To ease notation we write $m=m^{\intens}$.  We  furthermore define  $\intens^{a+b}$ and $I^{\intens^{a+b}}$ as in the beginning of the present Section  \ref{proofproplawgrid}, but with $a+b$ instead of $a$. 
  
 Since $\breve{T}^{a+b}_{L^{a+b}_s(m)}\geq s$ for all $s\geq 0$, the result is obvious if  $\breve{T}^a_x= \infty$. We thus assume that $\breve{T}^a_x< \infty$. 
 Suppose moreover that $H_{t_0}=  a+b+c$ for some $t_0 \in (\breve{T}^a_x, \breve{T}^{a+b}_{L^{a+b}_{\breve{T}^a_x}(m)})$ such that $m_{t_0}([0,a+b))=0$. 
 Then, for some excursion  $i\in I^{\intens^{a+b}}$  above level $a+b$ whose parent excursion above $0$ has no marks below $a+b$, we must have $t_0\in  (\alpha^{(i)},\beta^{(i)})$ and  then  $L^{a+b}_{\alpha^{(i)}}(m)=  L^{a+b}_{t_0}(m). $ Since  $L^{a+b}_{t_0}(m)\leq L^{a+b}_{\breve{T}^a_x}(m)$ by definition of  $\breve{T}^{a+b}_{L^{a+b}_{\breve{T}^a_x}(m)}$,  we get $\alpha^{(i)}\leq \breve{T}^a_x$,  a contradiction since  $\breve{T}^a_x<t_0$ and  $H_{\breve{T}^a_x}=a<a+b$. 
Therefore,  we get 
\[
 \int_{\breve{T}^a_x}^{ \breve{T}^{a+b}_{L^{a+b}_{\breve{T}^a_x}(m)}}  \diffd L^{a+b+c}_t(m)= \int_{\breve{T}^a_x}^{ \breve{T}^{a+b}_{L^{a+b}_{\breve{T}^a_x}(m)}}  \I_{\{ H_t=a+b+c\}} \diffd L^{a+b+c}_t(m)= 0 
 \]
and we conclude the desired identity.
    \end{proof}
            
\subsection{Proof of Proposition  \ref{convgrid}}
   
The remainder of this section is devoted to the proof of Proposition \ref{convgrid}. 
The proof is quite technical and will need different types of estimates and localization  with respect to different variables (which is  one of the reasons why we  only get convergence in probability). 
Next result provides a ``continuity-type'' estimate for pruned local times. 
Recall that the $(\F_t^{\rho})$-stopping time $T^{M}$ was defined in \eqref{TM}.

\begin{lemma}\label{estimates}
Let $((\rho_t,m^{\intens_1}_t):t\geq 0)$ and  $((\rho_t,m^{\intens_2}_t):t\geq 0)$) be  marked exploration process  and $\tau$ be an arbitrary $(\F_t^{\rho})$-stopping time. 
\begin{itemize} 
 \item[i)] For  $h\geq 0$ we have that
 \[
  \E \left(  \left|  L_{\tau }^h(m^{\intens_1})-  L_{\tau }^h(m^{\intens_2})\right|\right)
     \leq \, \E\left( \int_0^{ \tau } \diffd L_t^h \, \int_0^{h} \left| \intens_1(t,r) -  \intens_2(t,r) \right|  \, \diffd r \right) .
\]  
 \item[ii)] If moreover  $\tau\leq T^M$ a.s. for  some $M\geq 0$ and for $i=1,2$  there is $\tilde{\intens}_i$  such that $ \intens_i=g(L(m^{ \tilde{\intens}_i}))$, we have
 \[
  \E\left( \int_0^{ \tau }\diffd L_t^h \, \int_0^{h} \left| \intens_1(t,r) -  \intens_2(t,r) \right|\, \diffd r \right) \leq   c(M) M  \,  \E \left(  \int_0^{\tau}   \I_{\{H_s\leq h \}}  \int_0^{H_s}  |\tilde{\intens}_1(s,r) - \tilde{\intens}_2(s,r)| \diffd r \diffd t \right)
 \]
 where $c(M)>0$ is a Lipschitz constant of $g$ in $[0,M]$. 
\end{itemize}
\end{lemma}
    
\begin{proof}
 {\em i)} We have  a.s.  for all $t\geq 0$,
 \[
  \left|  L_{t}^h(m^{\intens_1})-  L_{t}^h(m^{\intens_2})\right|  \leq \int_0^{t}  \left| m^{\intens_1}_s([0,h)) - m^{\intens_2}_s([0,h))\right| \diffd L_s^h  \leq  \int_0^{t}f_h(s) \diffd L_s^h  ,        
 \]
 with
 \[
 f_h(s)= \int_0^h \int_0^{\infty} \I_{\{ \intens_1 (s,r) \wedge  \intens_2 (s,r)) <  \nu \leq   \intens_1 (s,r) \vee  \intens_2 (s,r)  \} } \N_s(\diffd r, \diffd \nu).
 \]
 Applying conditionally on ${\cal F}^{\rho}_s={\cal G}_0^{(s)}$ the compensation formula we get for any $({\cal F}^{\rho}_t)_{t\geq 0}$-stopping time $\tau$ that
 \[
  \E\left(  f_h(s)  \I_{\{s\leq \tau \} } \bigg|   {\cal F}_s ^{\rho}  \right) \leq \ \int_0^h \E\left( | \intens_1 (s,r) - \intens_2 (s,r) | \bigg|  {\cal F}_s^{\rho}\right) \I_{\{s\leq \tau \} }\diffd r ,
 \]
 and the statement follows. 
 For point ii), proceeding as in the first part of the proof of Lemma \ref{lemma pre-Gronwall} we get
 \[
   \I_{{t\le \tau} } \int_0^{h}\left| \intens_1(t,r) -  \intens_2(t,r) \right|\, \diffd r  \, \leq  c(M)    \I_{{t\le \tau} } \int_0^t  \I_{\{ H_s\leq h \}}\big|m_s^{\tilde{\intens}_1}([0,H_s)) - m_s^{\tilde{\intens}_2 }([0,H_s))  \big|  \diffd s  \, , 
  \]
 hence
 \begin{linenomath}
  \begin{align*}
   \E\left( \int_0^{ \tau }\right.&\left. \diffd L_t^h \, \int_0^{h} \left| \intens_1(t,r) -  \intens_2(t,r) \right|\, \diffd r \right)\\
    \leq  & \, c(M)   \E \left( \int_0^{ \tau } \diffd L_t^h   \int_0^t    \I_{\{ H_s\leq h \}}  \big|m_s^{\tilde{\intens}_1}([0,H_s)) - m_s^{\tilde{\intens}_2 }([0,H_s))  \big|   \diffd s  \right) \\
           \leq  &\,   c(M) M  \,  \int_0^{\infty}  \E \left(      \I_{\{ s\le \tau, \, H_s\leq h \}}  
   \big|m_s^{\tilde{\intens}_1}([0,H_s)) - m_s^{\tilde{\intens}_2 }([0,H_s))  \big|    \diffd s  \right).
  \end{align*}
 \end{linenomath}
 Conditioning on ${\cal F}_s^{\rho}$ inside the last expectation, and following the last lines in the proof of Lemma \ref{lemma pre-Gronwall} yields the upper bound
 \[
   c(M) M \, \int_0^{\infty}\E \left(\I_{\{ s\le \tau, \, H_s\leq h \}}    \E\left( \int_0^{H_s}  |\tilde{\intens}_1(s,r) - \tilde{\intens}_2(s,r)|\diffd r \bigg|\cF_s^\rho\right)\right) \diffd s                
 \]
 as well as  the desired estimate.
\end{proof}

In the proof of Proposition \ref{convgrid} we will also need to compare accumulated pruned local times at heights that are not in the grid, with local times at heights in the grid right below. 
For local times pruned at fixed rate, this type of comparison can be deduced from a variant of the approximation \eqref{aprox1}. 
An extension to local-time dependent rates of the present framework is however not immediate, in part because the pruning rate is globally unbounded.  We will thus need to localize such approximation argument with respect to accumulated local times, in order to deal with pruning rates taking values  in compact intervals.

In order to ensure that the localization parameter can be removed, while at the same time, making the grid parameters go to $0$, we  will moreover need quantitative information about the speed at which approximations of local times such as \eqref{aprox1} converge. 
Such a result is the content of next lemma. 
Its technical proof, given in Appendix \ref{proofaproxlocprunvar}, relies on a snake variant of $L^2$-Poisson calculus developed in \cite{DLG} for the excursion of the exploration process. 
In a similar way as in that work, to avoid making additional integrability assumptions on the underlying  L\'evy process $X_t$, we need to also localize the exploration process with respect to its mass $ \langle \rho_t, 1\rangle$. We thus introduce an additional parameter $K>0$ and the stopping time 
\[
\tau^K:=\inf\{s>0 \, : \langle \rho_s, 1\rangle \geq K\}.
\] 
       
\begin{lemma}[{\bf Quantitative approximation  of variably pruned local  times at level $0$}]\label{aproxlocprunvar}
Consider a c\`agl\`ad   function   $\Theta:\R_+\to [0,\bar{\theta}]$  with $\bar\theta \geq  0$ and the marked exploration process $((\rho_t, m_t):t\geq 0)$ with $m_t= m^{\intens}_t$ associated with 
\[
\intens(t,h):= \Theta(L_t^0) \quad  \forall  \, (t,h)\in \R_2^+ \,.
\]
\begin{itemize}
 \item[a)] There exists an explicit nonnegative function $(\varepsilon,\bar\theta)\mapsto\hat{\cal C}(\bar\theta , K,\varepsilon)$ going to $0$ when $\varepsilon\to 0$ and increasing both in $\varepsilon$ and $\bar\theta$, such that for all $x\geq 0$: 
 \[ 
 \E \left[  \sup_{y\in [0,x]}\left|y -  \frac{1}{\varepsilon}  \int_0^{T_y} \I_{\{ 0<H_s \leq \varepsilon, \, m_s =0 \}} \diffd s \right| \I_{\{\tau^K> T_x\}}  \right] \leq \hat{\cal C}(\bar\theta ,K,\varepsilon) ( x+ \sqrt{x}).  
 \]
 \item[b)] We deduce that for   all $x\geq 0$,
 \[ 
 \E \left[  \sup_{t\in [0,T_x]}\left| L^{\varepsilon}_t(m)  - L_t^0 \right| \I_{\{\tau^K> T_x\}}  \right] \leq {\cal C}(\bar\theta ,K,\varepsilon) ( x+ \sqrt{x}) 
 \]
 for some explicit nonnegative  function $(\varepsilon,\bar\theta)\mapsto{\cal C}(\bar\theta , K,\varepsilon)$ with similar properties as $\hat{\cal C}(\bar\theta , K,\varepsilon)$. 
\end{itemize}
\end{lemma}
  
Recall the notation $k_{h}= \sup \{k \in \NN : k\varepsilon <  h \}$ and $\overline{n}_t^{k}=\sup\{n \in \NN: n\delta <L_t^{k\varepsilon}(\varepsilon,\delta)\}$ introduced for fixed $\varepsilon,\delta \geq 0$  in Sections \ref{gridapproxflow} and  \ref{gridapproxloctime}. 
From the previous result we deduce 

\begin{lemma} \label{aproxgridlemma}
Let us fix real numbers $ \varepsilon, \delta , K , M > 0$. 
\begin{itemize}
 \item[a)]  For all  $a>0$, 
  \[
  \E \left[\sup_{t \leq T^{M} \wedge \tau^K}  | L_t^{k_a\varepsilon}(\varepsilon, \delta) - L_t^a(\varepsilon, \delta)|  \right]	\leq {\cal C}(g(M ),K,\varepsilon) ( M+ \sqrt{M}) +\Gamma( M,K), 
  \]
  where $\Gamma( M,K)=  2M \p(\tau^K\leq  T_M ).$
 \item[b)] For each $t\geq 0$, 
  \begin{linenomath}
  \begin{equation*}
  \begin{split}
  \E\left[ \I_{\{t <T^{M}\wedge \tau^K \}}\I_{\{H_t \leq a\}} \right.& \left.  \int_0^{H_t} | \intens^{\varepsilon,\delta}(t,h)-\intens^*(t,h)| \diffd h \right]\\ 
  & \leq c (M) a \left(\delta + \Gamma( M,K) +  
  {\cal C}(g(M ),K,\varepsilon) ( M+ \sqrt{M})\right)e^{c(M)t}.
  \end{split}
  \end{equation*}	
  \end{linenomath}
 \end{itemize}
\end{lemma}
  
\begin{proof}
 a) We start noting that if $t \leq T^{M} $, we have $ L_t^{h}(\varepsilon, \delta)\leq L_t^{h}\leq M $ for every $h \geq 0$. 
 Thus, if $k_a=0$, since $T^M\leq T_M$, we have  
 \[
 \sup_{t \leq T^{M} \wedge \tau^K}  | L_t^{k_a\varepsilon}(\varepsilon, \delta) - L_t^a(\varepsilon, \delta)| \leq \sup_{t \leq T_M }  | L_t^0 - L_t^a(\varepsilon, \delta)| \mathbf{1}_{\{\tau^K> T_M\}} +2M \I_{\{\tau^K\leq T_M\}} 
 \]
 and the desired  inequality follows from part b) of Lemma \ref{aproxlocprunvar} with $\bar{\theta}=g(M)$. 
 To prove the property for arbitrary  $k_a= k$, we first observe  that, since $T^M\leq T^{k\varepsilon} _M$, 
 \begin{equation}\label{sLLsLL}
  \sup_{t \leq T^{M} \wedge \tau^K}  | L_t^{k\varepsilon}(\varepsilon, \delta) - L_t^a(\varepsilon, \delta)|  \leq \sup_{t \leq T^{k\varepsilon}_M }  | L_t^{k\varepsilon}(\varepsilon,\delta) - L_t^a(\varepsilon, \delta)| \I_{\{\tau^K> T_M\}} +2M   \I_{\{\tau^K\leq T_M\}},
 \end{equation}
 so it enough to bound the first term on the right hand side of \eqref{sLLsLL} by ${\cal C}(g(M) ,K,\varepsilon) (M+ \sqrt{M})$ in order  to obtain the desired inequality. 
 To that end, consider again the processes $\widehat{L}^r_t $, $\widehat{T}_x$ and $\widehat{H}_t$ associated with $(\widehat{\rho},\widehat\N)$, the snake process pruned below  level $k\varepsilon$  already used in the proof of Proposition \ref{lawgrid}. 
Observe that $\I_{\{\tau^K> T^{k\varepsilon}_M\}}\leq  \I_{\{\widehat \tau^K> \widehat{T}_{L^{k\varepsilon}_{T^{k\varepsilon}_M} (\varepsilon,\delta)}\}}$, where $\widehat{\tau}^K:=\inf\{s>0 \, : \langle \widehat{\rho}_s, 1\rangle \geq K\}$,   and that  (by similar arguments as in the proof of  Proposition \ref{lawgrid}) the supremum on the right-hand side of \eqref{sLLsLL}  satisfies 
 \[
 \sup_{t \leq T^{k\varepsilon}_M }  | L_t^{k\varepsilon}(\varepsilon,\delta) - L_t^a(\varepsilon, \delta)|\leq 
 \sup_{s \leq \widehat{T}_{L^{k\varepsilon}_{T^{k\varepsilon}_M}(\varepsilon,\delta)}}
  |\widehat{ L}_s^0- \int_0^s\I_{\{\widehat{m}_u^{\varepsilon,\delta}([0,\epsilon'))=0\}}\diffd \widehat{L}_u^{\epsilon'}|= \sup_{s \leq \widehat{T}_{L^{k\varepsilon}_{T^{k\varepsilon}_M}(\varepsilon,\delta)}}
 |\widehat{ L}_s^0- \widehat{L}_s^{\epsilon'} (\widehat m^{\varepsilon,\delta})|, 
 \] 
 where $\epsilon' = a - k\varepsilon\in [0,\varepsilon]$, and $\widehat m^{\varepsilon,\delta}$ is  given in \eqref{widehatm}. 
Thanks to Lemma \ref{prunPEX},  when taking expectation to the first summand  on the r.h.s.\ of \eqref{sLLsLL}, we can first condition on ${\cal E}_{k\varepsilon}$   and (conditionally) apply  part b) of Lemma \ref{aproxlocprunvar}. The result then follows since  $L^{k\varepsilon}_{T^{k\varepsilon}_M}(\varepsilon,\delta)\leq M$ and $y\mapsto {\cal C}(g(y) ,K,\varepsilon) ( y+ \sqrt{y})$ is increasing.  
  
% \textcolor{red}{The result for general $k_a=k+1$ is obtained by induction, describing the $\intens^{\varepsilon,\delta}$-pruned local times above level $(k+1)\varepsilon$ in terms of the excursions above $k\varepsilon$ not marked below that level,  as in the proof of Proposition \ref{lawgrid}.}
 
 \medskip
  
 b) Writing $\Delta_t^*:=  \int_0^{H_t} | \intens^{\varepsilon,\delta}(t,h)-\intens^*(t,h)| \diffd h $,  we have
 \begin{linenomath}
 	\begin{align}
 	\E \left[ \I_{\{t <T^{M}\wedge \tau^K  \}}\I_{\{H_t \leq a\}} \Delta^*_t \right]  \leq  &  \E \left[ \I_{\{t <T^{M}\wedge \tau^K  \}}\I_{\{H_t \leq a\}}\int_0^{H_t} \left|g( L_t^r(\varepsilon, \delta)) - g(L_t^r(m^*)) \right| \diffd r \right] \label{I} \\
 	& +  \E \left[ \I_{\{t 
 		<T^{M}\wedge \tau^K  \}}\I_{\{H_t \leq a\}} \int_0^{H_t} |g\left( \overline{n}_t^{k_r} \delta \right) - g\left(L_t^{k_r\varepsilon}(\varepsilon,\delta)\right)| \diffd r \right] \label{II}\\
 	& +  \E \left[\I_{\{t 
 		<T^{M}\wedge \tau^K  \}} \I_{\{H_t \leq a\}}  \int_0^{H_t} | g\left( L_t^{k_r\varepsilon}(\varepsilon, \delta)\right) - g\left(L_t^r(\varepsilon, \delta)\right)| \diffd r \right] \label{III}.
 	\end{align}
 \end{linenomath}
 Taking into account the relations   \eqref{eqL*} and \eqref{Lepsdeleq},  the term on the r.h.s.\  of  \eqref{I} is  bounded by 
 \[ 
  c(M)  \int_0^t  \E \left[ \I_{\{s <T^{M}\wedge \tau^K  \}}\I_{\{H_s \leq a\}}\Delta^*_s  \right] \diffd s \, ,
 \]
 thanks to Lemma \ref{lemma pre-Gronwall}. 
 From the definition of the integers $k_{h}=k_{h} (\varepsilon)$ and $\overline{n}_t^{k} = \overline{n}_t^{k}(\varepsilon,\delta)$, term \eqref{II} is bounded by $c(M)  a\delta$.
 Finally, by part a) term \eqref{III} is bounded by
 \[
  c(M)  \E  \left[  \int_0^a   \sup_{s \leq T^{M} \wedge \tau^K}|L_s^{k_r\varepsilon} (\varepsilon, \delta) - L_s^r(\varepsilon, \delta)| \diffd r \right] \leq  c(M) \, a \left( {\cal C}(g(M ),K,\varepsilon) ( M+ \sqrt{M}) + 2M \Gamma( M,K) \right).
 \]
 The statement follows by Gronwall's lemma. 
\end{proof}

\medskip
 
We can finally give the
\begin{proof}[Proof of Proposition \ref{convgrid}] 
 We fix $M>0$ and  $x\geq 0$ and consider the  $({\cal F}_t^{\rho})-$stopping time 
 \[ 
 \tau = T \wedge \tau^K \wedge T_x
 \] 
 with $T, K\geq 0$  constants to be fixed later in terms of $M$. 
 Thanks to the relations \eqref{eqL*} and \eqref{Lepsdeleq} we can apply inequality i) of Lemma \ref{estimates} to get for all $a\geq 0$ that
 \[
  \E \left[ |L^a_{\tau \wedge T^{M}}(\varepsilon,\delta) - L^a_{\tau\wedge T^{M}}(m^*)| \right]
  \leq  \E \left[\int_0^{\tau \wedge T^{M}} \diffd L_s^a  \int_0^{a} |g\left( \overline{n}_s^{k_r} \delta\right)- g\left(L_s^r(m^*)\right)| \diffd r  \right],
 \]
 from where
 \begin{linenomath} 
 \begin{align}
  \E \left[|L^a_{\tau \wedge T^{M}}(\varepsilon,\delta) - L^a_{\tau\wedge T^{M}}(m^*)| \right]
  \leq &  \, \E \left[\int_0^{\tau\wedge T^{M}}\diffd L_s^a \int_0^{a} |g\left( \overline{n}_s^{k_r} \delta\right)
   - g\left( L_s^{k_r\varepsilon}(\varepsilon,\delta)\right)| \diffd r \right]  \label{1}\\
       & +  \E \left[\int_0^{\tau \wedge T^{M}}\diffd L_s^a \int_0^{a} |g\left( L_s^{k_r\varepsilon} 
  (\varepsilon,\delta)\right)  -  g\left( L^{r}_{s} (\varepsilon,\delta)\right)| \diffd r \right] \label{2}\\
       & +  \E \left[\int_0^{\tau \wedge T^{M}}\diffd L_s^a \int_0^{a} |g(L_s^{r}(\varepsilon,\delta)) - 
  g(L^{r}_{s} (m^*))| \diffd r \right]  \label{3}.
 \end{align}
\end{linenomath}
 The right hand side of  \eqref{1} is bounded by $c(M) M\, a\delta$, by construction of the process $L(\varepsilon,\delta)$ and definition of $T^M$. 
 Term \eqref{2}  is bounded by
 \begin{linenomath}
  \begin{align*}
   c(M)\,\E & \left[\int_0^{\tau \wedge T^{M}}\diffd L_s^a   \int_0^{a} \sup_{u \leq T^{M} \wedge \tau^K}  | L_u^{k_r\varepsilon}(\varepsilon, \delta) - L_u^r(\varepsilon, \delta)|   \diffd r \right]\\
       \leq & \,  c(M) \, M  \, \E \left[ \int_0^a \sup_{u \leq T^{M} \wedge \tau^K}  | 
   L_u^{k_r\varepsilon}(\varepsilon, \delta) - L_u^r(\varepsilon, \delta)|   \diffd r \right]\\
       \leq & \,  a c(M) M \, ( {\cal C}(g(M ),K,\varepsilon) ( M+ \sqrt{M}) +\Gamma( M,K)),
  \end{align*}
 \end{linenomath}
 where in the last inequality we used part  a) of Lemma \ref{aproxgridlemma}.

 To bound \eqref{3}, we use inequality ii) of Lemma \ref{estimates} with  the stopping time $\tau$  in that statement replaced by, in  current proof's notation, $\tau\wedge T^M=  T \wedge \tau^K \wedge T_x\wedge T^M\leq T^M $, and  with  $\intens_1 =g(L(\varepsilon,\delta))$,  $\intens_2= g(L(m^*))$,  $\tilde{\intens}_1 = \intens^{\varepsilon,\delta}$  and $ \tilde{\intens}_2=\intens^*$. This gives us  the upper bound  
 \[
 c(M) M  \,  \E \left(  \int_0^{T\wedge T_x}   \I_{\{H_s\leq a \}} \I_{\{s <T^{M}\wedge \tau^K \}} \int_0^{H_s}  | \intens^{\varepsilon,\delta}(s,r) - \intens^*(s,r)| \diffd r \, \diffd t \right)
 \]
 for expression \eqref{3}. 
 By Lemma \ref{aproxgridlemma} b), the latter  is bounded by 
 \[
 a   \, T  c(M) ^2 \, M  \left(\delta + \Gamma( M,K)+  {\cal C}(g(M) ,K,\varepsilon) ( M+ \sqrt{M})\right) e^{c(M) T}. 
 \]
 Bringing all together,  we  have shown that
 \begin{linenomath}
 \begin{equation*}
 \begin{split}
 \E & \left[|L^a_{  T_x \wedge  (\tau^K  \wedge T \wedge T^{M})}(\varepsilon,\delta) - L^a_{T_x \wedge ( \tau^K  \wedge T \wedge T^{M})}(m^*)| \right]\\
 & \leq  a c(M) M (1+Tc(M)) \left(\delta + \Gamma( M,K)+  {\cal C}(g(M) ,K,\varepsilon) ( M+ \sqrt{M})\right) 
 e^{c(M)T} .
 \end{split}
 \end{equation*}	
 \end{linenomath}
 We now choose  for each $M>0$, $T:=M$. 
 Since $\tau^K\to \infty $ when $K\to \infty$, for each $M>0$ we can moreover find some  $K=K(M)$  going to $\infty$ with $M$ and such that $\p(\tau^{K(M) }\leq  T_M )\leq (c(M) M^3 (1+M\, c(M)) ) e^{c(M) M})^{-1}$. 
 With these choices, we have 
 \[
 c(M) M (1+M\, c(M))\Gamma( M,K(M))e^{c(M)  M}\leq M^{-1} \to 0
 \] 
 when $M\to \infty$, whereas  the sequence of stopping  times ${\cal T}_M:= M\wedge \tau^{K(M)} \wedge T^{M}$ goes a.s. to $\infty$. 
 Thus, for each $\eta>0$ and $M\geq 0$, 
 \begin{linenomath}
  \begin{align*}
   \p \left[|L^a_{T_x }(\varepsilon,\delta)\right.& \left. - L^a_{T_x  }(m^*)|>\eta  \right]\\
                                                  & \leq \p \left[|L^a_{T_x  \wedge {\cal T}_M} 
   (\varepsilon,\delta) - L^a_{T_x   \wedge{\cal T}_M}(m^*)|>\eta  \right] + \p(T_x >{\cal T}_M ) \\
                                                  & \leq \frac{  a c(M) M (1+M\, c(M))}{\eta}  \left(\delta + 
   \Gamma( M,K(M))+  {\cal C}(g(M ),K(M),\varepsilon) ( M+ \sqrt{M})\right) e^{c(M)M} \\
                                                  & \quad + \p(T_x >{\cal T}_M ). 
  \end{align*}
 \end{linenomath}
 Hence, 
 \begin{linenomath}
 \begin{equation*}
 \begin{split}
 \limsup\limits_{(\varepsilon,\delta)\to (0,0)} \p \left[|L^a_{T_x }(\varepsilon,\delta) - L^a_{T_x  }(m^*)|>\eta  \right] \leq & \frac{  a c(M) M (1+M\, c(M))}{\eta}  \Gamma( M,K(M))e^{c(M)M} \\
 & + \p(T_x >{\cal T}_M ).
 \end{split}
 \end{equation*}	
 \end{linenomath}
 
 Letting $M \to \infty$, we have shown that $\lim\limits_{(\varepsilon,\delta)\to( 0,0)} \p \left[|L^a_{T_x } (\varepsilon,\delta) - L^a_{T_x  }(m^*)|>\eta  \right]=0.$
\end{proof}

\bigskip     

{\bf Acknowledgments} J.B. acknowledges  support of grants ANR-14-CE25-0014 (ANR GRAAL)
and ANR-14-CE25-0013 (ANR NONLOCAL). M.C.F.'s research  was supported by a doctoral grant  of   BASAL-Conicyt Center for Mathematical Modeling (CMM) at University of Chile and a postdoctoral grant from the National Agency for Science and Technology Promotion, Argentina. 
She thanks also  ICM Millennium Nucleus NC120062 (Chile) for travel support, a doctoral mobility grant of the  French Embassy in Chile for a three months visit to  Laboratoire de Probabilit\'es et Mod\`eles Al\'eatoires of University Paris 6 and  hospitality of the latter  during her stay.  
J.F. acknowledges partial support of ICM Millennium Nucleus NC120062 and BASAL-Conicyt CMM. The authors thank Jean-Fran\c cois Delmas for  enlightening  discussions at the beginning of this work and  Victor Rivero and Juan Carlos Pardo for kindly providing some useful mathematical insights. They also thank two anonymous referees for their careful reading of the manuscript,   for valuable remarks  which allowed us to  clarify the presentation of our results, and for pointing out to us references [7] and [23]. 

\appendix
       
\section{Appendix} 
 
\subsection{Proof of Proposition \ref{markRK}}\label{proofmarkRK}

Recall from \cite[Theorem 0.3]{ADV} that $(\rho^\theta_t: t\geq 0) : = (\rho_{C^\theta_t}: t\geq 0)$, with  $C^\theta_t $ the right-continuous inverse of $A^\theta_t : =   \int_0^t \I_{\{m^{\theta}_s  =0\}} \diffd s$, 
has the same law as the exploration process associated with a L\'evy process of Laplace exponent \eqref{psitheta}. Denote by $\left(\bar{L}_s^a :t\geq 0 \right) $ its  local time at level $a$. 
Applying \eqref{aprox1} to  $\rho^{\theta}$ and performing the change of variable $C^{\theta}_r\mapsto u$, we deduce that  
\[
\bar{L}_t^a=  \lim\limits_{\epsilon\to 0}  \epsilon^{-1}\int_0^{C^{\theta}_t} \I_{\{a<H_u\leq a+\epsilon, \, m^{\theta}_u=0\}}\diffd u \quad \mbox{for all  }  t\geq 0, 
\]
in the  $L^1(\p)$ sense. Let us check that this limit is equal to $L_{C^{\theta}_t}^a(m^\theta)$.  Since
$C^{\theta}_t<\infty$ a.s. (see \cite{ADV}), we deduce from  \eqref{aprox1} applied to the original exploration process $(\rho_t: t\geq  0)$  that, on the interval $[0,C^{\theta}_t ]$,  the finite measures $\epsilon^{-1}  \I_{\{a-\epsilon<H_s\leq a\}}\diffd s$ converge weakly in probability  towards   $\diffd L_s^a$ as $\epsilon \to 0$. 
By Lemma \ref{total marks lsc}  the function $s\mapsto \I_{\{ m_s^{\theta} ([0,H_s))=0\}}$ has a.s.\ at most countably many discontinuities, therefore a.s.\ it is  continuous  a.e.\   with respect to the measure  $\I_{[0,C^{\theta}_t]}(s)\diffd L_s^a$.  We  thus deduce that: 
\begin{equation}\label{LbarL}
L_{C^{\theta}_t}^a(m^{\theta}) =  \P-\lim\limits_{\epsilon\to 0}  \epsilon^{-1}\int_0^{C^{\theta}_t} \I_{\{a<H_u\leq a +\epsilon, \, m^{\theta}_u=0\}}\diffd u =  \bar{L}_t^a  .
\end{equation}
 By right-continuity in $t\geq 0$ of the first and third expressions, $\left( \bar{L}_t^a :t\geq 0 \right)$ and $\left(L_{C^{\theta}_t}^a(m^{\theta})  :t\geq 0 \right)$ are indistinguishable for each $a\geq 0$. 
In particular, if we set $\bar{T}_x=\inf \{ s>0 : \,  \bar{L}_t^0 >x\}$, for each $a\geq 0$ we a.s. have
\[
\bar{L}_{\bar{T}_x}^a = L_{C^{\theta}_{\bar{T}_x}}^a(m^\theta)= L_{T_x}^a(m^\theta) 
\]
since $\bar{L}^0_{\cdot}=L_{C^{\theta}_{\cdot}}^0(m^\theta)$ (by \eqref{LbarL} with $a=0$) and $L^0_{\cdot}(m^{\theta})=L^0_{\cdot}$. 
The result follows from  the above identities and Theorem \ref{RKDL} applied to the local times of $\rho^{\theta}$.

\subsection{Proof of  Lemma \ref{weirdwnandpp} }\label{proofLemmaweirdwnandpp}

Write $M_t= (M^i_t)_{ i\leq n}$, $N_t= (N^k_t)_{ k\leq m} $, let  $\lambda
=(\lambda_1,\cdots,\lambda_n)\in \R^n$,  $\mu=(\mu_1,\cdots,\mu_m)\in \R^m$ and  set $\hat{\lambda}
=(\lambda_1 a_1^{-1},\cdots,\lambda_n a_n^{-1})$.  It is enough to show for all $t,s\geq 0$ that 
$$\E(\exp\{i \hat{\lambda} \cdot (M_{t+s}-M_t)+ i\mu \cdot(N_{t+s}-N_t)\}\vert {\cal S}_t)=
\exp\left\{-s \frac{|\lambda|^2}{2}+ s\sum_{k=1}^m b_k(e^{i\mu_k}-1)\right\}$$
(where ``$\cdot$'' is the Euclidean inner product and $i=\sqrt{-1}$). Equivalently,  we need to show that 
$$ \exp\left\{i \hat{\lambda} \cdot M_{t} +t\frac{|\lambda|^2}{2} + i\mu \cdot N_t -t\sum_{k=1}^m b_k(e^{i\mu_k}-1) \right\}$$
is a $({\cal S}_t)$- martingale. By  i),  
$ {\cal E}_t^{(M)}:=\exp\left\{i \hat{\lambda} \cdot M_{t} +t\frac{|\lambda|^2}{2} \right\}= \exp\left\{i \hat{\lambda} \cdot M_{t} -  \frac{i^2}{2} [ \hat{\lambda} \cdot M]_t \right\} $ clearly is a   $({\cal S}_t)$-martingale. On the other hand, writing ${\cal E}_t^{(N)}=   \exp\left\{ i\mu \cdot N_t -t\sum_{k=1}^m b_k(e^{i\mu_k}-1) \right\}$ we have
\begin{equation*}
\begin{split}
 {\cal E}_t^{(N)}- 1=\, &  \sum_{0<s\leq t}   {\cal E}_{s-}^{(N)}  \left( e^{ i\mu \cdot  \Delta N_s }-1\right) - \int_0^t \sum_{k=1}^m b_k(e^{i\mu_k}-1)  {\cal E}_{s}^{(N)}\diffd s   \\
 =\, &  \sum_{0<s\leq t}  \sum_{k=1}^m {\cal E}_{s-}^{(N)}  \left( e^{ i\mu_k }-1\right)  \Delta N^k_s  - \int_0^t \sum_{k=1}^m b_k(e^{i\mu_k}-1)  {\cal E}_{s-}^{(N)}\diffd s   \\
 \end{split}
\end{equation*}
 by the first property in ii) and since $\Delta N^k_s=0$ or $1$ for  each $k=1,\dots,m$.  The second property in ii) then grants  that
 $ {\cal E}_t^{(N)}- 1=\int_0^t  \sum_{k=1}^m b_k(e^{i\mu_k}-1)  {\cal E}_{s-}^{(N)} \diffd \tilde{N}^k_s$,  with $\tilde{N}^k_t=N^k_t-b_k t$,  is a  $({\cal S}_t)$- martingale too.  Since ${\cal E}^{(M)}$ and  ${\cal E}^{(N)}$ respectively are continuous and  pure jump  martingales, we have  $[{\cal E}^{(M)},{\cal E}^{(N)}]=0$. Thus,  ${\cal E}^{(M)}{\cal E}^{(N)}$  is  a  $({\cal S}_t)$- martingale as required.

\subsection{Proof of  estimate \eqref{boundFuLi}}\label{proofboundFuLi}
We follow ideas in  the proofs of \cite[Theorem 2.1]{DL}  and  \cite[Proposition 3.1]{FL}. Define a  function $\varrho$  on $\R_+$   by $\varrho(x):= [\sigma + \int_{0}^{1} r^{2}\Pi(\diffd r)]\sqrt{x}$. Notice  that for all  $y\leq x$ one has   
 \begin{equation}\label{boundc}
  \int_0^{\infty} \diffd \nu \int_0^1 \Pi(\diffd r) \int_0^1 \dfrac{r^2\I_{\{y<\nu<x\}} (1-t)}{\varrho(|x-y +tr\I_{\{y<\nu<x\}}|)^2}\diffd t \leq c := \dfrac{\int_{0}^{1}r^{2}\Pi(\diffd r)}{\left(\sigma + \int_{0}^{1} r^{2}\Pi(\diffd r)\right)^2}.
 \end{equation}
Consider a  real sequence $\{a_{j}\}_{j\geq 1}$ defined by   $a_{0} = 1$ and $a_{j}= a_{j-1} \,e^{-j[\sigma+ \int_{0}^{1} r^2\Pi(\diffd r)]^{2}}$, so that  $a_{j}\searrow 0$  and  $\int_{a_{j}}^{a_{j-1}} \varrho(z)^{-2 }\diffd z = j$.  
For each $j\geq 1$,  let $\psi_{j}$  be  a non-negative continuous function on $\R$ supported in $(a_{j} , a_{j-1})$ such that $0 \leq \psi_{j}(z)\leq 2j^{-1}\varrho(z)^{-2}$ and  $\int_{a_{j}}^{a_{j-1}} \psi_{j}(z)\diffd z = 1$. 
Define also non-negative twice continuously differentiable functions $\phi_{j}$ on $\R$ by 
\[
\phi_{j}(x) =\int_{0}^{|x|}\diffd y\int_{0}^{y}\psi_{j}(z)\diffd z, \quad x \in \R. 
\]
Notice that for each  $x\in\R$ it holds    for all $j\geq 1$ that  $ 0 \leq \phi'_{j}(x)\mbox{sign}(x) \leq 1 $,   $\phi_{j}''(x)\geq 0$ and $ \phi''_{j}(x)\sigma^2|x| \leq 2j^{-1}$. Moreover, we have $\phi_{j}(x) \nearrow |x|$  as $j \rightarrow \infty$.  

For any  $z,h, r, \nu ,x,y\geq 0$ and a differentiable function $f$, set now $l(r,\nu; ,x, y):  = r\left[\I_{\{\nu<x\}} - \I_{\{\nu<y\}}\right]$,  $\Delta_{h}f(z):= f(z+h)-f(z)$ and  $D_h f (\zeta) :=\Delta_h f(\zeta)  - f' (\zeta)h$. Using  \eqref{eq1:y-z}  and It\^{o}'s formula, we get 

\begin{equation}\label{eq:itophij}
 \begin{split}
  \phi_{j}(\zeta^{\varepsilon,\delta}_{t\wedge\tau_{m}}(x)) 
   =   & M_{t\wedge \tau_{m}}+  \alpha_0 \int_{0}^{t\wedge \tau_{m}} \phi'_{j}(\zeta^{\varepsilon,\delta}_{s}(x)) \zeta^{\varepsilon,\delta}_{s}(x) \diffd s \\
       & + \int_{0}^{t\wedge\tau_{m}}\int_{0}^{\infty}\int_1^{\infty} 
   \Delta_{l(r,\nu;Z_{s}(x),Z^{\varepsilon,\delta}_{s}(x))}\phi_{k}(\zeta^{\varepsilon,\delta}_{s}(x))\diffd \nu\Pi(\diffd r) \diffd s\\
       & + \frac{  \sigma^2}{2}\int_{0}^{t\wedge\tau_{m}}\phi''_{j}(\zeta^{\varepsilon,\delta}_{s}(x))
   \zeta^{\varepsilon,\delta}_{s}(x)  \diffd s\\
       & + \int_{0}^{t\wedge\tau_{m}} \int_0^{\infty} \int_0^1 D_{l(r,\nu; 
   Z_{s}(x),Z^{\varepsilon,\delta}_{s}(x))}\phi_{k}(\zeta^{\varepsilon,\delta}_{s})\diffd \nu\Pi(\diffd r) \diffd s \\
       & - \int_0^{t\wedge\tau_{m}} \phi'_{j}(\zeta^{\varepsilon,\delta}_{s}(x)) \left[G(Z_s(x)) - 
   G(Z^{\varepsilon,\delta}_s(x)) \right] \diffd s\\
       & - \int_0^{t\wedge \tau_m} \sum\limits_{n =0}^{ n^{k_s\varepsilon}_x} 
   \int_0^{Z^{\varepsilon,\delta}_{k_s\varepsilon}(x)}\phi'_{j}(\zeta^{\varepsilon,\delta}_{s}(x)) \I_{\{n\delta<u\leq (n+1)\delta\}}\left[ g(u) - g(n\delta) \right] \diffd _uZ_{k_s\varepsilon,s}(u)\diffd s,     
  \end{split}
\end{equation}
where $(M_{t \wedge \tau_{m}})_{t\geq 0}$ is a martingale. By properties of  $\phi'_j$,  the integrands   in the first and second lines are respectively bounded by $|\alpha_0||Z_{s}(x) - Z^{\varepsilon,\delta}_{s}(x)|$ and $   \int_1^{\infty} r \Pi(\diffd r)  |Z_{s}(x) - Z^{\varepsilon,\delta}_{s}(x)|$. The fact that 
\[
  0 \leq \int_{0}^{\infty}\int_0^{1} D_{l(r,\nu; x,y)}\phi_{j}(x-y)\diffd \nu\Pi(\diffd r)\leq \dfrac{2c}{j} 
\]
 following from estimate \eqref{boundc},   together with the  mentioned properties of $\phi''_{j}$  ensure that the terms on the third and fourth lines  vanish when $j\to \infty$. 
Taking expectation in \eqref{eq:itophij} and letting $j\rightarrow \infty$, we get the desired bound, noting that $g(m)$ is a Lipschitz constant for the function $ G$ in $[0,m]$. 

\subsection{Proof of Lemma \ref{PPPa}}\label{proofPPPa}
In a similar way as for the process $\left( (\rho_t,{\cal N}_t)\, : t\geq 0 \right)$, the trajectories of the process $\left( (\rho^a_t,{\cal N}_t^a)\, : t\geq 0 \right)$ are determined in a unique (measurable) way from the atoms of the point process \eqref{PEXa}. 
It is therefore enough to establish the first claim. 
To do so, one easily adapts first  the arguments of the proof of Proposition 4.2.3 in \cite{DLG} in order to prove  that, under the excursion measure $\mathbb{N}$, the process 
\begin{equation}\label{PEXaj}
 \sum_{i\in I^j} \delta_{(\ell^{(i)},  \rho^{(i)},\,  {\cal N}^{(i)})} ,
\end{equation} 
with $I_j:=\{ i \in I:  (\alpha^{(i)},\beta^{(i)}) \subset (\alpha^{j},\beta^{j})\}$  denoting the sub-excursions above level $a$ of the excursion away from $0$ labeled $j$, is  conditionally on ${\cal E}_a$ a Poisson point process of intensity  $\diffd x \I_{[L^a_{\alpha_j}, L^a_{\beta_j}  ]} \otimes \mathbb{N}(\diffd \rho \, , \diffd {\cal N})$ (our superscripts ``$(i)$'' correspond to  superscripts ``$i$'' in the statement of \cite{DLG}). 
The only difference is that, in the computation analogous to the one at end of that proof,  we must consider here test functions depending also on the  components $\ell^{(i)}$ of the atoms, and depending on the excursions of the spatial  component above level $a$ only through their increments respect to their values at  that $a$.  
Since $I$ is equal to the disjoint union $\bigcup_{j\in J} I_j$, one then concludes applying  conditionally on ${\cal E}_a$  the additivity of Poisson point measures.

\subsection{Proof of Lemma \ref{nvarphi}}\label{proofnvarphi} 
By standard properties of Poisson processes, for any nonnegative  predictable process $f$ and  stopping time $\tau$ in the given filtration, it holds that $\E \left[e^{-u \int_0^{\tau\wedge t} f(\ell)N(\diffd \ell) + \lambda \int_0^{\tau\wedge t } (1- e^{-u f (\ell)})\diffd \ell}\vert {\cal K}_0 \right] =1$  a.s. for all $u\geq 0$ and $t\geq 0$. 
If moreover $\tau$ is such that  $ \E \left[e^{ \lambda \int_0^{\tau} (1- e^{-u f (\ell)})\diffd \ell}\right]<\infty$, by dominated convergence one gets after letting $t\to \infty$ that $\E \left[e^{-u \int_0^{\tau} f(\ell)N(\diffd \ell) + \lambda \int_0^{\tau } (1- e^{-u f (\ell)})\diffd \ell}\vert {\cal K}_0 \right] =1$ a.s. 
Now, from our assumptions and  by a monotone class argument, we can check that $\I_{E}$ is a predictable process. Since  $e^{\lambda \int_0^{\phi_x} (1- e^{-u \I_{E}(\ell)})\diffd \ell}=e^{\lambda \int_0^{\phi_x} (1- e^{-u}) \I_{E}(\ell)\diffd \ell}=  e^{\lambda   (1- e^{-u}) (x  \wedge \int\I_{E}(\ell)\diffd \ell  )}  \leq   e^{\lambda   (1- e^{-u}) x } $,   from the previous  we conclude that for all $x\geq 0$, 
\[
\E \left[e^{-u \int_0^{\phi_x} \I_{E}(\ell)N(\diffd \ell) } \vert {\cal K}_0 \right]=  e^{-\lambda x  (1- e^{-u}) }
\mbox{ a.s.}
\] 
 Applying  conditionally on $ {\cal K}_{\phi_{y}}$ this argument   to increments  $\int_{\phi_y} ^{\phi_x} \I_{E}(\ell)N(\diffd \ell) $, the proof is complete.
 
\subsection{Proof of Lemma \ref{aproxlocprunvar}}\label{proofaproxlocprunvar}
   
a) The proof is inspired by that of   Lemma 1.3.2  in \cite{DLG}. 
We have
\begin{linenomath}
\begin{multline}\label{ESF} 
\E\left(\sup_{y\in [0,x]} \left|y -  \frac{1}{\varepsilon}  \int_0^{T_y}  \I_{\{ 0<H_s \leq \varepsilon, \, m_s =0 \}} \diffd s   \right|  \mathbf{1}_{\{\tau^K> T_x\}} \right)   \\
\leq \E\left(\sup_{y\in [0,x]}\left| \frac{1}{\varepsilon}  \int_0^{T_y}  \I_{\{ 0<H_s \leq \varepsilon, m_s =0,  \langle \rho_s, 1\rangle  \leq K\}}  \diffd s -  \frac{1}{\varepsilon} \E\left( \int_0^{T_y} \I_{\{ 0<H_s \leq \varepsilon, m_s =0,  \langle \rho_s, 1\rangle  \leq K\}}  \diffd s \right) \right|  \right) \\
+  \E\left(\sup_{y\in [0,x]}  \bigg| \frac{1}{\varepsilon} \E\left( \int_0^{T_y} \I_{\{ 0<H_s \leq \varepsilon, m_s =0,  \langle \rho_s, 1\rangle  \leq K\}}  \diffd s \right)  - y\bigg| \right)  
\end{multline}	
\end{linenomath}
The time integral in the above expressions can be written in terms of the excursion point process \eqref{PEX} as follows: 
\begin{equation}\label{int=sum}
 \int_0^{T_y} \I_{\{ 0<H_s \leq \varepsilon, m_s =0,\, \langle \rho_s, 1\rangle  \leq K\}}\diffd s = \sum_{j\in J_y} \int_0^{\zeta^j} \I_{\{ 0<H^j_s \leq \varepsilon, m^j_s =0,\, \langle \rho_s^j, 1\rangle \leq K\}}\diffd s\, ,
\end{equation}
where $J_y:=\{ j\in J: \, \ell^j \leq y\}$,   $H_s^j=H_{\alpha^j+s}$, $m_s^j =\N^j_s (\, \cdot \times [0, \Theta(\ell^j) )) $ and $\zeta^j$  is the length of the excursion  labelled $j$. 
By compensation, the desintegration $\mathbb{N}(\diffd \rho \, , \diffd {\cal N}) = \mathbf{N}(\diffd \rho) Q^{H(\rho)}( \diffd {\cal N} )$ and the very definition of the snake $(\rho,\N)$, we get
\begin{linenomath}
\begin{equation*}
 \begin{split}
  \E & \left( \int_0^{T_y} \I_{\{ 0<H_s \leq \varepsilon, m_s =0,\, \langle \rho_s, 1\rangle  \leq K\}}  \diffd s   \right)\\
   = & \int_0^y \diffd \ell  \, \mathbb{N}  \left( \int_0^{\zeta}    \mathbf{1}_{\{ 0<H_s(\rho) \leq \varepsilon, \,  \N_s (\, \cdot \times [0,\Theta(\ell )))=0\, ,   \langle \rho_s, 1\rangle  \leq K\}}  \diffd s   \right)\\
   = & \int_0^y \diffd \ell  \, \mathbf{N}  \left( \int_0^{\zeta}  e^{ -  \Theta(\ell ) H_s(\rho)} \I_{\{ 0<H_s(\rho) \leq \varepsilon\, ,  \langle \rho_s, 1\rangle  \leq K\}}  \diffd s   \right),
 \end{split}
\end{equation*}
\end{linenomath}
with $\zeta$ the length of the canonical excursion and $(H_s(\rho): 0\leq s\leq \zeta)$ its height process. 
Thus, the second term in  the r.h.s.  of \eqref{ESF} is bounded  by  
\[
\int_0^x \diffd \ell  \, \E  \left[ \bigg| \varepsilon^{-1} \mathbf{N}  \left( \int_0^{\zeta}  e^{ -  \Theta(\ell ) H_s(\rho)  }  \mathbf{1}_{\{ 0<H_s(\rho) \leq \varepsilon \,  ,  \langle \rho_s, 1\rangle  \leq K\}}  \diffd s   \right)- 1\bigg|\right] . 
\]
Using Proposition 1.2.5 in \cite{DLG} to compute the integral with respect to $\mathbf{N}$ for each $\ell\in [0,x]$, the latter expression  is seen to be equal to
\begin{equation}\label{upbound1}
 \int_0^x \diffd \ell  \,   \left[ 1 -\frac{1}{\varepsilon}  \int_0^{\varepsilon} e^{-\alpha b} e^{ -  \Theta(\ell) b  }  \p(  S_b \leq K  ) \diffd b \right]  \leq    x \left[ 1- \frac{ 1- e^{-(\alpha+\bar{\theta} ) \varepsilon}}{(\alpha+\bar{\theta} ) \varepsilon}  \p(  S_\varepsilon \leq K  ) \right], 
\end{equation}
where $(S_b)_{b\geq 0}$ is a subordinator of Laplace exponent $\widehat{\psi}(\lambda):= \frac{\sigma^2}{2}\lambda + \int_0^\infty(1 - e^{-\lambda} )\Pi([r, \infty))\diffd r$ which does not depend on the drift coefficient $\alpha$ of the underlying L\'evy process $X$.
In particular, the expression on the r.h.s.\  of \eqref{upbound1} goes to $0$ with $\varepsilon$. 
    
Let now $({\cal Q}_\ell)_{\ell \geq 0}$ denote the right continuous completion of the filtration $(\sigma(\mathbf{ M}([0,x],d\rho,d\mathcal{N}): 0 \leq x \leq \ell))_{\ell \geq 0}$, with $\mathbf{ M}$ the point process defined in \eqref{PEX}. 
Since the first term on the  r.h.s.\  of \eqref{ESF} is the expected supremum of the absolute value of  a $({\cal Q}_{\ell})_{\ell \geq 0}$-martingale, we can  bound it using   BDG inequality  by some universal  constant $C_1$ times
\[
\sqrt{ Var \left[ \frac{1}{\varepsilon}  \int_0^{T_x}  \mathbf{1}_{\{ 0<H_s \leq \varepsilon, m_s =0,  \langle \rho_s, 1\rangle  \leq K\}}   \diffd s   \right] }. 
\]
Written in terms of the excursion Poisson point  process \eqref{PEX},   the previous quantity reads
\[ 
\sqrt{ Var \left[  \frac{1}{\varepsilon}  \sum_{j\in J_x} \int_0^{\zeta_j}    \mathbf{1}_{\{ 0<H^j_s \leq \varepsilon, m^j_s =0,\,    \langle \rho_s^j, 1\rangle  \leq K\}}  \diffd s  \right] }
\] 
and can be estimated by the same arguments  as  in  the proof of  Lemma 1.3.2  of \cite{DLG}, as follows (see also the proof of Lemma 1.1.3  therein for  details on related arguments):
\begin{equation}\label{upbound2}
 \begin{split}
  Var \left[ \frac{1}{\varepsilon}\int_0^{T_x} \I_{\{ 0<H_s \leq \varepsilon, m_s =0,  \langle \rho_s, 1\rangle  \leq K\}}\diffd s\right] =\, & \frac{x}{\varepsilon^2} \mathbb{N}\left(\left(  \int_0^{\zeta}  \I_{\{ 0<H_s \leq \varepsilon, m_s =0,  \langle \rho_s, 1\rangle  \leq K\}}\diffd s  \right)^2  \right) \\
                          \leq &  \, \frac{x}{\varepsilon^2} \mathbf{N}\left(\left(  \int_0^{\zeta} \I_{\{ 0<H_s 
  \leq \varepsilon,  \langle \rho_s, 1\rangle  \leq K\}}   \diffd s  \right)^2  \right)  \\
                          \leq & \,  2 x\E( X_{L^{-1}(\varepsilon)}\wedge K),\\ 
 \end{split}
\end{equation}
where $ \varepsilon\mapsto X_{L^{-1}(\varepsilon)}$ is a subordinator of Laplace exponent  $\exp\left(-  t\left( \widehat{\psi}(\lambda)+ \alpha  \right)\right)$. 
That is, the same subordinator $S$ as above, but killed  at an independent exponential time of parameter $\alpha$. 
Thus, we have $\E( X_{L^{-1}( \varepsilon)}\wedge K)\leq \E(S_ \varepsilon\wedge K) + K(1-e^{-\alpha  \varepsilon}) \to 0$ as $ \varepsilon\to 0$.  
The statement now follows by bringing together \eqref{ESF},  \eqref{upbound2} and \eqref{upbound1}  with the r.h.s. of the latter replaced by its supremum over $\varepsilon'\in [0,\varepsilon]$, which is an increasing function of $\varepsilon$ going  to $0$ as $\varepsilon\to 0$, as required.

\medskip

b)  Observe first that by continuity of $s\mapsto L_s^0$
\[
\forall t < T_x, \forall n\geq \dfrac{1}{x}, \exists y_n,z_n\in [\dfrac{1}{n},x] \text{ such that } T_{y_n}\leq  t\leq T_{z_n} \text{ and }  z_n- \dfrac{1}{n} \leq  L_t^0 \leq   y_n+ \dfrac{1}{n}.
\] 
We deduce that  
\[
 |L_t^0- L_t^{\varepsilon}(m) |	\leq 
 \begin{cases}
   y_n-  L^{\varepsilon}_{T_{y_n}}(m) + \dfrac{1}{n} & \text{ if } \, L_t^{\varepsilon}(m)< L_t^0\\
                                                                                                 \\
   L^{\varepsilon}_{T_{z_n}}(m) -z_n  + \dfrac{1}{n} & \text{ if } \, L_t^{\varepsilon}(m)> L_t^0
 \end{cases} .
\] 
Therefore, we have $\sup\limits_{t\in [0,T_x]}| L_t^0- L_t^{\varepsilon}(m) |\leq  \sup\limits_{y\in [0,x]} \left| L^{\varepsilon}_{T_y}(m)- y \right|$ and it is enough to establish the required upper bound  for the quantity
\begin{equation}\label{LTy-y}
 \E \left[  \sup_{y\in [0,x]}
    \left| L^{\varepsilon}_{T_y}(m)  - y \right|  \mathbf{1}_{\{\tau^K> T_x\}}  \right].
\end{equation}
We have
\begin{linenomath}
	\begin{equation*}
	\begin{split}
	L^\varepsilon_{T_y}(m)-y =  &  \left[   L^\varepsilon_{T_y}(m) -\frac{1}{\varepsilon} \int_0^{T_y}  
	\I_{\{ \varepsilon<H_s \leq 2\varepsilon, m_s((0,\varepsilon))=0\}} \diffd s \right] + 2\left[ \frac{1}{2\varepsilon}  \int_0^{T_y} \I_{\{ 0<H_s \leq 2\varepsilon,m_s((0,\varepsilon))=0\}} \diffd s - y  \right]  \\
	& + \left[y-  \frac{1}{\varepsilon}  \int_0^{T_y}  \mathbf{1}_{\{ 0<H_s \leq 
		\varepsilon,m_s((0,\varepsilon))=0\}} \diffd s   \right],
	\end{split}
	\end{equation*}
\end{linenomath}
and the absolute value of the second term on the right hand side is bounded by 
\[
2 \left| \frac{1}{2\varepsilon}  \int_0^{T_y} \I_{\{ 0<H_s \leq 2\varepsilon,m_s((0,2\varepsilon))=0\}} \diffd s - y  \right| + 2 \left| \frac{1}{2\varepsilon}  \int_0^{T_y} \I_{\{ 0<H_s \leq 2\varepsilon\}} \diffd s - y  \right|, 
\]
thanks to the inequalities $\I_{\{ 0<H_s \leq 2\varepsilon\}}\geq \I_{\{ 0<H_s \leq 2 \varepsilon,m_s ((0,\varepsilon))=0\}} \geq \I_{ \{0<H_s \leq 2\varepsilon,m_s((0,2\varepsilon))=0\}}$. 
It follows from part a) that  expression  \eqref{LTy-y} is upper bounded by
\begin{linenomath}
\begin{multline}\label{CCCTy} 
 \left( 2 \hat{\cal C}(\bar\theta ,K,2\varepsilon) + 2\hat{\cal C}(0  ,K,\varepsilon) + \hat{\cal C}(\bar\theta ,K,\varepsilon) \right)( x+ \sqrt{x})  \\
 + \E \left[  \sup_{y\in [0,x]} \left|L^\varepsilon_{T_y}(m) -\frac{1}{\varepsilon} \int_0^{T_y} \I_{\{ \varepsilon < H_s \leq 2\varepsilon, m_s((0,\varepsilon))=0\}} \diffd s \right| \I_{\{\tau^K> T_x\}}  \right]   
\end{multline}
\end{linenomath}
and it only remains us  to bound this last expectation. 
Notice to that end that the inner supremum can be written in terms of  the Poisson excursions point process pruned below the level $a=\varepsilon$ as considered  in  \eqref{PEXtheta}, namely
\begin{equation}\label{PEXmvariabl}
 \sum_{i\in I^{\intens^\varepsilon}} \delta_{(L_{\alpha^{(i)}}^\varepsilon(m) ,  \rho^{(i)},\,  {\cal N}^{(i)})}  \, ,
\end{equation}
where $I^{\intens^{\varepsilon}}=  \{ i\in I \, :  \, m^{\intens^{\varepsilon}}_{\alpha^{(i)}}([0,\varepsilon))= 0\}$.   
As before, let  $\widehat{L}^0_t$, $\widehat{T}_x$, $\widehat{H}_t$ and $\widehat{\tau}^K$ denote the corresponding local time at $0$, inverse local time at $0$, height process, and the stopping time $\widehat{\tau}^K:=\inf\{s>0 \, : \langle \widehat{\rho}_s, 1\rangle \geq K\}$, all of them now  associated with $(\widehat{\rho}, \widehat\N)$, the snake process associated with \eqref{PEXmvariabl}.

Then, writing in a similar way as in \eqref{int=sum} the time integral in \eqref{CCCTy}, that is, as  a sum of integrals over (now) non marked excursion intervals above level $\varepsilon$,  we get
\[ 
L^\varepsilon_{T_y}(m) -\frac{1}{\varepsilon} \int_0^{T_y}  \I_{\{ \varepsilon<H_s \leq 2\varepsilon, m_s((0,\varepsilon))=0\}} \diffd s= \widehat{L}^0_{\widehat{T}_{L^\varepsilon_{T_y}(m)}} -\frac{1}{\varepsilon} \int_0^{\widehat{T}_{L^\varepsilon_{T_y}(m)}}  \I_{\{ 0<\widehat{H}_s \leq \varepsilon\}} \diffd s. 
\] 
Since $T_x\geq  \widehat{T}_{L^\varepsilon_{T_x}(m)}$ and $\sup\limits_{t\leq T_x} \langle \rho_t, 1\rangle \geq  \sup\limits_{t\leq \widehat{T}_{L^a_{T_x}} (m)} \langle \widehat{\rho}_t, 1\rangle $, the expectation in \eqref{CCCTy} is seen to be less than
\[
 \E \left[  \sup_{z\in [0, L^\varepsilon_{T_x}(m) ]} \left|  z  -\frac{1}{\varepsilon} \int_0^{\widehat{T}_z}  \I_{\{ 0<\widehat{H}_s \leq \varepsilon\}} \diffd s \right| \I_{\{\widehat\tau^K>   \widehat{T}_{L^\varepsilon_{T_x}(m)} \}}  \right]\\\leq \hat{\cal C}(0  ,K,\varepsilon) \E \left[  L^\varepsilon_{T_x}(m)+\sqrt{L^\varepsilon_{T_x}(m)}  \right], 
\] 
 by applying,   thanks to Lemma \ref{prunPEX},  the previous part a) (with $m=0$ or equivalently $\bar\theta =0$) conditionally on ${\cal E}_{\varepsilon}$. 
With the obvious bounds $L^\varepsilon_{T_x}(m)\leq L^\varepsilon_{T_x}$ a.s., $\E\left[ \sqrt{L^\varepsilon_{T_x}}\right] \leq \sqrt{\E\left[ L^\varepsilon_{T_x}\right]}$ and the equalities $\E(   L^{\varepsilon}_{T_x} )= x \mathbf{N}\left( L^{\varepsilon}_{\zeta} \right)= x e^{- \alpha {\varepsilon}} $ following from Corollary 1.3.4 in \cite{DLG},  the desired  result is seen to hold with ${\cal C}(\bar\theta ,K,\varepsilon)=  \left( 2\hat{\cal C}(\bar\theta ,K,2\varepsilon) + 3\hat{\cal C}(0  ,K,\varepsilon) + \hat{\cal C}(\bar\theta ,K,\varepsilon) \right)$.

\subsection{Lamperti-type representation}\label{Lampertitransform}

\begin{proposition}\label{LampOU} 
Let $X_t$ be a L\'evy process with Laplace exponent satisfying  \eqref{psi2} and \eqref{psi3}, and assume that the locally bounded competition mechanisms $g$ is such that $\lim\limits_{z\to 0} \frac{g(z)}{z}$ exists. 
For each $x>0$ there is a unique strong solution $U_t$  to the  SDE \eqref{OHish}. 
Moreover, if we set  
\begin{equation}\label{lampertish}
  V_{t}:= \left\{  
  \begin{aligned}    
   & U_{C_{t}} &\quad & \text{ if } 0\leq t<\eta_{\infty}, \\
   & 0 &\quad & \text { if } \eta_{\infty}<\infty \wedge t\geq \eta_{\infty},
  \end{aligned}   
  \right.
\end{equation}
with $C_t$ the right inverse of  $\eta_{t}:=\inf \left\{s>0:\int_{0}^{s}\frac{\diffd r}{U_r }>t \right\}$, there exists in some enlarged probability space a Brownian motion $B^V$ and an independent Poisson point process $N^V$ in $[0,\infty)^3$ with intensity measure $\diffd t\otimes \diffd \nu \otimes \Pi(\diffd r)$, such that    
\begin{equation}\label{SDECSBPC}
 V_{t}= x - \alpha \int_{0}^{t}V_{s}\diffd s  + \sigma\int_{0}^{t}\sqrt{V_{s}}\diffd B_{s}^{V} 
          + \int_{0}^{t}\int_{0}^{V_{s-}}\int_{0}^{\infty}r\tilde{N}^{V}(\diffd s, \diffd \nu, \diffd r)
          - \int_{0}^{t}G(V_{s})\diffd s , 
\end{equation}
for all $t\geq 0$. 
Last, pathwise uniqueness (and then also  in law) holds for \eqref{SDECSBPC}. 
\end{proposition}

Since $(Z_t(x): t\geq 0)$ in \eqref{flow lb} satisfies \eqref{SDECSBPC} with the Brownian motion
\[
B_t:=\int_0^t  \int_{0}^{Z_{s-}(x)}  \left(Z_{s-}(x)\right)^{-{\frac{1}{2}}}  W(\diffd s,\diffd u)\, , 
\]  
we conclude that  $(Z_t(x): t\geq 0)$ and  $(V_t :t\geq 0)$ are equal in law.  
  
\begin{proof}
 Let $B^X$ and $N^X$ respectively denote a standard Brownian motion and a Poisson random measure on $[0,\infty)^2$ with intensity $\diffd s\otimes \Pi(\diffd r)$ such that $\diffd X_t= -\alpha \diffd t + \sigma \diffd B^X_t  + \int_0^{\infty} r\widetilde{N}^X (\diffd t,\diffd r)$. 
 Standard localization arguments using a sequence $G^R$ of globally Lipschitz functions equal to $G$ on $[0,R]$ (the local Lipschitz character of $G$ following from the assumptions) show the existence of a unique strong solution $U$ until some random explosion time. 
 For each $R,K\geq 0$, set $\tau^R=\inf\{s\geq 0 \, : U^2_s\geq R^2\}$ and  $\theta_K=\inf\{s\geq 0 \, : [U,U]_s  \geq K\}$. 
 Applying It\^o' s formula to the solution of equation \eqref{OHish} with $G_R$ instead of $G$ while keeping in mind the sign of $G$, we get using  Gronwall's lemma that $\E (U_{t\wedge \theta_K\wedge \tau^{R}}^2)  \leq c+ c'K$ for some finite constants $c,c'>0$ depending on $t,x$ and the characteristics of $X$ but not on $G$. 
 Fatou's Lemma then yields $t\wedge \theta_K\leq  \tau^{\infty} =\sup_{R\geq 0}\tau^{R}$   a.s., from where $\tau^{\infty}=\infty$ a.s. and 
 \begin{equation}\label{SDELevy} 
  \diffd U_t =- \alpha \diffd t + \sigma \diffd B^X_t + \int_0^{\infty} r \widetilde{N}^X (\diffd t,\diffd r) -  \frac{G(|U_t| )}{U_t} \diffd t .
 \end{equation} 
 Let us now set  $T_0:= \inf\{t>0:U_t=0\}$ and $T:= \inf\{t>0:V_t=0\}= \inf\{t>0: U_{C_t}=0\}\wedge \eta_{\infty}.$ 
 As $C_t$ is right-continuous, we have that $U_{C_T}=0,$ so that  $C_T = T_0$ since $\eta_{r}= \int_0^{r \wedge T_0} \dfrac{\diffd s}{U_s}= \eta_{\infty}$ for all $r \geq T_0$. 
 In order to show that the  time-changed process $V = (U_{C_{t}}: t \geq 0)$ is solution of \eqref{SDECSBPC}, we follow the arguments of Caballero \textit{et al.} \cite[Proposition 4]{CLU} providing the existence of a standard Brownian motion $B$ such that
 \begin{equation}\label{B}     
  \int_0^t \sqrt{V_s} \diffd B_s = B^{X}(C_t\wedge T_0)
 \end{equation}
 and of a Poisson random measure $N$ with intensity $\diffd s \otimes \diffd \nu \otimes \Pi(\diffd r)$ such that
 \[
 \sum\limits_{\{n:t^X_n<C_t\}} r^X_n = \sum\limits_{\{n:t_n<t\}}\Delta_n = \int_0^t \int_0^{V_{s-}}\int_0^{\infty} r\I_{\{r \geq 1\}}N(\diffd s,\diffd \nu,\diffd r),
 \]
 where $((\Delta_n, t_n) : n \in \NN)$ is a fixed but arbitrary labeling of the jump times and sizes of $V$ and  $((r^X_n, t^X_n) : n \in \NN)$  are the atoms of $N^X$. 
 We have in an $L^2$ sense that
 \begin{equation} \label{N} 
  \lim\limits_{\varepsilon\searrow 0}\left[\sum\limits_{\{n:t_n<t\}} \Delta_n \I_{\{\Delta_n > \varepsilon\}} - \int_0^t V_s \diffd s \int_{\varepsilon}^{\infty} r\Pi(\diffd r)\right] = \lim\limits_{\varepsilon\searrow 0}\left[\sum\limits_{\{n:t^X_n<C_t\}}r^X_n \I_{\{r^X_n> \varepsilon\}} - \int_0^{C_t} \diffd s \int_{\varepsilon}^{\infty} r\Pi(\diffd r)\right],
 \end{equation}
 so that the compensated measures satisfy 
 \begin{equation}\label{tildeN} 
  \int_0^t \int_0^{V_{s^{-}}}\int_0^\infty r \widetilde{N}(\diffd s,\diffd \nu,\diffd r)  = \int_0^{C_t}\int_0^{\infty} r\widetilde{N}^X(\diffd s,\diffd r).
 \end{equation}          
 Inserting the identities \eqref{B}, \eqref{N} and \eqref{tildeN} into equation \eqref{SDELevy}, we deduce  that
 \[
  \diffd U_{C_t}  = -\alpha \diffd C_t  + \sigma \sqrt{V_t} \diffd B_t + \int_0^{V_{t-}}\int_0^{\infty} r \widetilde{N} (\diffd t,\diffd \nu,\diffd r) - cU_{C_t} \diffd C_t.
 \]
 By \eqref{lampertish},  $\diffd C_t =  V_t \diffd t $ and $ \dfrac{G(|U_{C_t}| )}{U_{C_t}} \diffd C_t = G(|V_t|)\diffd t$, thus  $V=(V_t: t\geq 0)$ is a solution of \eqref{SDECSBPC}. 
 Uniqueness for \eqref{SDECSBPC} follows from general results in \cite{FL}. 
\end{proof}

\bibliographystyle{plain}
\addcontentsline{toc}{section}{\refname}
  \bibliography{references}

\end{document}